\documentclass[12pt]{article}
\usepackage[colorlinks=true,citecolor=black,linkcolor=black,urlcolor=blue]{hyperref}
\usepackage[margin=25mm]{geometry}
\usepackage{amsthm}
\usepackage[noabbrev,capitalise]{cleveref}
\usepackage[numbers,sort&compress]{natbib}
\usepackage{varwidth}
\usepackage{cutwin}
\usepackage{wrapfig}
\usepackage{amssymb}
\usepackage{amsmath}
\usepackage{longtable}
\usepackage{wasysym}
\usepackage{morefloats}
\usepackage{url}
\usepackage{caption}
\usepackage{subcaption}
\usepackage{dsfont}
\usepackage[english]{babel}
\usepackage{multicol}
\usepackage{textcomp}
\usepackage{enumerate}
\usepackage{tikz}
\usetikzlibrary{calc}
\newtheorem{thm}{Theorem}
\newtheorem{lem}[thm]{Lemma}
\newtheorem{cor}[thm]{Corollary}
\newtheorem{conj}[thm]{Conjecture}
\newtheorem{clm}[thm]{Claim}
\newtheorem{case}{Case}
\newtheorem{case0}{Case}

\newtheorem{case2}{Case}
\newtheorem{case3}{Case}

\newtheorem{case5}{Case}
\newtheorem{case6}{Case}

\theoremstyle{definition}

\newcommand{\doi}[1]{doi:\,\href{http://dx.doi.org/#1}{#1}}
\newcommand{\arXiv}[1]{arXiv:\,\href{http://arxiv.org/abs/#1}{#1}}
\crefname{case}{Case}{Cases}
\crefname{case0}{Case}{Cases}
\crefname{case1}{Case}{Cases}
\crefname{case2}{Case}{Cases}
\crefname{case3}{Case}{Cases}
\crefname{case4}{Case}{Cases}
\crefname{case5}{Case}{Cases}
\crefname{case6}{Case}{Cases}
\crefname{case7}{Case}{Cases}
\crefname{case8}{Case}{Cases}
\crefname{clm}{Claim}{Claims}
\crefname{lem}{Lemma}{Lemmas}
\crefname{cor}{Corollary}{Corollaries}
\makeatletter

\let\c@table\c@figure
\makeatother 

\renewcommand{\geq}{\geqslant}
\renewcommand{\leq}{\leqslant}

\makeatletter
\renewcommand\section{\@startsection {section}{1}{\z@}%
                                   {-3ex \@plus -1ex \@minus -.2ex}%
                                   {2ex \@plus.2ex}%
                                   {\normalfont\large\bfseries}}
\renewcommand\subsection{\@startsection{subsection}{2}{\z@}%
                                     {-2.5ex\@plus -1ex \@minus -.2ex}%
                                     {1.5ex \@plus .2ex}%
                                     {\normalfont\normalsize\bfseries}}
\renewcommand\subsubsection{\@startsection{subsubsection}{3}{\z@}%
                                     {-2ex\@plus -1ex \@minus -.2ex}%
                                     {1ex \@plus .2ex}%
                                     {\normalfont\normalsize\bfseries}}
 \renewcommand\paragraph{\@startsection{paragraph}{4}{\z@}%
                                    {1.5ex \@plus.5ex \@minus.2ex}%
                                    {-1em}%
                                    {\normalfont\normalsize\bfseries}}
\renewcommand\subparagraph{\@startsection{subparagraph}{5}{\parindent}%
                                       {1.5ex \@plus.5ex \@minus .2ex}%
                                       {-1em}%
                                      {\normalfont\normalsize\bfseries}}
\makeatother
\DeclareMathOperator{\dist}{dist}
\DeclareMathOperator{\exm}{ex_m}
\def \mod#1{{\:({\rm mod}\ #1)}}

  \renewenvironment{thebibliography}[1]{%
    \begin{oldthebibliography}{#1}%
      \setlength{\parskip}{0.25ex}%
      \setlength{\itemsep}{0.25ex}%
  }%
  {%
    \end{oldthebibliography}%
  }

\title{\Large\bf The extremal function for Petersen minors}
\usepackage{textcomp}

\author{Kevin Hendrey\footnote{Research supported by an Australian Postgraduate Award. \texttt{kevin.hendrey@monash.edu}} \qquad 
David R. Wood\footnote{Research supported by the Australian Research Council. \texttt{david.wood@monash.edu}}\\[1ex]
\normalsize School of Mathematical Sciences\\[-0.5ex]
\normalsize Monash University\\[-0.5ex]
\normalsize Melbourne, Australia}

\date{\normalsize August 24, 2015; Revised: \today}

\begin{document}
\maketitle

\begin{abstract}
We prove that every graph with $n$ vertices and at least $5n-8$ edges contains the Petersen graph as a minor, and this bound is best possible. Moreover we characterise all Petersen-minor-free graphs with at least $5n-11$ edges. It follows that every graph containing no Petersen minor is 9-colourable and has vertex arboricity at most 5. These results are also best possible. 
\end{abstract}

\section{Introduction}\label{intro}
A graph $H$ is a \emph{minor} of a graph $G$ if a graph isomorphic to
$H$ can be obtained from $G$ by the following operations: vertex
deletion, edge deletion and edge contraction. The theory of graph
minors, initiated in the seminal work of Robertson and Seymour, is at
the forefront of research in graph theory. A
fundamental question at the intersection of graph minor theory and
extremal graph theory asks, for a given graph $H$, what is the maximum
number $\exm(n,H)$ of edges in an $n$-vertex graph containing no
$H$-minor? The function $\exm(n,H)$ is called the \emph{extremal
function} for $H$-minors.

The extremal function is known for several graphs, including the complete graphs 
$K_4$ and $K_5$ \cite{Wagner,dirac},
$K_6$ and $K_7$ \cite{mader1},
$K_8$ \cite{jorg} and
$K_9$ \cite{song}, the bipartite graphs
$K_{3,3}$ \cite{Hall} and
$K_{2,t}$ \cite{CRS}, and the octahedron
$K_{2,2,2}$ \cite{guoli}, and the complete graph on eight vertices minus an edge $K_8^-$ \cite{song05}.
Tight bounds on the extremal function are known for general complete graphs $K_t$
\cite{vega,Kostochka, Kostochka1,thomason,Thomason01},
unbalanced complete bipartite graphs $K_{s,t}$ \cite{KP08,KP10,KP12,KO05},
disjoint unions of complete graphs \cite{Thomason08},
disjoint unions of cycles \cite{HarveyWood-Cycles,CLNWY15},
general dense graphs \cite{MT-Comb05} and
general sparse graphs \cite{SparseMinor,HarveyWood16}.

\subsection{Petersen Minors}

We study the extremal function when the excluded minor is the Petersen
graph (see \cref{Pet}), denoted by $\mathcal{P}$. Our primary result is the following:

\begin{thm}\label{main}$\exm(n,\mathcal{P})\leq 5n-9$, with equality
if and only if $n\equiv 2 \mod 7$.\end{thm}

For $n\equiv 2 \mod 7$, we in fact completely characterise the extremal graphs (see Theorem~\ref{extremalgraphs} below).

\bigskip\bigskip

\renewcommand\windowpagestuff{%

\begin{center}
\begin{tikzpicture}[line width=1pt,vertex/.style={circle,inner sep=0pt,minimum size=0.2cm}] 

    \pgfmathsetmacro{\n}{5}; 
    \pgfmathsetmacro{\m}{\n-1};

  \node[draw=black,fill=gray] (V_0) at ($(90-0*360/\n:1.5)$) [vertex] {};
 \node[draw=black,fill=gray] (V_1) at ($(90-1*360/\n:1.5)$) [vertex] {};
 \node[draw=black,fill=gray] (V_2) at ($(90-2*360/\n:1.5)$) [vertex] {}; 
\node[draw=black,fill=gray] (V_3) at ($(90-3*360/\n:1.5)$) [vertex] {};
 \node[draw=black,fill=gray] (V_4) at ($(90-4*360/\n:1.5)$) [vertex] {};
  \node[draw=black,fill=gray] (V_{a0}) at ($(90-0*360/\n:0.75)$) [vertex] {};
 \node[draw=black,fill=gray] (V_{a1}) at ($(90-1*360/\n:0.75)$) [vertex] {};
 \node[draw=black,fill=gray] (V_{a2}) at ($(90-2*360/\n:0.75)$) [vertex] {}; 
\node[draw=black,fill=gray] (V_{a3}) at ($(90-3*360/\n:0.75)$) [vertex] {};
 \node[draw=black,fill=gray] (V_{a4}) at ($(90-4*360/\n:0.75)$) [vertex] {};
    \foreach \x in {0,...,4} {
  
\draw[draw=black] (V_\x) -- (V_{a\x});
    }

\draw[draw=black](V_1)--(V_2);
\draw[draw=black](V_2)--(V_3);
\draw[draw=black](V_3)--(V_4);
\draw[draw=black](V_4)--(V_0);
\draw[draw=black](V_0)--(V_1);
\draw[draw=black](V_{a1})--(V_{a3});
\draw[draw=black](V_{a2})--(V_{a4});
\draw[draw=black](V_{a3})--(V_{a0});
\draw[draw=black](V_{a4})--(V_{a1});
\draw[draw=black](V_{a0})--(V_{a2});
\end{tikzpicture}
\captionof{figure}{}\label{Pet}
\end{center}
}
\opencutright
  \begin{cutout}{2}{0.75\linewidth}{0pt}{9}
 The class of $\mathcal{P}$-minor-free graphs is interesting for
several reasons. As an extension of the 4-colour theorem,  Tutte
\cite{tutte} conjectured that every bridgeless graph with no
$\mathcal{P}$-minor has a nowhere zero 4-flow. Edwards, Robertson, Sanders, Seymour and Thomas
\cite{robseytom2,robseytom3, robseytom4,sanseytom,edsanseytom} have
announced a proof that every bridgeless cubic $\mathcal{P}$-minor-free graph is edge 3-colourable, which is equivalent to Tutte's conjecture in the cubic case. Alspach, Goddyn and Zhang
\cite{alspach} showed that a graph has the circuit cover property if
and only if it has no $\mathcal{P}$-minor. It is recognised that
determining the structure of $\mathcal{P}$-minor-free graphs is a key
open problem in graph minor theory (see  \cite{dinglewch, maharry}
for example). \cref{main} is a step in this direction.
  \end{cutout}

\subsection{Extremal Graphs}
We now present the lower bound in \cref{main}, 
and describe the class of extremal graphs. For a graph $H$ and non-negative integer $t$,
an {\it $(H,t)$-cockade} is defined as follows: $H$ itself is an
$(H,t)$-cockade, and any other graph $G$ is an $(H,t)$-cockade if
there are $(H,t)$-cockades $G_1$ and $G_2$ distinct from $G$ such that $G_1\cup G_2=G$ and $G_1\cap G_2 \cong K_t$. It is
well known that for every $(t+1)$-connected graph $H$ and every non-negative integer $s<|V(H)|$,
every $(K_s,t)$-cockade is $H$-minor-free (see Appendix~\ref{app} for a proof). Since $\mathcal{P}$ is 3-connected and $|V(\mathcal{P})|=10$, every $(K_9,2)$-cockade is $\mathcal{P}$-minor-free. Every $n$-vertex $(K_9,2)$-cockade has $5n-9$ edges. For $n\equiv 2 \mod 7$ there is at least one $n$-vertex $(K_9,2)$-cockade, hence $\exm(n,\mathcal{P})\geq 5n-9$ for $n\equiv 2 \mod 7$.

\cref{main} is implied by the following stronger result, which also shows that $(K_9,2)$-cockades are the unique extremal examples of $\mathcal{P}$-minor-free graphs. Indeed, this theorem characterises $\mathcal{P}$-minor-free graphs that are within two edges of extremal.
\begin{thm}\label{extremalgraphs}Every graph with $n\geq 3$ vertices and $m\geq 5n-11$ edges contains a Petersen minor or is a $(K_9,2)$-cockade minus at most two edges.\end{thm}


Since $(K_9,2)$-cockades have connectivity 2, it is interesting to ask
for the maximum number of edges in more highly connected $\mathcal{P}$-minor-free graphs. 
First note that \cref{extremalgraphs} implies that 3-connected $\mathcal{P}$-minor-free graphs, 
with the exception of $K_9$, have at most $5n-12$ edges. To see that this is tight, consider the 
class $\mathcal{C}$ of all graphs $G$ such that there is some subset $S$ of the vertices of $G$ such that $|S|\leq 3$ and each component of $G-S$ contains at most five vertices. Then $\mathcal{C}$ is minor-closed, and it is quick to check that $\mathcal{P}$ is not in $\mathcal{C}$. If $G\in \mathcal{C}$ is such that $|S|=3$, every vertex in $S$ is dominant, and every component of $G-S$ is a copy of $K_5$, then $G$ has $5n-12$ edges and is 3-connected, and is $\mathcal{P}$-minor-free.

We now show that there are 5-connected $\mathcal{P}$-minor-free graphs with almost as many edges
as $(K_9,2)$-cockades. Consider the class $\mathcal{C'}$ of all graphs $G$ with a vertex cover of size at most 5. $C'$ is minor-closed, and $\mathcal{P}$ is not in $\mathcal{C'}$. Let $G:=K_5 + \overline{K_{n-5}}$ for $n\geq 6$. Then $G$ is 5-connected with $|E(G)|=5n-15$, and $G$ is in $\mathcal{C'}$ and thus is $\mathcal{P}$-minor-free.


Now consider 6-connected $\mathcal{P}$-minor-free graphs. A graph $G$
is \emph{apex} if $G-v$ is \emph{planar} for some vertex $v$.
Since $K_{3,3}$ is a minor of $\mathcal{P}-v$ for each vertex $v$, the
Petersen graph is not apex and every apex graph is
$\mathcal{P}$-minor-free. A graph $G$ obtained from a 5-connected
planar triangulation by adding one dominant vertex is 6-connected,
$\mathcal{P}$-minor-free, and has $4n-10$ edges. We know of no
infinite families of 6-connected $\mathcal{P}$-minor-free graphs with
more edges. We also know of no infinite families of 7-connected
$\mathcal{P}$-minor-free graphs. Indeed, it is possible that every
sufficiently large 7-connected graph contains a $\mathcal{P}$-minor.
The following conjecture  is even possible.

\begin{conj}
Every sufficiently large 6-connected $\mathcal{P}$-minor-free graph is apex.
\end{conj}

This is reminiscent of J{\o}rgensen's conjecture \cite{jorg}, which
asserts that every  6-connected $K_6$-minor-free graph is apex.
J{\o}rgensen's conjecture has recently been proved for sufficiently
large graphs \cite{kawnortomwol,kawnortomwol1}. In this respect, $K_6$ and
$\mathcal{P}$ possibly behave similarly. Indeed, they are both members
of the so-called Petersen family \cite {Sachs1983,robseytom95,lovshri}. Note however, that the extremal functions of
$K_6$ and $\mathcal{P}$ are different, since $\exm(n,K_6)=4n-10$
\cite{mader1}.

\subsection{Graph Colouring}
Graph colouring provides further motivation for studying extremal
functions for graph minors. A graph is \emph{$k$-colourable} if each
vertex can be assigned one of $k$ colours such that adjacent vertices
get distinct colours. The \emph{chromatic number} of a graph $G$ is
the minimum integer $k$ such that $G$ is $k$-colourable. In 1943,
Hadwiger~\cite{Hadwiger43} conjectured that every $K_t$-minor-free
graph is $(t-1)$-colourable. This is widely regarded as one of the
most significant open problems in graph theory; see \cite{SeymourHC}
for a recent survey, and see \cite{RS16,albgon} for recent results. Extremal functions provide a natural approach for
colouring graphs excluding a given minor, as summarised in the
following folklore result (see Appendix~\ref{app} for a proof).

\begin{lem}\label{gencolour}
Let $H$ be a graph such that $\exm(n,H)<cn$ for some positive integer $c$. Then every $H$-minor-free graph is $2c$-colourable, and if
$|V(H)|\leq 2c$ then every $H$-minor-free graph is
$(2c-1)$-colourable.
\end{lem}

\cref{main} and \cref{gencolour} with $c=5$ imply the
following Hadwiger-type theorem for $\mathcal{P}$-minors, which is best possible for $\mathcal{P}$-minor-free graphs with $K_9$ subgraphs, for example $(K_9,2)$-cockades.

\begin{thm}\label{colour}
Every $\mathcal{P}$-minor-free graph is 9-colourable.
\end{thm}

For a given graph $G$, a graph colouring can be thought of as a partition of $V(G)$ such that each part induces an edgeless subgraph, equivalently a subgraph with no $K_2$-minor. One way of generalising this is to instead ask for a partition of $V(G)$ such that each part induces a $K_t$-minor-free subgraph for some larger value of $t$. The minimum integer $k$ such that there exist a partition of $V(G)$ into $k$ sets such that each set induces a $K_3$-minor-free subgraph (equivalently a forest), is called the {\it vertex arboricity} of $G$. A graph is {\it $d$-degenerate} if every subgraph has minimum degree at most $d$. Chartrand and Kronk \cite{charkronk} proved that every $d$-degenerate graph has vertex arboricity at most $\lceil \frac{d+1}{2} \rceil$.
%
By \cref{main} every $\mathcal{P}$-minor-free graph is 9-degenerate. Hence, we have the following result, which again is best possible for $\mathcal{P}$-minor-free graphs with $K_9$ subgraphs.

 \begin{thm}Every $\mathcal{P}$-minor-free graph has vertex arboricity at most 5.
 \end{thm}
 
 Other classes of graphs for which the maximum vertex arboricity is known include planar graphs \cite{charkronk}, locally planar graphs \cite{skrek02}, triangle-free locally planar graphs \cite{skrek02}, for each $k\in \{3,4,5,6,7\}$ the class of planar graphs with no $k$-cycles \cite{raspwang,hsww}, planar graphs of diameter 2 \cite{yangyuan}, $K_5$-minor-free graphs of diameter 2 \cite{huawanyua},  and $K_{4,4}$-minor-free graphs \cite{jorg01}.

\subsection{Notation}

The following notation will be used throughout the paper. Let $G$ be a graph, and let $vw$ be an edge of $G$. The graph $G/vw$ is the graph obtained from $G-\{v,w\}$ by adding a new vertex adjacent to all the neighbours of $v$ except $w$ and all the neighbours of $w$ except $v$. The operation which takes $G$ to $G/e$ is a {\it contraction}. If a graph isomorphic to $H$ can be obtained from $G$ by performing edge deletions, vertex deletions and contractions, then $H$ is a {\it minor} of $G$. A graph $G$ is {\it $H$-minor-free} if $H$ is not a minor of $G$. 

The {\it components} of $G$ are the maximal connected subgraphs of $G$. For $S\subseteq V(G)$, let $G[S]$ be the subgraph of $G$ induced by $S$. If $G[S]$ is a complete graph, $S$ is a {\it clique}. We denote by $G-S$ the graph $G[V(G)\setminus S]$. Similarly, if $S\subseteq E(G)$, let $G-S$ be the graph with vertex set $V(G)$ and edge set $E(G)\setminus S$. For simplicity, we write $G-x$ for $G-\{x\}$. For any subgraph $H$ of $G$, we write $G-H$ for $G-V(H)$. 

For each vertex $v$ in $G$, let $N_G(v):=\{w\in V(G): vw\in E(G)\}$ and $N_G[v]:=\{v\}\cup N_G(v)$. Similarly, for each subgraph $C$ of $G$, let $N_G(C)$ be the set of vertices in $G-C$ that are adjacent in $G$ to some vertex of $C$, and let $N_G[C]:=V(C)\cup N_G(C)$. When there is no ambiguity, we write $N(v)$, $N[v]$, $N(C)$ and  $N[C]$ respectively for $N_G(v)$, $N_G[v]$, $N_G(C)$ and $N_G[C]$. A vertex $v$ is {\it dominant} in $G$ if $N_G[v]=V(G)$, and {\it isolated} if $N_G(v)=\emptyset$. 

We denote by $\delta(G)$ the minimum degree of $G$ and by $\Delta(G)$ the maximum degree of $G$. For $i\in \mathbb{N}$, we denote by $V_i(G)$ the set of vertices in $G$ with degree $i$, and by $V_{\geq i}(G)$ the set of vertices of $G$ of degree at least $i$. 

For a tree $T$ and $v,w\in V(T)$, let $vTw$ be the path in $T$ from $v$ to $w$. A vertex of $T$ is {\it high degree} if it is in $V_{\geq 3}(T)$. For a path $P$ with endpoints $x$ and $y$, int$(P):=xy$ if $E(P)=\{xy\}$ and int$(P):=V(P)\setminus \{x,y\}$ otherwise. 

We denote by $G\dot{\cup} H$ the disjoint union of two graphs $G$ and $H$. A subset $S$ of $V(G)$ is a {\it fragment} if $G[S]$ is connected. Distinct fragments $X$ and $Y$ are {\it adjacent} if some vertex in $X$ is adjacent to some vertex in $Y$.

\section{Outline of Proof}\label{sketch}

 We now sketch the proof of \cref{extremalgraphs}. Assume to the contrary that there is some counterexample to \cref{extremalgraphs}, and select a minor-minimal counterexample $G$. Define $\mathcal{L}$ to be the set of vertices $v$ of $G$ such that $\deg(v)\leq 9$ and there is no vertex $u$ with $N[u]\subsetneq N[v]$. For a vertex $v\in V(G)$, a subgraph $H\subseteq G$ is {\it $v$-suitable} if it is a component of $G-N[v]$ that contains some vertex of $\mathcal{L}$.

\cref{basic} shows some elementary results that are used throughout the other sections. In particular, it shows that $\delta(G)\in\{6,7,8,9\}$, and hence that $\mathcal{L}\neq \emptyset$. \cref{7sec,8sec} respectively show that that no vertex of $G$ has degree 7 and that no vertex of $G$ has degree 8. \cref{6sec,9sec} show that for every $v\in \mathcal{L}$ with degree 6 or 9 respectively there is some $v$-suitable subgraph, and that for each $v\in \mathcal{L}$ with degree 6 or 9 and every $v$-suitable subgraph $C$ of $G$ there is some $v$-suitable subgraph $C'$ of $G$ such that $N(C')\setminus N(C)\neq \emptyset$. 

Pick $u\in \mathcal{L}$ and a $u$-suitable subgraph $H$ of $G$ such that $|V(H)|$ is minimised. By the definition of $u$-suitable, there is some $v\in \mathcal{L}\cap V(H)$. Let $C$ be a $v$-suitable subgraph of $G$ containing $u$, and let $C'$ be a $v$-suitable subgraph of $G$ such that $N(C')\setminus N(C)\neq \emptyset$. \cref{endsec} shows that $C'$ selected in this way is a proper subgraph of $H$, contradicting our choice of $H$. 

The basic idea of our proof is similar to proofs used for example in \cite{song} and \cite{albgon}, with the major points of difference conceptually being the use of skeletons, defined in \cref{basic}, to rule out certain configurations, and the proof in \cref{basic} that the minimal counterexample is 4-connected.

\section{Basic Results}\label{basic}

To prove \cref{extremalgraphs}, suppose for contradiction that $G$ is a minor-minimal counterexample to \cref{extremalgraphs}. That is, $G$ is a graph with the following properties:
\begin{enumerate}[(i)]
\item $|V(G)|\geq 3$,
\item $|E(G)|\geq 5|V(G)|-11$,
\item $G$ is not a spanning subgraph of a $(K_9,2)$-cockade,
\item $\mathcal{P}$ is not a minor of $G$,
\item Every proper minor $H$ of $G$ with at least three vertices satisfies $|E(H)|\leq 5|V(H)|-12$ or is a spanning subgraph of a $(K_9,2)$-cockade.
\end{enumerate}
If $H$ is a $(K_9, 2)$-cockade or $K_2$, then $|E(H)|=5|V(H)|-9$. Hence, (v) immediately implies:
\begin{enumerate}
\item[(vi)] Every proper minor $H$ of $G$ with at least two vertices satisfies $|E(H)|\leq 5|V(H)|-9$.
\end{enumerate}


\begin{lem}\label{minvertices} $G$ has at least 10 vertices.\end{lem}
\begin{proof}
Since $5n-11>{n \choose 2}$ for $n\in \{2,3,\dots ,8\}$, every graph satisfying (i) and (ii) has at least 9 vertices. Every 9-vertex graph is a spanning subgraph of a $(K_9,2)$-cockade.
\end{proof}

A {\it separation} of a graph $H$ is a pair $(A,B)$ of subsets of $V(H)$ such that both $A\setminus B$ and $B\setminus A$ are non-empty and $H=H[A]\cup H[B]$. The {\it order} of a separation $(A,B)$ is $|A\cap B|$. A {\it $k$-separation} is a separation of order $k$. A {\it $(\leq k)$-separation} is a separation of order at most $k$. A graph is {\it $k$-connected} if it has at least $k+1$ vertices and no separation of order less than $k$.

Let $x$, $y$ and $z$ be distinct vertices of a graph $H$. A $K_3$-minor \emph{rooted at} $\{x,y,z\}$ is a set of three pairwise-disjoint, pairwise-adjacent fragments $\{X,Y,Z\}$ of $H$ such that $x\in X$, $y\in Y$, $z\in Z$. The following lemma is well known and has been proved, for example, by Wood and Linusson \cite{woodlin}.

\begin{lem}\label{rootedK3lem}Let $x$, $y$ and $z$ be distinct vertices of a graph $H$. There is a $K_3$-minor of $H$ rooted at $\{x,y,z\}$ if and only if there is no vertex $v\in V(H)$ for which the vertices in $\{x,y,z\}\setminus \{v\}$ are in distinct components of $H-v$.\end{lem}

\begin{lem}\label{connect} $G$ is 4-connected.\end{lem}
\begin{proof}
By \cref{minvertices}, $|V(G)|\geq 10$. Suppose for contradiction that there is a $(\leq 3)$-separation $(A,B)$ of $G$. Note that $A\setminus B$ and $B\setminus A$ are both non-empty by definition. We separate into cases based on $|A\cap B|$ and on whether $|A\setminus B|$ is a singleton. Note that while \cref{onevertexcase} is redundant, it is useful to know that \cref{onevertexcase} does not hold when proving that \cref{order1or0case,order3case} do not hold.

\begin{case0}\label{onevertexcase}There is a $(\leq 3)$-separation $(A,B)$ of $G$ such that $|A\setminus B|=\{v\}$\textup{:}\end{case0}
 By \cref{minvertices}, $|B|\geq 9$. Now by (vi) we have
$$|E(G)|\leq |E(G[B])|+\deg(v)\leq 5(|V(G)|-1)-9+3=5|V(G)|-11.$$
By (ii), equality holds throughout. In particular $\deg(v)= 3$ and $|E(G[B])|=5|B|-9$ so $G[B]$ is a $(K_9,2)$-cockade  by (v). For every edge $e$ incident to $v$, we have $E(G/e)=E(G[B])$ by (vi). Hence, $|A\cap B|$ is a clique, and is therefore contained in a subgraph $H\cong K_9$ of $G[B]$. Then $\mathcal{P}\subseteq H\cup G[A]\subseteq G$ contradicting (iv).

\begin{case0}\label{order1or0case} There is a $(\leq 1)$-separation $(A,B)$ of $G$\textup{:}\end{case0}
If either $|A\setminus B|=1$ or $|B\setminus A|=1$ then we are in \cref{onevertexcase}. Otherwise, $|A|\geq 2$ and $|B|\geq 2$, so by (v) we have $|E(G[A])|\leq 5|A|-9$, with equality if and only if $G[A]\cong K_2$ or $G[A]$ is a $(K_9,2)$-cockade, and the same for $B$. Now
$$|E(G)|=|E(G[A])|+|E(G[B])|\leq 5(|V(G)|+1)-9-9=5|V(G)|-13,$$
contradicting (ii).

\begin{case0}\label{order2case}There is a 2-separation $(A,B)$ of $G$\textup{:}\end{case0}
If there is a component $C$ of $G-(A\cap B)$ such that $N(C)\neq A\cap B$, then $G$ has a $(\leq 1)$-separation, and we are in \cref{order1or0case}. Otherwise, let $C_B$ be a component of $G-A$ and let $G_A$ be the graph obtained from $G$ by contracting $G[N[C_B]]$ down to a copy of $K_2$ rooted at $A\cap B$ and deleting all other vertices of $B$. Let $G_B$ be defined analogously. If $|E(G_A)|\leq 5|A|-12$, then 
$$|E(G)|\leq |E(G_A)|+|E(G_B)|-1\leq 5(|V(G)|+2)-12-9-1=5|V(G)|-12,$$
contradicting (ii). Hence, $|E(G_A)|\geq 5|A|-11$, and by (v), $G_A$ is a spanning subgraph of a $(K_9,2)$-cockade $H_A$. By symmetry, $G_B$ is a spanning subgraph of a $(K_9,2)$-cockade $H_B$. Then $G$ is a spanning subgraph of the $(K_9,2)$-cockade formed by gluing $H_A$ and $H_B$ together on $A\cap B$, contradicting (iii).

\begin{case0}\label{order3case}There is a $3$-separation $(A,B)$ of $G$\textup{:}\end{case0}
First, suppose that $G[A]$ does not contain a $K_3$ minor rooted at $A\cap B$. Then there exists a vertex $v$ such that the vertices in $A\cap B$ are in distinct components of $G[A]-v$ by \cref{rootedK3lem}. Recall that $|A\setminus B|>1$, so there is a vertex $w\neq v$ in $A\setminus B$. Let $C$ be the component of $G[A]-v$ containing $w$. Then there is a $(\leq 2)$-separation $(A',B')$ of $G$ where $A'\setminus B'=V(C)\setminus (A\cap B)$, so we are in either \cref{order1or0case} or \cref{order2case}. Hence, there is a $K_3$ minor of  $G[A]$ rooted at $A\cap B$, and by the same argument a $K_3$ minor of $G[B]$ rooted at $A\cap B$. Let $G_A$ be obtained from $G$ by contracting $G[B]$ down to a triangle on $A\cap B$, and let $G_B$ be obtained from $G$ by contracting $G[A]$ down to a triangle on $A\cap B$. Suppose $|E(G_A)|\geq 5|A|-11$. Since $G$ satisfies (v), we have that $G_A$ is a spanning subgraph of a $(K_9,2)$-cockade, and so $G_A$ is a $(K_9,2)$-cockade minus at most two edges. Since $A\cap B$ is a clique of $G_A$, there is some set $S$ of nine vertices in $A$, containing $A\cap B$, such that $G_A[S]$ is $K_9$ minus at most two edges. Let $C$ be a component of $G-A$, and note that $N(C)=A\cap B$, or else we are in \cref{order1or0case} or \cref{order2case}. Now it is quick to check that the graph obtained from $G[S\cup V(C)]$ by contracting $C$ to a single vertex contains $\mathcal{P}$ as a subgraph, contradicting (iv). Hence, $|E(G_A)|\leq 5|A|-12$, and by symmetry $|E(G_B)|\leq 5|B|-12$. Now 
$$|E(G)|\leq |E(G_A)|+|E(G_B)|-3\leq 5(|V(G)|+3)-12-12-3=5|V(G)|-12,$$
contradicting (ii).
\end{proof}

\begin{lem}\label{triangles} $\delta(G) \in \{6,7,8,9\}$ and every edge is in at least five triangles.\end{lem}
\begin{proof}
Suppose for contradiction that some edge $vw$ is in $t$ triangles with $t\leq 4$. Now $$|E(G/vw)|\geq |E(G)|-t-1\geq  5|V(G)|-12-t\geq 5|V(G/e)|-11.$$
Since $G$ satisfies (v), $G/vw$ is a spanning subgraph of some $(K_9,2)$-cockade $H$. By \cref{connect}, $G$ is 4-connected, which implies $G/vw$ is 3-connected,  so $G/vw$ is $K_9$ minus at most two edges. It follows from (ii) that $G$ is a 10-vertex graph with at most six non-edges. It is possible at this point to manually prove that $\mathcal{P}\subseteq G$. 
Rather than detailing this argument, we instead report that a simple random searching algorithm verifies (in six minutes) that $\mathcal{P}$ is a subgraph of every 10-vertex graph with at most six non-edges. Hence, every edge of $G$ is in at least five triangles. By \cref{connect}, $G$ has no isolated vertex,  and $\delta(G)\geq 6$. 

Let $e$ be an edge of $G$. By (vi), $|E(G-e)|\leq 5|V(G)|-9$, so $|E(G)|\leq 5|V(G)|-8$, and hence $\delta(G)\leq 9$.
\end{proof}

Recall that $\mathcal{L}$ is the set of vertices $v$ of $G$ such that $\deg(v)\leq 9$ and there is no vertex $u$ with $N[u]\subsetneq N[v]$.
By \cref{triangles}, every vertex of minimum degree is in $\mathcal{L}$, and $\mathcal{L}\neq\emptyset$. 


The following result is the tool we use for finding $v$-suitable subgraphs. 


\begin{lem}\label{separations} If $(A,B)$ is a separation of $G$ of order $k\leq 6$ such that there is a vertex $v\in B\setminus A$ with $A\cap B\subseteq N(v)$, then there is some vertex $u\in (A\setminus B) \cap \mathcal{L}$.\end{lem}
\begin{proof}
We may assume that every vertex in $A\cap B$ has a neighbour in $A\setminus B$.

Let $u$ be a vertex in $A\setminus B$ with minimum degree in $G$. Suppose for a contradiction that $\deg_G(u)\geq 10$. It follows that every vertex in $A\setminus B$ has degree at least 10 in $G[A]$. Hence, $G[A]$ has at most six vertices of degree less than 10, so $G[A]$ is not a spanning subgraph of a $(K_9,2)$-cockade. Now $|A|\geq |N[u]|\geq 11$, so by (v), 
\begin{equation}\label{edgedeficiteq}\sum_{w\in A\cap B}\deg_{G[A]}(w)=2|E(G[A])|-\sum_{w\in A\setminus B}\deg_{G[A]}(w)\leq  2(5|A|-12)-10|A\setminus B|=10k-24.\end{equation}
Let $X$ be the set of edges of $G$ with one endpoint in $A\cap B$ and the other endpoint in $A\setminus B$. It follows from \cref{connect} that there are a pair of disjoint edges $e_1$ and $e_2$ in $X$, since deleting the endpoints of an edge $e_1\in X$ from $G$ does not leave a disconnected graph and $|A\setminus B|\geq |N[u]|-k\geq 5$. By \cref{triangles}, $e_1$ is in at least five triangles. Each of these triangles contains some edge in $X\setminus \{e_1,e_2\}$, so $|X|\geq 7$. By (\ref{edgedeficiteq}),
$$\delta(G[A\cap B])\leq \frac{1}{k}\sum_{w\in A\cap B}\deg_{G[A\cap B]}(w)=\frac{1}{k}\left(\left(\sum_{w\in A\cap B}\deg_{G[A]}(w)\right)-|X|\right)\leq \frac{1}{k}(10k-31).$$
Since $k\leq 6$, some vertex $x\in A\cap B$ has degree at most 4 in $G[A\cap B]$. Let $G':=G[A\cup \{v\}]/vx$. Then $|E(G')|\geq |E(G[A])|+(k-5)$. Recall that every vertex in $A\setminus B$ has degree at least 10 in $G[A]$. Further, every vertex in $A\cap B$ is incident with some edge in $X$, and hence has at least six neighbours in $A$ by \cref{triangles}. Hence $|E(G')|\geq \frac{1}{2}(10|A\setminus B|+6k)+(k-5)\geq \frac{1}{2}(10|A|-4k)+k-5\geq 5|A|-11$. Then $G'$ is a spanning subgraph of a $(K_9,2)$-cockade by (v), and so $G[A]$ is a spanning subgraph of a $(K_9,2)$-cockade, a contradiction.

Hence, $\deg_G(u)\leq 9$. Suppose for contradiction that $N[w]\subsetneq N[u]$ for some vertex $w$. Then $w\in N(u)$ and $\deg_G(w)<\deg_G(u)$, so $w\in A\cap B$. But $N[w]\subseteq N[u]$, so $w\notin N(v)$, which contradicts the assumption that $A\cap B\subseteq N(v)$. Therefore $u\in \mathcal{L}$, as required.
\end{proof}

For an induced subgraph $H$ of $G$, a subtree $T$ of $G[N[H]]$ is a {\it skeleton} of $H$ if $V_1(T)=N(H)$.

\begin{lem}\label{path}Let $S$ be a fragment of $G$, let $T$ be a skeleton of $G[S]$, and let $v$ and $w$ be distinct vertices of $T$. If $vw\notin E(T)$ and $T\neq vTw$, then there is a path $P$ of $G[N[S]]-\{v,w\}$ from $vTw$ to $T-vTw$ with no internal vertex in $T$.\end{lem}
\begin{proof}
$G-\{v,w\}$ is connected by \cref{connect}, so there is a path in $G-\{v,w\}$ from  $vTw$ to $T-vTw$. Let $P$ be a vertex-minimal example of such a path with endpoints $x$ in $vTw$ and $y$ in $T-vTw$. 

Suppose to the contrary that there is some internal vertex $z$ of $P$ in $T$. Then either $z$ is in $vTw$ and the subpath of $P$ from $z$ to $y$ contradicts the minimality of $P$, or $z$ is in $T-vTw$ and the subpath of $P$ from $x$ to $z$ contradicts the minimality of $P$.

Suppose to the contrary that there is some vertex $z$ in $P- N[S]$. The subpath $P'$ of $P$ from $x$ to $z$ has one end in $S$ and one end in $G-N[S]$, so there is some internal vertex $z'$ of $P'$ in $N(S)$. But $N(S)\subseteq V(T)$, so $z'$ is an internal vertex of $P$ in $T$, a contradiction.
 \end{proof}
 
\begin{lem}\label{add2lem} If $(A,B)$ is a separation of $G$ such that $N(A\setminus B)=A\cap B$, $|A\setminus B|\geq 2$ and $G[A\setminus B]$ is connected, then there is a skeleton of $G[A\setminus B]$ with at least two high degree vertices.\end{lem}

\begin{proof}
There is at least one subtree of $G[A]$ in which every vertex of $A\cap B$ is a leaf, since we can obtain such a tree by taking a spanning subtree of $G[A\setminus B]$ and adding the vertices in $A\cap B$ and, for each vertex in $A\cap B$, exactly one edge $e\in E(G)$ between that vertex and some vertex of $A\setminus B$. We can therefore select $T$ a subtree of $G[A]$ such that $A\cap B\subseteq V_1(T)$ and such that there is no proper subtree $T'$ of $T$ such that $A\cap B \subseteq V_1(T')$. There is no vertex $v$ in $V_1(T)\setminus B$, since for any such vertex $T-v$ is a proper subtree of $T$ and $A\cap B\subseteq V_1(T-v)$, a contradiction. Hence, $V_1(T)=A\cap B$. If $|V_{\geq 3}(T)|\geq 2$ then we are done, so we may assume there is a unique vertex $w$ in $V_{\geq 3}(T)$.

Suppose that for some $x\in A\cap B$ there is some vertex in $\mathrm{int} (xTw)$. By \cref{path}, there is a path $P$ of $G[A]-\{x,w\}$ from $xTw$ to $T-xTw$ with no internal vertex in $T$. Let $y$ be the endpoint of $P$ in $xTw$ and let $z$ be the other endpoint. Then $T':=(T\cup P)-\mathrm{int} (zTw)$ is a skeleton of $G[A\setminus B]$ that has a vertex of degree exactly 3. Since $|V_1(T')|=|A\cap B|\geq 4$ by \cref{connect}, $T'$ has at least two high degree vertices, (namely $y$ and $w$).

Suppose instead that $V(T)=\{w\}\cup (A\cap B)$. By \cref{connect} $G$ is 4-connected, so $(A\setminus B, B\cup \{w\})$ is not a separation of $G$, so there is some vertex $y$ in $A\setminus (B\cup \{w\})$ adjacent to some vertex $x$ in $A\cap B$. Let $P_1$ be a minimal length path from $y$ to $A\cap B$ in $G-\{x,w\}$ (and hence in $G[A]-\{x,w\}$), and let $z$ be the endpoint of $P_1$ in $A\cap B$. Let $P_1'$ be the path formed by adding the vertex $x$ and the edge $xy$ to $P_1$. Since $G[A\setminus B]$ is connected, we can select a minimal length path $P_2$ of $G[A\setminus B]$ from $P_1$ to $w$. Then $(T\cup P_1'\cup P_2)-\{xw, zw\}$ is a skeleton of $G[A\setminus B]$ that has a degree 3 vertex, and therefore at least two high degree vertices, (namely the endpoints of $P_2$).
\end{proof}

For any graph $H$ a {\it table} of $H$ is an ordered 6-tuple $\mathcal{X}:=(X_1,\dots ,X_6)$ of pairwise disjoint fragments of $H$ such that $X_5$ is  adjacent to $X_1, X_2$ and $X_6$, and $X_6$ adjacent to $X_3$ and $X_4$. For any subset $S$ of $V(H)$, $\mathcal{X}$ is {\it rooted} at $S$ if $|X_i\cap S|=1$ for $i\in\{1,2,3 ,4\}$ and $X_5\cap S=X_6\cap S=\emptyset$.
\begin{lem}\label{add2} If $(A,B)$ is a separation of $G$ such that $N(A\setminus  B)=A\cap B$, $|A\cap B|\geq 4$, $|A\setminus B|\geq 2$ and $G[A\setminus B]$ is connected, then there is a table of $G[A]$ rooted at $A\cap B$.\end{lem}
\begin{proof}
By \cref{add2lem}, there is some skeleton $T$ of $G[A\setminus B]$ such that $|V_{\geq 3}(T)|\geq 2$. Let $w$ and $x$ be distinct vertices in $V_{\geq3}(T)$. Let $w_1, w_2$ and $w'$ be three neighbours of $w$ in $T$, and let $x', x_3$ and $x_4$ be three neighbours of $x$ in $T$, labelled so that $w'$ and $x'$ are both in $V(xTw)$. For $i\in\{1,2\}$ let $X_i$ be the vertex set of a path from $w_i$ to a leaf of $T$ in the component subtree of $T-w$ that contains $w_i$, and for $i\in \{3,4\}$ let $X_i$ be the vertex set of a path from $x_i$ to a leaf of $T$ in the component subtree of $T-x$ that contains $x_i$. Since $V_1(T)=A\cap B$, $|X_i\cap B|=1$ for  $i\in \{1,2,3,4\}$. Let $X_5:=V(wTx')$ and let $X_6:=\{x\}$. Then $\mathcal{X}:=(X_1,\dots , X_6)$ satisfies our claim.
\end{proof}

\section{Degree 7 Vertices}\label{7sec}

In this section we show that $V_7(G)=\emptyset$.

\begin{clm}\label{nosing}If $v\in V_7(G)$, then there is no isolated vertex in $G-N[v]$.\end{clm}
\begin{proof}
Suppose for contradiction that there is some isolated vertex $u$ in $G-N[v]$. By \cref{triangles}, $|N(u)|\geq 6$. By \cref{minvertices}, there is some component $C$ of $G- N[v]$ not containing $u$. Since $|N(C)|\geq 4$ by \cref{connect} and $|N(u)\cup N(C)|\leq |N(v)|=7$, there is some vertex $v_1$ in $N(u)\cap N(C)$. Let $v_1$, $v_2$ and $v_3$ be distinct vertices in $N(C)$, and let $v_4$ and $v_5$ be distinct vertices in $N(u)\setminus \{v_1, v_2, v_3\}$. Let $v_6$ and $v_7$ be the remaining vertices of $N(v)$. By \cref{triangles}, for $i\in \{1,2,\dots ,7\}$, $N(v_i)\cap N(v)\geq 5$. If some vertex in $\{v_2,v_3\}$, say $v_2$, is not adjacent to some vertex in $\{v_4,v_5\}$, say $v_5$, then $v_2$ and $v_5$ are both adjacent to every other vertex in $N(v)$, and in particular $v_2v_4$ and $v_3v_5$ are edges in $G$. Hence, there are two disjoint edges between $\{v_2,v_3\}$ and $\{v_4,v_5\}$. Without loss of generality, $\{v_2v_4, v_3v_5\}\subseteq E(G)$. We now consider two cases depending on whether $v_6v_7\in E(G)$.

\begin{case2} $v_6v_7\in E(G)$\textup{:}\end{case2}

Since $v_1$ is adjacent to all but at most one of the other neighbours of $v$, either $v_1v_6\in E(G)$ or $v_1v_7\in E(G)$, so without loss of generality $v_1v_6\in E(G)$. Since $v_7$ is adjacent to all but at most one of the other neighbours of $v$, either $\{v_7v_2, v_7v_5\}\subseteq E(G)$ or $\{v_7v_3,v_7v_4\}\subseteq E(G)$, so without loss of generality $\{v_7v_2,v_7v_5\}\subseteq E(G)$. Let $G'$ be obtained from $G$ by contracting $C$ to a single vertex. Then $\mathcal{P}\subseteq G'$ (see \cref{nosing1}a), contradicting (iv).

\begin{case2} $v_6v_7\notin E(G)$\textup{:}\end{case2}
Then $v_6$ and $v_7$ are both adjacent to every other neighbour of $v$. Let $G'$ be obtained from $G$ by contracting $C$ to a single vertex. Then $\mathcal{P}\subseteq G'$ (see \cref{nosing1}b), contradicting (iv).
\end{proof}

\begin{figure}[htb]
\centering
\begin{tikzpicture}[line width=1pt,vertex/.style={circle,inner sep=0pt,minimum size=0.2cm}] 

    \pgfmathsetmacro{\n}{5}; 
    \pgfmathsetmacro{\m}{\n-1};
\node[draw=white, fill=none, label=left: a)] (K) at ($(-2,-1.2)$) [vertex] {};
  \node[draw=black,fill=gray, label=above:$C$] (V_0) at ($(90-0*360/\n:1.5)$) [vertex] {};
 \node[draw=black,fill=gray, label=above:$v_1$] (V_1) at ($(90-1*360/\n:1.5)$) [vertex] {};
 \node[draw=black,fill=gray, label=right:$u$] (V_2) at ($(90-2*360/\n:1.5)$) [vertex] {}; 
\node[draw=black,fill=gray, label=left:$v_4$] (V_3) at ($(90-3*360/\n:1.5)$) [vertex] {};
 \node[draw=black,fill=gray, label=above:$v_2$] (V_4) at ($(90-4*360/\n:1.5)$) [vertex] {};
  \node[draw=black,fill=gray, label={[label distance=-3mm]87:$v_3$}] (V_{a0}) at ($(90-0*360/\n:0.75)$) [vertex] {};
 \node[draw=black,fill=gray, label={[label distance=-1mm]above:$v_6$}] (V_{a1}) at ($(90-1*360/\n:0.75)$) [vertex] {};
 \node[draw=black,fill=gray, label={[label distance=-1.5mm]85:$v_5$}] (V_{a2}) at ($(90-2*360/\n:0.75)$) [vertex] {}; 
\node[draw=black,fill=gray, label={[label distance=-1.5mm]95:$v$}] (V_{a3}) at ($(90-3*360/\n:0.75)$) [vertex] {};
 \node[draw=black,fill=gray, label={[label distance=-1mm]above:$v_7$}] (V_{a4}) at ($(90-4*360/\n:0.75)$) [vertex] {};
    \foreach \x in {0,...,4} {
  
\draw[draw=black] (V_\x) -- (V_{a\x});
    }

\draw[draw=black](V_1)--(V_2);
\draw[draw=black](V_2)--(V_3);
\draw[draw=black](V_3)--(V_4);
\draw[draw=black](V_4)--(V_0);
\draw[draw=black](V_0)--(V_1);
\draw[draw=black](V_{a1})--(V_{a3});
\draw[draw=black](V_{a2})--(V_{a4});
\draw[draw=black](V_{a3})--(V_{a0});
\draw[draw=black](V_{a4})--(V_{a1});
\draw[draw=black](V_{a0})--(V_{a2});
\end{tikzpicture}
\hspace{0mm}
\begin{tikzpicture}[line width=1pt,vertex/.style={circle,inner sep=0pt,minimum size=0.2cm}] 

    \pgfmathsetmacro{\n}{5}; 
    \pgfmathsetmacro{\m}{\n-1};
    \node[draw=white, fill=none, label=left: b)] (K) at ($(-2,-1.2)$) [vertex] {};
  \node[draw=black,fill=gray, label=above:$C$] (V_0) at ($(90-0*360/\n:1.5)$) [vertex] {};
 \node[draw=black,fill=gray, label=above:$v_1$] (V_1) at ($(90-1*360/\n:1.5)$) [vertex] {};
 \node[draw=black,fill=gray, label=right:$u$] (V_2) at ($(90-2*360/\n:1.5)$) [vertex] {}; 
\node[draw=black,fill=gray, label=left:$v_4$] (V_3) at ($(90-3*360/\n:1.5)$) [vertex] {};
 \node[draw=black,fill=gray, label=above:$v_2$] (V_4) at ($(90-4*360/\n:1.5)$) [vertex] {};
  \node[draw=black,fill=gray, label={[label distance=-3mm]87:$v_3$}] (V_{a0}) at ($(90-0*360/\n:0.75)$) [vertex] {};
 \node[draw=black,fill=gray, label={[label distance=-1mm]above:$v$}] (V_{a1}) at ($(90-1*360/\n:0.75)$) [vertex] {};
 \node[draw=black,fill=gray, label={[label distance=-1.5mm]85:$v_5$}] (V_{a2}) at ($(90-2*360/\n:0.75)$) [vertex] {}; 
\node[draw=black,fill=gray, label={[label distance=-1.5mm]95:$v_6$}] (V_{a3}) at ($(90-3*360/\n:0.75)$) [vertex] {};
 \node[draw=black,fill=gray, label={[label distance=-1mm]above:$v_7$}] (V_{a4}) at ($(90-4*360/\n:0.75)$) [vertex] {};
    \foreach \x in {0,...,4} {
  
\draw[draw=black] (V_\x) -- (V_{a\x});
    }

\draw[draw=black](V_1)--(V_2);
\draw[draw=black](V_2)--(V_3);
\draw[draw=black](V_3)--(V_4);
\draw[draw=black](V_4)--(V_0);
\draw[draw=black](V_0)--(V_1);
\draw[draw=black](V_{a1})--(V_{a3});
\draw[draw=black](V_{a2})--(V_{a4});
\draw[draw=black](V_{a3})--(V_{a0});
\draw[draw=black](V_{a4})--(V_{a1});
\draw[draw=black](V_{a0})--(V_{a2});
\end{tikzpicture}
\hspace{0mm}
\begin{tikzpicture}[line width=1pt,vertex/.style={circle,inner sep=0pt,minimum size=0.2cm}] 

    \pgfmathsetmacro{\n}{5}; 
    \pgfmathsetmacro{\m}{\n-1};
\node[draw=white, fill=none, label=left: c)] (K) at ($(-2,-1.2)$) [vertex] {};
  \node[draw=black,fill=gray, label=above:$v$] (V_0) at ($(90-0*360/\n:1.5)$) [vertex] {};
 \node[draw=black,fill=gray, label=above:$X_4$] (V_1) at ($(90-1*360/\n:1.5)$) [vertex] {};
 \node[draw=black,fill=gray, label=right:$X_6$] (V_2) at ($(90-2*360/\n:1.5)$) [vertex] {}; 
\node[draw=black,fill=gray, label=left:$X_5$] (V_3) at ($(90-3*360/\n:1.5)$) [vertex] {};
 \node[draw=black,fill=gray, label=above:$X_1$] (V_4) at ($(90-4*360/\n:1.5)$) [vertex] {};
  \node[draw=black,fill=gray,  label={[label distance=-3mm]87:$v_5$}] (V_{a0}) at ($(90-0*360/\n:0.75)$) [vertex] {};
 \node[draw=black,fill=gray, label={[label distance=-1mm]above:$v_7$}] (V_{a1}) at ($(90-1*360/\n:0.75)$) [vertex] {};
 \node[draw=black,fill=gray,  label={[label distance=-1.5mm]85:$X_3$}] (V_{a2}) at ($(90-2*360/\n:0.75)$) [vertex] {}; 
\node[draw=black,fill=gray, label={[label distance=-1.5mm]95:$X_2$}] (V_{a3}) at ($(90-3*360/\n:0.75)$) [vertex] {};
 \node[draw=black,fill=gray, label={[label distance=-1mm]above:$v_6$}] (V_{a4}) at ($(90-4*360/\n:0.75)$) [vertex] {};
    \foreach \x in {0,...,4} {
  
\draw[draw=black] (V_\x) -- (V_{a\x});
    }

\draw[draw=black](V_1)--(V_2);
\draw[draw=black](V_2)--(V_3);
\draw[draw=black](V_3)--(V_4);
\draw[draw=black](V_4)--(V_0);
\draw[draw=black](V_0)--(V_1);
\draw[draw=black](V_{a1})--(V_{a3});
\draw[draw=black](V_{a2})--(V_{a4});
\draw[draw=black](V_{a3})--(V_{a0});
\draw[draw=black](V_{a4})--(V_{a1});
\draw[draw=black](V_{a0})--(V_{a2});
\end{tikzpicture}
\caption{\label{nosing1}}
\end{figure}

The following is the main result of this section.

\begin{lem}\label{7N}$V_7(G)=\emptyset$.\end{lem}
\begin{proof}
Suppose for contradiction that there is some vertex $v\in V_7(G)$. By \cref{minvertices}, there is a non-empty component $C$ of $G-N[v]$. By \cref{connect}, $|N(C)|\geq 4$ and by \cref{nosing}, $|V(C)|\geq 2$. Hence, by \cref{add2} with $A:=N[C]$ and $B:=V(G-C)$, there is a table $\mathcal{X}:=(X_1,\dots ,X_6)$ of $G[N[C]]$ rooted at $N(C)$.

Let $\{v_1,\dots ,v_7\}:=N(v)$, with $v_i\in X_i$ for $i\in \{1,2,3,4\}$. By \cref{triangles}, $|N(v_i)\cap N(v)|\geq 5$ for $i\in \{1,2,\dots ,7\}$. We consider two cases depending on whether $v_5v_6v_7$ is a triangle of $G$.

\begin{case3} $v_5v_6v_7$ is a triangle of $G$\textup{:}\end{case3}

Let $Q$ be the bipartite graph with bipartition $V:=\{v_1,v_2,v_3,v_4\}$, $W:=\{v_5,v_6,v_7\}$ and $E(Q):=\{xy: xy\notin E(G),x\in V, y\in W\}$. Then $\Delta (Q)\leq 1$, so without loss of generality $E(Q)\subseteq \{v_1v_5,v_2v_6, v_3v_7\}$. Let $G'$ be obtained from $G$ by contracting $G[X_i]$ to a single vertex for each $i\in \{1,2,\dots , 6\}$. Then $\mathcal{P}\subseteq G'$ (see \cref{nosing1}c), contradicting (iv).

\begin{case3} $v_5v_6v_7$ is not a triangle of $G$\textup{:}\end{case3}
We may assume without loss of generality that $v_5v_6\notin E(G)$. Then $v_5$ and $v_6$ are both adjacent to every other neighbour of $v$. At most one neighbour of $v$ is not adjacent to $v_7$, so $v_7$ has some neighbour in $\{v_1,v_2\}$, say $v_2$, and some neighbour in $\{v_3,v_4\}$, say $v_4$. Let $G'$ be obtained from $G$ by contracting $G[X_i]$ to a single vertex for each $i\in \{1,2,\dots ,7 \}$. Then $\mathcal{P}\subseteq G'$ (see \cref{nosing1}c), contradicting (iv).
\end{proof}
  
  \section{Degree 8 Vertices}\label{8sec}

We now prove that $V_8(G)=\emptyset$. 
 Note that the following lemma applies to any graph, not just $G$. This means we can apply it to minors of $G$, which we do in 
\cref{dist3,2dist3}. 
\begin{clm}\label{deg8lem} If $H$ is a graph that contains a vertex $v$ such that $\deg(v)=8$, $|N(v')\cap N(v)|\geq 5$ for all $v'\in N(v)$, and $C$ is a component of $H\setminus N[v]$ with $|N_H(C)|\geq 3$, then $\mathcal{P}$ is a minor of $H$ unless all of the following conditions hold:
\begin{enumerate}
\item $\overline{K_3}$ is an induced subgraph of $H[N(v)\setminus N(C)]$,
\item $\overline{C_4}$ is an induced subgraph of $H[N(v)]$,
\item $H[N(C)]\cong K_3$.
\end{enumerate}
\end{clm}

\begin{proof}
By assumption, $\delta(H[N(v)])\geq 5$. Let $H'$ be an edge-minimal spanning subgraph of $H[N(v)]$ such that 
 $\delta(H')\geq 5$. 
Every edge $e$ in $H'$ is incident to some vertex of degree 5, since otherwise $\delta(H'-e)\geq 5$, contradicting the minimality of $H'$. Hence, the vertices of degree at most 1 in $\overline{H'}$ form a clique in $\overline{H'}$. Now $\Delta(\overline{H})\leq 2$, since $|V(H')|=\deg(v)=8$ and $\delta(H')\geq 5$. It follows that $\overline{H'}$ is the disjoint union of some number of cycles, all on at least three vertices, and a complete graph on at most two vertices. Let $x$, $y$ and $z$ be three vertices in $N(C)$, and let $1,2, \dots , 5$ be the remaining vertices of $N(v)$. Colour $x,y,$ and $z$ white and colour $1,2,\dots,5$ black. In \cref{longtable} we examine every possible graph $\overline{H'}$, up to colour preserving isomorphism. We use cycle notation to label the graphs, with an ordered pair representing an edge and a singleton representing an isolated vertex. In each case we find $\mathcal{P}$ as a subgraph of the graph $G'$ obtained from $G$ by contracting $C$ to a single vertex, except in the unique case where $K_3$ is an induced subgraph of $\overline{H'}-\{x,y,z\}$, $C_4$ is an induced subgraph of $\overline{H'}$ and $\{x,y,z\}$ is an independent set of vertices in $\overline{H'}$.

\begin{longtable}{|c|c|}
\caption{}\label{longtable}\endfirsthead \caption{}(continued) \endhead
\hline
\includegraphics{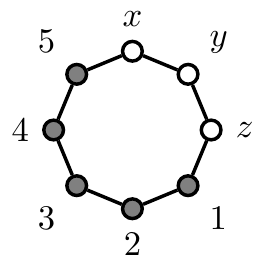}
\includegraphics{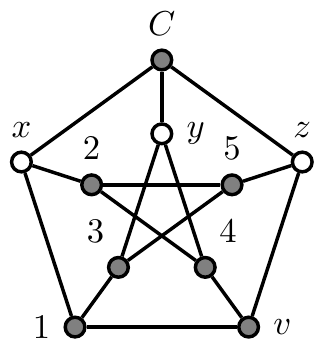}
&
\includegraphics{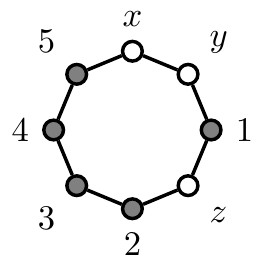}
\includegraphics{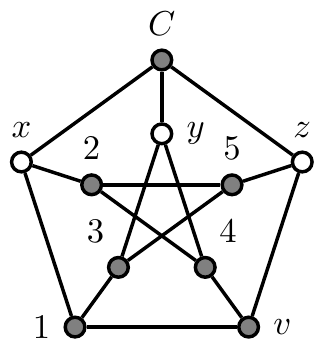}
\\
(wwwbbbbb)&(wwbwbbbb)
\\
\hline
\includegraphics{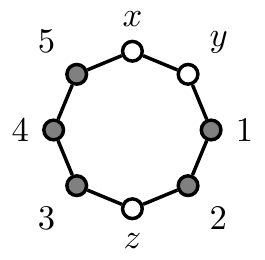}
\includegraphics{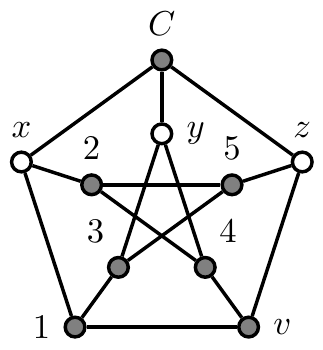}
&
\includegraphics{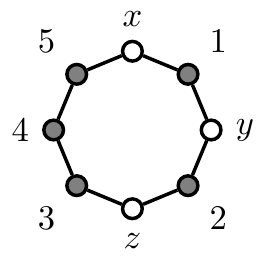}
\includegraphics{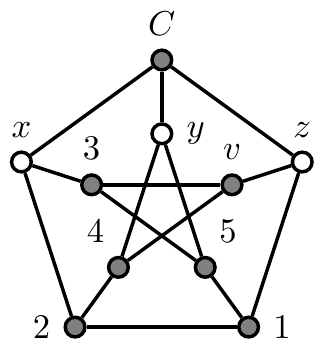}
\\
(wwbbwbbb)&(wbwbwbbb)
\\
\hline
\includegraphics{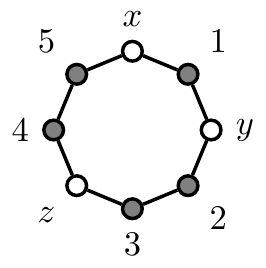}
\includegraphics{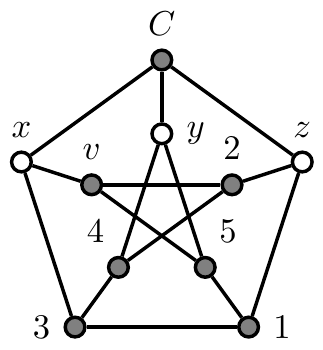}
&
\includegraphics{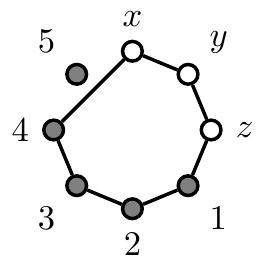}
\includegraphics{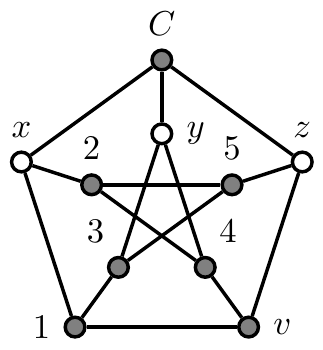}
\\
(wbwbbwbb)&(wwwbbbb)(b)
\\
\hline
\includegraphics{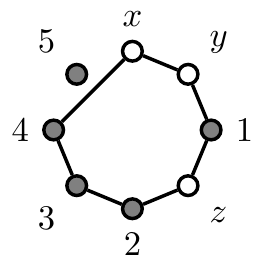}
\includegraphics{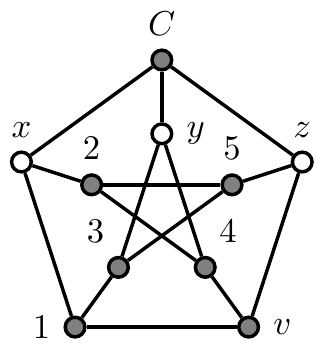}
&
\includegraphics{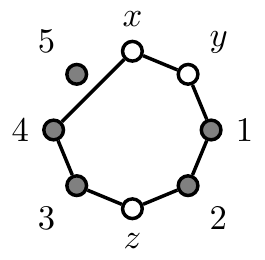}
\includegraphics{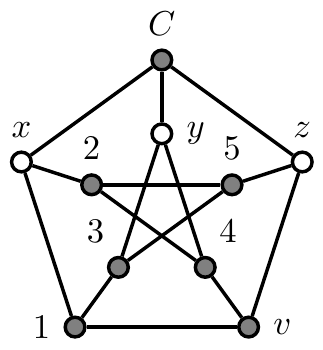}
\\
(wwbwbbb)(b)&(wwbbwbb)(b)
\\
\hline
\includegraphics{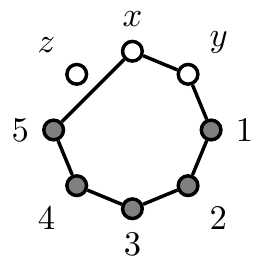}
\includegraphics{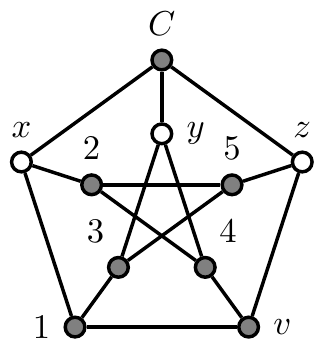}
&
\includegraphics{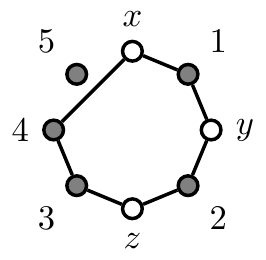}
\includegraphics{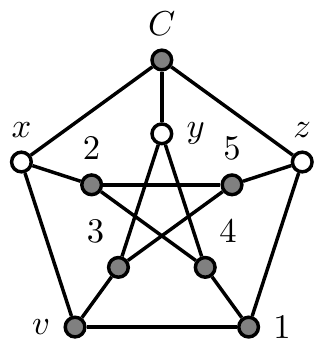}
\\
(wwbbbbb)(w)&(wbwbwbb)(b)
\\
\hline
\includegraphics{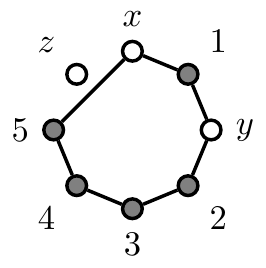}
\includegraphics{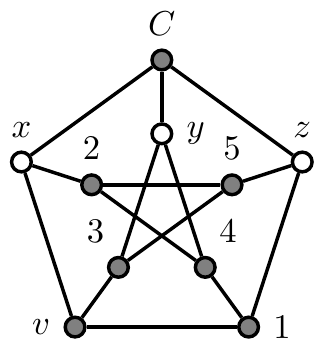}
&
\includegraphics{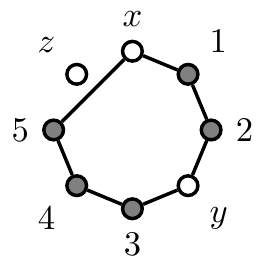}
\includegraphics{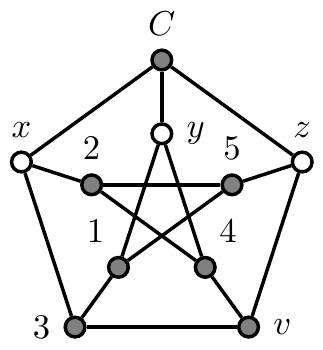}
\\
(wbwbbbb)(w)&(wbbwbbb)(w)
\\
\hline
\includegraphics{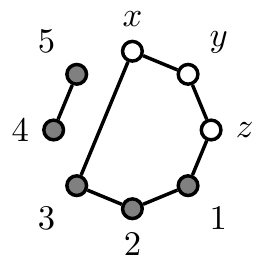}
\includegraphics{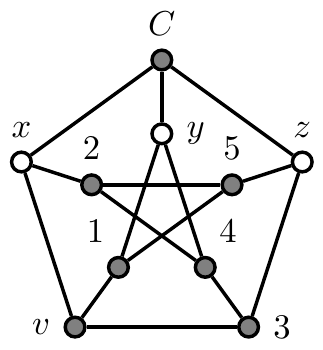}
&
\includegraphics{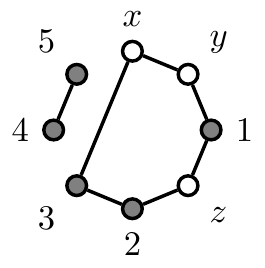}
\includegraphics{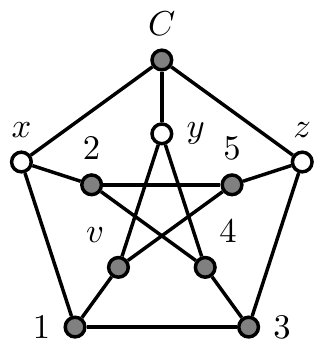}
\\
(wwwbbb)(bb)&(wwbwbb)(bb)
\\
\hline
\includegraphics{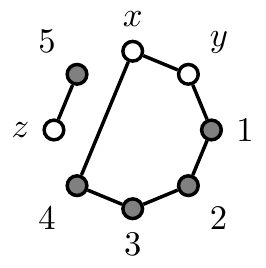}
\includegraphics{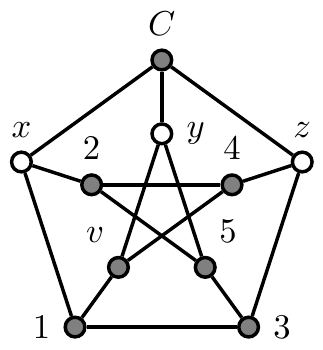}
&
\includegraphics{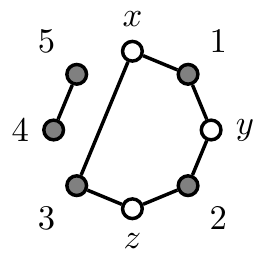}
\includegraphics{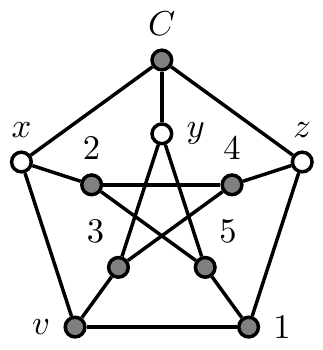}
\\
(wwbbbb)(wb)&(wbwbwb)(bb)
\\
\hline
\includegraphics{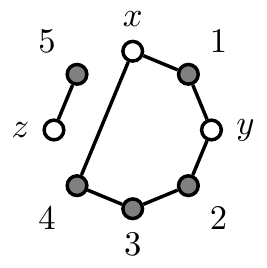}
\includegraphics{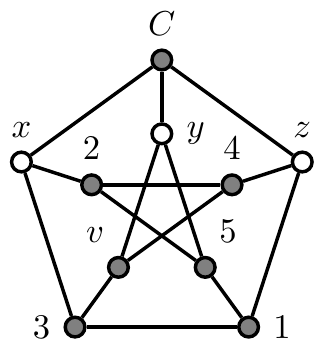}
&
\includegraphics{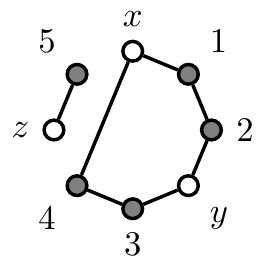}
\includegraphics{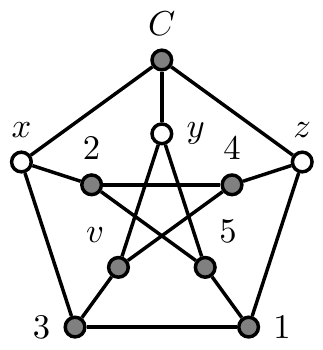}
\\
(wbwbbb)(wb)&(wbbwbb)(wb)
\\
\hline
\includegraphics{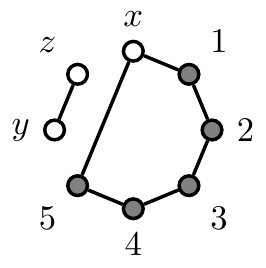}
\includegraphics{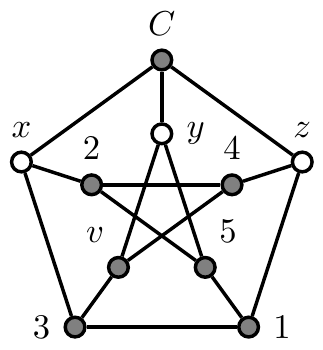}
&
\includegraphics{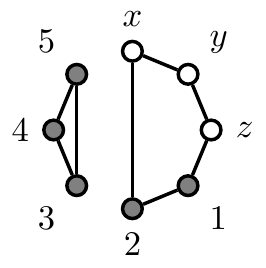}
\includegraphics{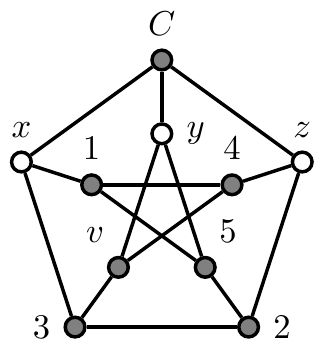}
\\
(wbbbbb)(ww)&(wwwbb)(bbb)
\\
\hline
\includegraphics{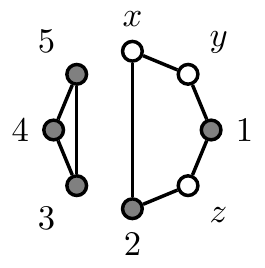}
\includegraphics{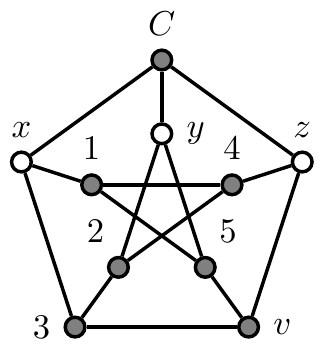}
&
\includegraphics{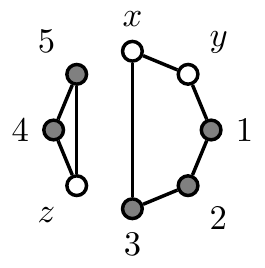}
\includegraphics{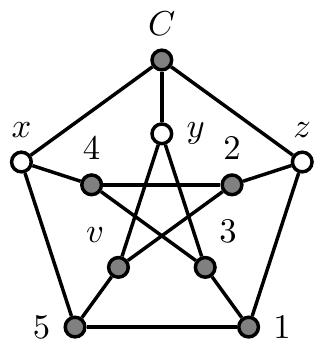}
\\
(wwbwb)(bbb)&(wwbbb)(wbb)
\\
\hline
\includegraphics{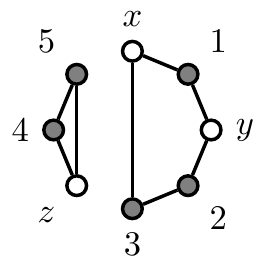}
\includegraphics{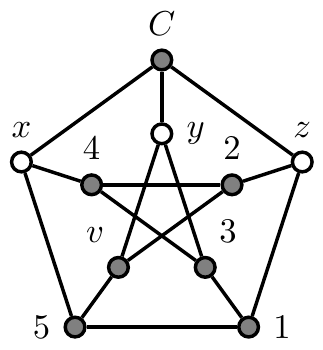}
&
\includegraphics{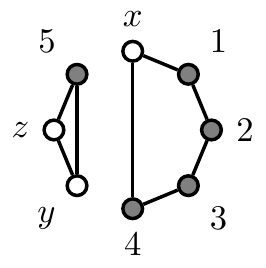}
\includegraphics{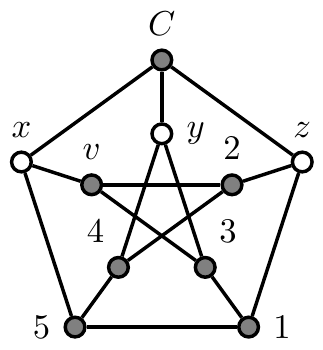}
\\
(wbwbb)(wbb)&(wbbbb)(wwb)
\\
\hline
\includegraphics{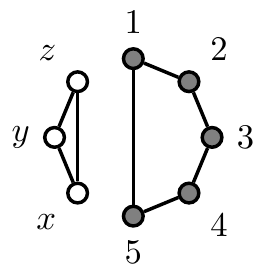}
\includegraphics{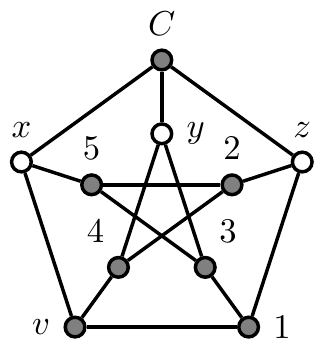}
&
\includegraphics{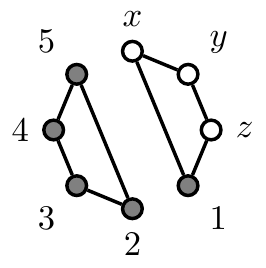}
\includegraphics{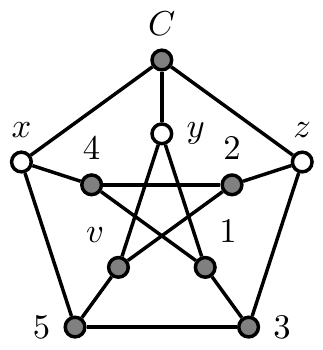}
\\
(bbbbb)(www)&(wwwb)(bbbb)
\\
\hline
\includegraphics{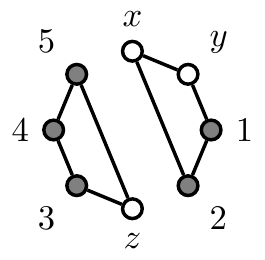}
\includegraphics{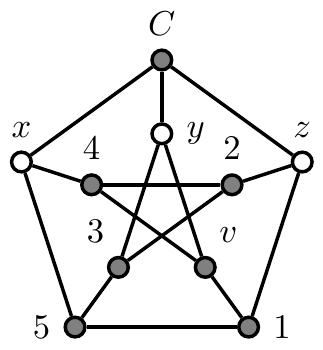}
&
\includegraphics{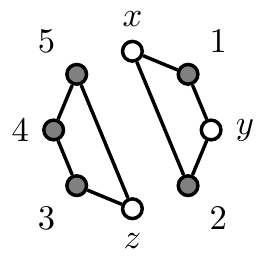}
\includegraphics{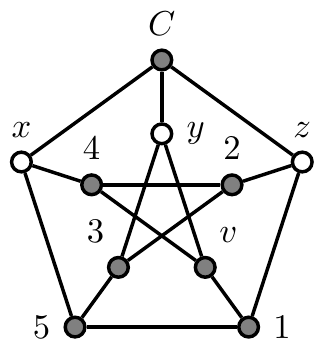}
\\
(wwbb)(wbbb)&(wbwb)(wbbb)
\\
\hline
\includegraphics{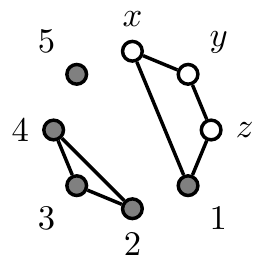}
\includegraphics{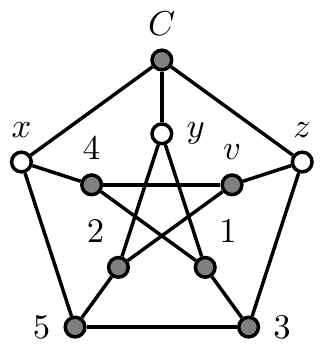}
&
\includegraphics{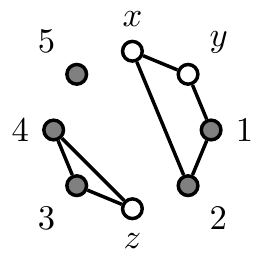}
\includegraphics{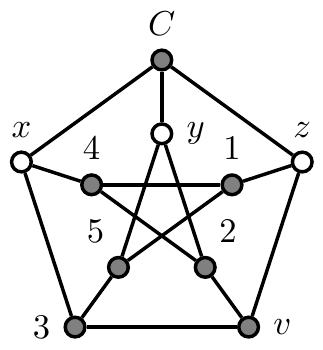}
\\
(wwwb)(bbb)(b)&(wwbb)(wbb)(b)
\\
\hline
\includegraphics{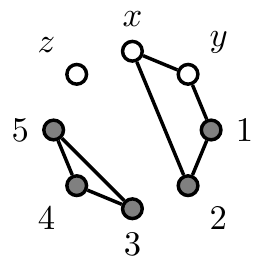}
\includegraphics{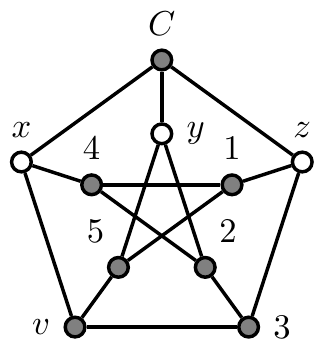}
&
\includegraphics{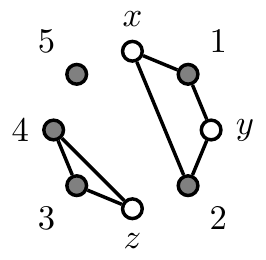}
\includegraphics{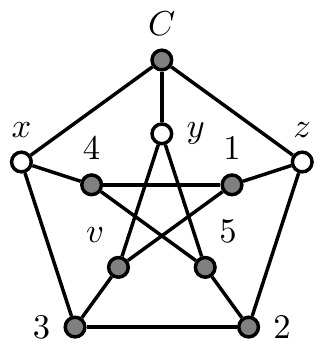}
\\
(wwbb)(bbb)(w)&(wbwb)(wbb)(b)
\\
\hline
\includegraphics{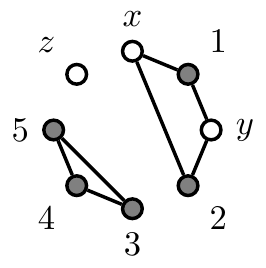}
1., 2. and 3. hold.
&
\includegraphics{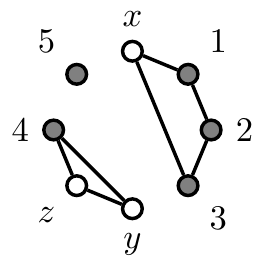}
\includegraphics{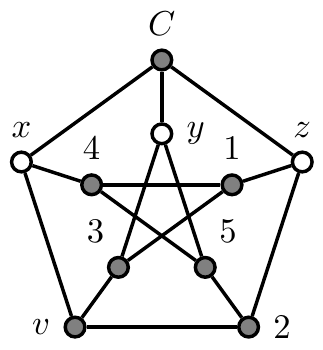}
\\
(wbwb)(bbb)(w)&(wbbb)(wwb)(b)
\\
\hline
\includegraphics{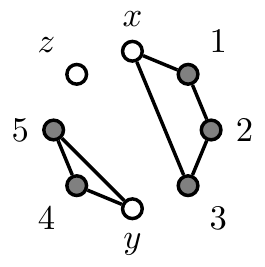}
\includegraphics{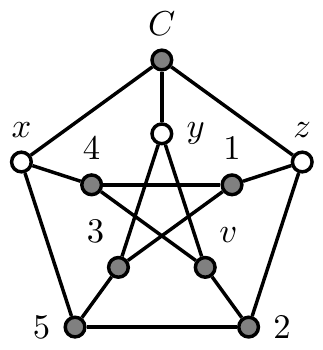}
&
\includegraphics{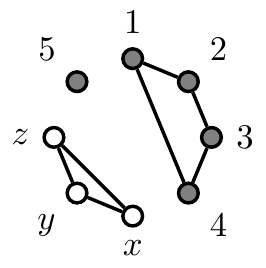}
\includegraphics{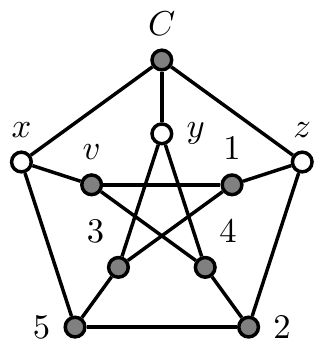}
\\
(wbbb)(wbb)(w)&(bbbb)(www)(b)
\\
\hline
\includegraphics{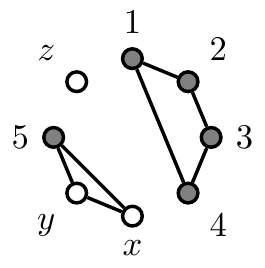}
\includegraphics{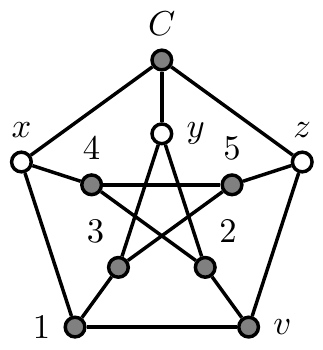}
&
\includegraphics{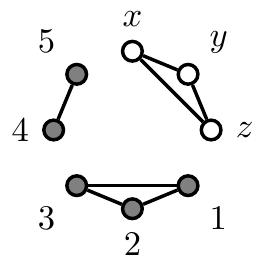}
\includegraphics{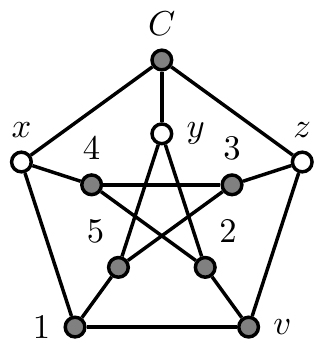}
\\
(bbbb)(wwb)(w)&(www)(bbb)(bb)
\\
\hline
\includegraphics{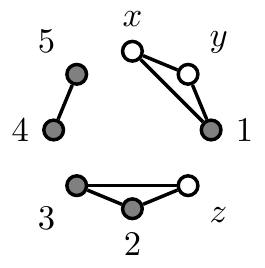}
\includegraphics{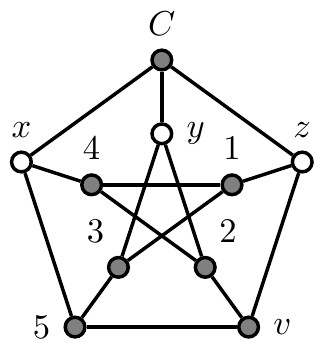}
&
\includegraphics{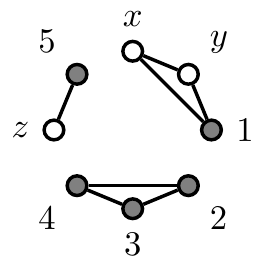}
\includegraphics{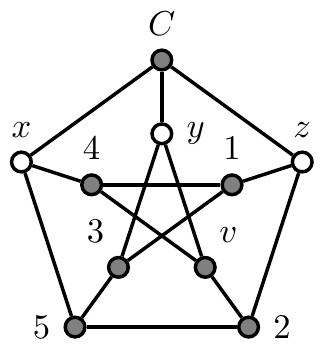}
\\
(wwb)(wbb)(bb)&(wwb)(bbb)(wb)
\\
\hline
\includegraphics{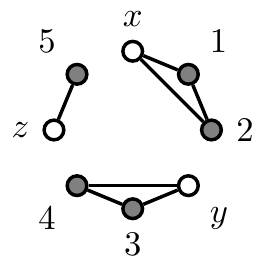}
\includegraphics{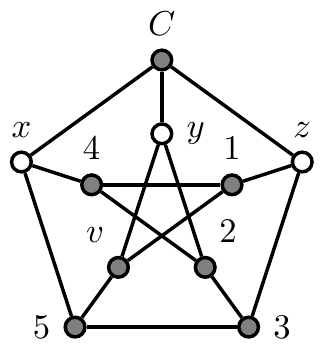}
&
\includegraphics{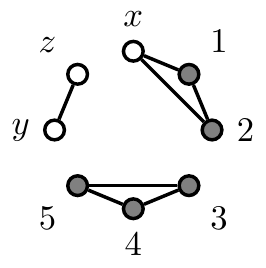}
\includegraphics{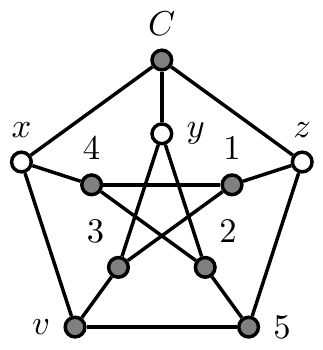}
\\
(wbb)(wbb)(wb)&(wbb)(bbb)(ww)
\\
\hline
\end{longtable}

\newpage
It follows that if $N(C)=\{x,y,z\}$, then the claim holds. Suppose to the contrary that $\mathcal{P}$ is not a minor of $H$ and $|N(C)|\geq 4$. As \cref{longtable} shows, $\overline{H'}$ contains both $K_3$ and $C_4$ as induced subgraphs. Since $\Delta(\overline{H'})\leq 2$, no vertex of $\overline{H'}$ is in more than one cycle, so there is a unique triangle in $\overline{H'}$. For any subset $S\subseteq N(C)$ of size 3, $S$ is a set of independent vertices in $\overline{H'}$, disjoint from the unique triangle of $\overline{H'}$ by the case analysis in \cref{longtable}. Hence, $N(C)$ is an independent set of at least four vertices in $\overline{H'}$, disjoint from the unique triangle of $\overline{H'}$. However, given the structure of $H$, there is no such set, a contradiction.
\end{proof}

The following is the main result of this section.

\begin{lem}\label{deg8} $V_8(G)=\emptyset$.\end{lem}
\begin{proof}
Suppose to the contrary that $v\in V(G)$ has degree 8. By \cref{triangles}, $|N(v')\cap N(v)|\geq 5$ for all $v'\in N(v)$. By \cref{minvertices}, $G-N[v]$ has some non-empty component $C$. By \cref{connect}, $|N(C)|\geq 4$, so $G[N(C)]\not\cong K_3$. Hence, by \cref{deg8lem}, $\mathcal{P}$ is  a minor of $G$, contradicting (iv).
\end{proof}

 \section{Degree 6 Vertices}\label{6sec}

In this section we focus on vertices of degree 6 in $G$. Recall that for a given vertex $v$ of our minimal counterexample $G$, a subgraph $H$ of $G$ is {\it $v$-suitable} if it is a component of $G-N[v]$ that contains some vertex of $\mathcal{L}$. The main result of this section is that if $v\in V_6(G)$, then for any $v$-suitable subgraph $H$ there is a $v$-suitable subgraph $H'$ such that $N(H')\setminus N(H)\neq \emptyset$ (see \cref{6suit}). 

\begin{clm}\label{complete} If $v\in V_6(G)$, then $N[v]$ is a clique.\end{clm}
\begin{proof}
By definition, $v$ is dominant in $G[N[v]]$. Let $w$ be a vertex in $N(v)$. Then $w$ is adjacent to each of the five other vertices in $N(v)$, by \cref{triangles} applied to the edge $vw$.\end{proof}

This result is useful because it means that for an induced subgraph $H$ of $\mathcal{P}$ on seven or fewer vertices, $H\subseteq G[N[v]]$. Throughout this section we show that certain statements about the structure of $G$ imply $\mathcal{P}$ is a minor of $G$, and are therefore false. When illustrating this, the vertices of $N[v]$ will be coloured white, for ease of checking. 


\begin{clm}\label{messy} If $v\in V_6(G)$ and $C$ is a component of $G- N[v]$ with $|N(C)|\geq 5$, then $|V(C)|=1$.\end{clm}
\begin{proof}
Suppose for contradiction that $|V(C)|>1$.  By \cref{add2lem} with $A:=N[C]$ and $B:=V(G-C)$, there is a skeleton $T$ of $C$ with at least two high degree vertices. 
The handshaking lemma implies \begin{align}\label{handshaking}\sum_{i=3}^\infty (i-2)\cdot |V_i(T)|=|V_1(T)|-2.\end{align}
Note that $|V_1(T)|=|N(C)|$ and $|N(C)|\in \{5, 6\}$, so $|V_1(T)|-2\in \{3,4\}$. Hence either $|V_{\geq 3}(T)|\in\{3,4\}$ (Case 2 below), $V_3(T)=\emptyset$ and $|V_4(T)|=2$ (Cases 3 and 4 below), or $|V_3(T)|=1$ and $|V_{\geq 4}(T)|=1$ (Case 5).

\begin{case}$|V(C)|= 2$\textup{:}\end{case}

\vspace{10mm}

\renewcommand\windowpagestuff{%

\begin{center}
\begin{tikzpicture}[line width=1pt,vertex/.style={circle,inner sep=0pt,minimum size=0.2cm}] 

    \pgfmathsetmacro{\n}{5}; 
    \pgfmathsetmacro{\m}{\n-1};

  \node[draw=black,fill=gray, label=above:$w$] (V_0) at ($(90-0*360/\n:1.5)$) [vertex] {};
 \node[draw=black,fill=gray, label=above:$x$] (V_1) at ($(90-1*360/\n:1.5)$) [vertex] {};
 \node[draw=black,fill=none, label=right:$v_1$] (V_2) at ($(90-2*360/\n:1.5)$) [vertex] {}; 
\node[draw=black,fill=gray, label=left:$C'$] (V_3) at ($(90-3*360/\n:1.5)$) [vertex] {};
 \node[draw=black,fill=none, label=above:$v_2$] (V_4) at ($(90-4*360/\n:1.5)$) [vertex] {};
  \node[draw=black,fill=none, label={[label distance=-3mm]87:$v_3$}] (V_{a0}) at ($(90-0*360/\n:0.75)$) [vertex] {};
 \node[draw=black,fill=none,  label={[label distance=-1mm]above:$v_4$}] (V_{a1}) at ($(90-1*360/\n:0.75)$) [vertex] {};
 \node[draw=black,fill=none, label={[label distance=-1.5mm]85:$v''$}] (V_{a2}) at ($(90-2*360/\n:0.75)$) [vertex] {}; 
\node[draw=black,fill=none, label={[label distance=-1.5mm]95:$v'$}] (V_{a3}) at ($(90-3*360/\n:0.75)$) [vertex] {};
 \node[draw=black,fill=none, label={[label distance=-1mm]above:$v$}] (V_{a4}) at ($(90-4*360/\n:0.75)$) [vertex] {};
    \foreach \x in {0,...,4} {
  
\draw[draw=black] (V_\x) -- (V_{a\x});
    }

\draw[draw=black](V_1)--(V_2);
\draw[draw=black](V_2)--(V_3);
\draw[draw=black](V_3)--(V_4);
\draw[draw=black](V_4)--(V_0);
\draw[draw=black](V_0)--(V_1);
\draw[draw=black](V_{a1})--(V_{a3});
\draw[draw=black](V_{a2})--(V_{a4});
\draw[draw=black](V_{a3})--(V_{a0});
\draw[draw=black](V_{a4})--(V_{a1});
\draw[draw=black](V_{a0})--(V_{a2});
\end{tikzpicture}
\captionof{figure}{}\label{messy1}
\end{center}
}
\opencutright
  \begin{cutout}{1}{0.75\linewidth}{0pt}{9}
  Since $C$ is connected, the two vertices $w$ and $x$ of $C$ are adjacent. By \cref{triangles} applied to $wx$, $w$ and $x$ have at least five common neighbours, $v_1,\dots ,v_5$. By \cref{minvertices}, $|V(G-N[v]-C)|\geq 1$, so there is some component $C'\neq C$ of $G- N[v]$. By \cref{connect}, $|N(C')|\geq 4$. Both $N(C')$ and $\{v_1,\dots ,v_5\}$ are subsets of $N(v)$ and $|N(v)|=6$, so $|N(C')\cap \{v_1,\dots ,v_5\}|\geq 3$. Assume without loss of generality that $N(C')\supseteq\{v_1,v_2,v'\}$, where $v'$ is neither $v_3$ nor $v_4$. Let $v''$ be the unique vertex in $N(v)\setminus \{v_1,v_2,v_3,v_4,v'\}$. Let $G'$ be obtained from $G$ by contracting $C'$ to a single vertex. Then $\mathcal{P}\subseteq G'$ by \cref{complete} (see \cref{messy1}), contradicting (iv).
  \end{cutout}

\begin{case}$C$ has a skeleton $T$ with at least three high degree vertices\textup{:}\end{case}
By repeatedly contracting edges of $T\cap C$, we can obtain a minor $T'$ of $T$ such that $T'$ is a tree, $V_1(T')=N(C)$, there are at least three vertices in $V_{\geq 3}(T')$ and $|V_{\geq 3}(T'/e)|\leq 2$ for every edge $e\in E(T'-V_1(T'))$. Contracting an edge of $T'-V_1(T')$ can only reduce $|V_{\geq 3}(T')|$ by 1, and only if both endpoints of the edge are in $|V_{\geq 3}(T')|$. Hence, there are exactly three

\vspace{8mm}

\renewcommand\windowpagestuff{%
\begin{center}
\begin{tikzpicture}[line width=1pt,vertex/.style={circle,inner sep=0pt,minimum size=0.2cm}] 

    \pgfmathsetmacro{\n}{5}; 
    \pgfmathsetmacro{\m}{\n-1};

  \node[draw=black,fill=gray, label=above:$w$] (V_0) at ($(90-0*360/\n:1.5)$) [vertex] {};
 \node[draw=black,fill=gray, label=above:$x$] (V_1) at ($(90-1*360/\n:1.5)$) [vertex] {};
 \node[draw=black,fill=gray, label=right:$y$] (V_2) at ($(90-2*360/\n:1.5)$) [vertex] {}; 
\node[draw=black,fill=none, label=left:$v_5$] (V_3) at ($(90-3*360/\n:1.5)$) [vertex] {};
 \node[draw=black,fill=none, label=above:$v_1$] (V_4) at ($(90-4*360/\n:1.5)$) [vertex] {};
  \node[draw=black,fill=none, label={[label distance=-3mm]87:$v_2$}] (V_{a0}) at ($(90-0*360/\n:0.75)$) [vertex] {};
 \node[draw=black,fill=none, label={[label distance=-1mm]above:$v_3$}] (V_{a1}) at ($(90-1*360/\n:0.75)$) [vertex] {};
 \node[draw=black,fill=none, label={[label distance=-1.5mm]85:$v_4$}] (V_{a2}) at ($(90-2*360/\n:0.75)$) [vertex] {}; 
\node[draw=black,fill=none,  label={[label distance=-1.5mm]95:$v_6$}] (V_{a3}) at ($(90-3*360/\n:0.75)$) [vertex] {};
 \node[draw=black,fill=none, label={[label distance=-1mm]above:$v$}] (V_{a4}) at ($(90-4*360/\n:0.75)$) [vertex] {};
    \foreach \x in {0,...,4} {
  
\draw[draw=black] (V_\x) -- (V_{a\x});
    }

\draw[draw=black](V_1)--(V_2);
\draw[draw=black](V_2)--(V_3);
\draw[draw=black](V_3)--(V_4);
\draw[draw=black](V_4)--(V_0);
\draw[draw=black](V_0)--(V_1);
\draw[draw=black](V_{a1})--(V_{a3});
\draw[draw=black](V_{a2})--(V_{a4});
\draw[draw=black](V_{a3})--(V_{a0});
\draw[draw=black](V_{a4})--(V_{a1});
\draw[draw=black](V_{a0})--(V_{a2});
\end{tikzpicture}
\captionof{figure}{}\label{messy2}
\end{center}
}
\opencutright
  \begin{cutout}{0}{0.75\linewidth}{0pt}{9}
\noindent vertices of $T'-V_1(T)$, and each has degree at least 3 in $T'$. Now $T'-V_1(T)$ is a tree on three vertices, and hence is a path $wxy$. Since $w$, $x$ and $y$ all have degree at least 3 in $T'$, there are distinct vertices $v_1,\dots ,v_5$ such that $w$ is adjacent to $v_1$ and $v_2$ in $T'$, $y$ is adjacent to $v_4$ and $v_5$ in $T'$, and $x$ is adjacent to $v_3$ in $T'$. Let $v_6$ be the remaining vertex in $N(v)$, and recall that $G[N[v]]$ is a complete subgraph of $G$ by \cref{complete}. Let $E$ be the set of edges that were contracted to obtain $T'$, and let $G':=G/E$. Then $\mathcal{P}\subseteq G'$ (see \cref{messy2}), contradicting (iv).
\end{cutout}

\begin{case}There is a skeleton $T$ of $C$ with $|V_4(T)|=2$ and with some $y\in V_2(T)$\textup{:}\end{case}
Let $w$ and $x$ be the vertices in $V_4(T)$. 

First, suppose that $y$ is in  $xTw$. Then by \cref{path}, there is a path $P$ of $G[N[C]]$ from $xTw$ to $T-xTw$ with no internal vertex in $T$. Let $a$ be the endpoint of $P$ in $xTw$ and let $b$ be the other endpoint. Without loss of generality, $w\notin V(xTb)$. Let $R:=(T\cup P)-\mathrm{int}(xTb)$. Then  $R$ is a skeleton of $C$ and $V_{\geq 3}(R)=\{x,w,a\}$, so we are in Case 2.

Suppose instead that $y$ is not in $xTw$. Without loss of generality, $y$ is in the component of $T-\mathrm{int}(xTw)$ containing $x$. Let $z$ be the leaf of $T$ such that $y$ is in $xTz$. By \cref{path}, there is a path $P$ of $G[N(C)]-\{x,z\}$ from $xTz$ to $T-xTz$ with no internal vertex in $T$. Let $a$ be the endpoint of $P$ in $xTz$ and let $b$ be the other endpoint. If $w\notin V(xTb)$ or $w=b$, then let $R:=(T\cup P)-\mathrm{int}(xTb)$. Otherwise, let $R:=(T\cup P)-\mathrm{int}(wTb)$. In either case, $R$ is a skeleton of $C$ and $V_{\geq 3}(R)=\{x,w,a\}$, so we are in Case 2.

\begin{case} There is a skeleton $T$ of $C$ with $|V_4(T)|=2$ and $V_2(T)=\emptyset$\textup{:}\end{case}
Since $T$ is a skeleton of $C$, $|V_1(T)|=|N(C)|\leq 6$. It then follows from (\ref{handshaking}) that $V(T)\setminus V_1(T)= V_4(T)$, and $|V_1(T)|=6$. We may assume that we are not in Case 1, so there is some vertex in $C-V_4(T)$. Since $C$ is connected, there is some vertex $y$ in $C-V_4(T)$ adjacent to some vertex $x$ in $V_4(T)$. Let $w$ be the other vertex of $V_4(T)$. By \cref{connect}, there is a path of $G-x$ from $y$ to $T$. Let $P$ be a vertex-minimal example of such a path, and note that $\mathrm{int}(P)$ is disjoint from $T$. Also, since $N(C)\subseteq V(T)$, every vertex of $P$ is in $N[C]$. Let $P'$ be the path formed from $P$ by adding $x$ and the edge $xy$, and let $b$ be the other endpoint of $P'$. 

Suppose that either $b=w$ or $w\notin V(bTx)$. Let $R:=(T\cup P')-\mathrm{int}(bTx)$. Then $R$ is a skeleton of $C$ with $|V_4(T)|=2$ and $y\in V_2(T)$, so we are in Case 3.

Suppose instead that $w\in \mathrm{int}(bTx)$. Note that $V(T)=\{x,w\}\cup V_1(T)$, and hence $xTw=xw$. Hence, by \cref{triangles}, $x$ and $w$ have at least five common neighbours. If some common neighbour $z$ of $x$ and $w$ is in $C$, then $R:=(T\cup wzx)-\mathrm{int}(xTw)$ is a skeleton of $C$ with $|V_4(R)|=2$ and $z\in V_2(R)$ and we are in Case 3. We may therefore assume that $N(x)\cap N(w)\subseteq N(C)$. Let $v_1, \dots , v_5$ be distinct vertices in $N(x)\cap N(w)$, and let $v_6$ be the remaining vertex of $N(C)$. Let $w_1$, $w_2$ and $w_3$ be distinct neighbours of $w$ in $\{v_1, \dots , v_6\}\setminus \{b\}$, with $w_1=v_6$ if possible. Since $\{v_1,\dots ,v_5\}\subseteq N(x)$ and at least one of $w$ and $x$ is adjacent to $v_6$, $x$ has two neighbours $x_1$ and $x_2$ in $\{v_1, \dots , v_6\}\setminus \{b, w_1,w_2,w_3\}$. Let $V(R):=\{x,w, v_1, \dots , v_6\}\cup V(P)$ and $E(R):=\{ww_1,ww_2,ww_3,xx_1,xx_2,xw\}\cup E(P')$. Then $R$ is a skeleton of $C$ with $V_4(R)=\{x,w\}$ and $y\in V_2(R)$, and we are in Case 3.

\begin{case}There is a skeleton $T$ of $C$ with exactly one vertex $x\in V_3(T)$ and exactly one vertex $w\in V_{\geq 4}(T)$\textup{:}\end{case}
Since $\deg_T(x)=3$ there are distinct leaves $v_1$ and $v_2$ such that $w\notin V(v_1Tv_2)$. Let $v_3,v_4,\dots ,v_k$ be the remaining leaves of $T$, where $k=|N(C)|$. Let $C'$ be the component of $C-w$ containing $x$, and note that $N(C')\subseteq N(C)\cup \{w\}$. Since $G$ is 4-connected by \cref{connect}, there is some vertex in $N(C')\cap(N(C)\setminus \{v_1,v_2\})$, and hence some path $P$ of $G[N[C]\setminus \{w,v_1,v_2\}]$ from $x$ to $N(C)\setminus \{v_1,v_2\}$. Let $P'$ be a subpath of $P$ of shortest possible length while having an endpoint $a$ in the component $T-w$ containing $x$ and an endpoint $b$ in some other component of $T-w$. Note that $P'\subseteq G[N[C]-\{w,v_1,v_2\}]$ and no internal vertex of $P'$ is in $T$.  Let $R:=(T\cup P')-\mathrm{int}(bTw)$, and note that $R$ is a skeleton of $C$. If $a\neq x$, then $V_{\geq 3}(R)=\{a,x,w\}$, and we are in Case 2. If $a=x$ and $w\in V_5(T)$, then $V_{4}(R)=\{x,w\}$, and we are in Case 3 or Case 4. Hence, we may assume $x=a$ and  $w\in V_4(T)$, meaning $|N(C)|=5$. We now consider two subcases, depending on whether $xw\in E(T)$.

\medskip
\textbf{Case 5a.} $wx\notin E(T)$:
\medskip

By \cref{path}, there is a path $Q$ of $G[N[C]]-\{x,w\}$ from $xTw$ to $T-xTw$ with no internal vertex in $T$. Let $c$ be the endpoint of $Q$ in $xTw$, and let $d$ be the other endpoint. 

Suppose first that $Q$ intersects $P'$. Let $Q'$ be the subpath of $Q$ from $c$ to $P'$ that is internally disjoint from $P'$, and let $d'$ be the endpoint of $Q'$ in $P'$. Let $S:=(R\cup Q')-\mathrm{int}(d'Rx)$. Then $S$ is a skeleton of $C$ with $V_{\geq 3}(S)=\{x,c,w\}$, and we are in Case 2.

Suppose instead that $Q$ is disjoint from $P'$. If $x\notin V(dTw)$, then let $S:=(T\cup Q)-\mathrm{int}(dTw)$. Otherwise, let $S:=(R\cup Q)-\mathrm{int}(dRx)$. Then $S$ is a skeleton of $C$ with $V_{\geq 3}(S)=\{x,c,w\}$, and we are in Case 2.

\medskip
\textbf{Case 5b.} $xTw=xw$:
\medskip

By \cref{triangles} applied to the edge $xw$, $|N(x)\cap N(w)|\geq 5$. 

Suppose there is some vertex $y\in (N(x)\cap N(w))\setminus N(C)$. If $y\in(N(x)\cap N(w))\setminus V(T)$, then let $S:=(T\cup xyw)-xw$. Then $S$ is a skeleton of $C$ with exactly one vertex $x\in V_3(S)$ and exactly one vertex $w\in V_{\geq 4}(S)$ and $xw\notin E(S)$, so we are in Case 5a. If $y\in N(x)\cap N(w)\cap V(xTv_i-v_i)$ for some $i\in \{1,2\}$, then let $S$ be the graph obtained from $R$ by adding the edge $wy$ and deleting the edge $wx$. If $y\in N(x)\cap N(w)\cap V(xTv_i-v_i)$ for some $i\in \{3,4,5\}$, then let $S$ be the graph obtained from $T$ by adding the edge $xy$ and deleting the edge $wx$. Then $S$ is a skeleton of $C$ with $V_{\geq 3}(S)=\{x,y,w\}$, and we are in Case 2.

Suppose instead that  $N(x)\cap N(w)\subseteq N(C)$. Since $|N(C)|=5$,  we have $N(x)\cap N(w)= N(C)$. We may assume we are not in Case 1, so by \cref{connect}, there is some vertex $y$ in $C-\{x,w\}$ adjacent to some vertex in $N(C)$. Since $\{x,w\}$ is complete to $N(C)$, assume without loss of generality that $v_5\in N(y)$. Since $C$ is connected, there is a path $Q$ of $C$ from $y$ to $\{w,x\}$. Choose $Q$ to be of shortest possible length, so that $\mathrm{int}(Q)$ is disjoint from $\{x,w\}$, and without loss of generality assume $x$ is an endpoint of $Q$ (since $\{x,w\}$ is complete to $N(C)$). Let $S$ be the skeleton with $V(S):=\{w,v_1,\dots ,v_5\}\cup V(Q)$ and $E(S):=\{wv_1,wv_2,wv_3,wx,xv_4, yv_5\}\cup E(Q)$. By \cref{path}, there is a path $Q'$ of $G[N[C]]-\{x,v_5\}$ from $xSv_5$ to $S-xSv_5$, internally disjoint from $S$. Let $c$ be the endpoint of $Q'$ in $xSv_5$ and let $d$ be the other endpoint. If $d\in \{v_1,v_2,v_3\}$, then let $S':=(S\cup Q')-dw$. Then $S'$ is a skeleton of $C$ with $V_{\geq3}(S')=\{w,x,c\}$, and we are in Case 2. If either $d=w$ and there is some vertex in $\mathrm{int}(Q')$, or $d=v_4$, then let $S':=(S\cup Q')-dx$. Then $S'$ is a skeleton of $C$ with exactly one vertex $c\in V_3(S')$ and exactly one vertex $w\in V_{\geq 4}(S')$, and $cw\notin E(S)$, so we are in Case 5a. If $d=w$ and there is no vertex in $\mathrm{int}(Q')$, then either $c\in N(x)\cap N(w)$, contradicting the assumption that $N(x)\cap N(w)\subseteq N(C)$, or $|V(ySc\cup Q')|<|V(Q)|$, contradicting our choice of $Q$.
\end{proof}

\begin{clm}\label{exactneighbourhood} If $v\in V_6(G)$ and $C$ is a component of $G-N[v]$, then $V(C)\neq \emptyset$ and $|N(C)|=4$.\end{clm}
\begin{proof}

By \cref{minvertices}, $V(G)\setminus N[v]$ is non-empty, so $V(C)\neq \emptyset$. Hence $|N(C)|\geq 4$ by \cref{connect}. Suppose for contradiction that $|N(C)|\geq 5$. Then $|V(C)|=1$ by \cref{messy}. Hence, by \cref{triangles}, $|N(C)|\geq 6$, so $N(C)=N(v)$.

Suppose that there is some component $C'$ of $G-N[v]$ with $|N(C')|=4$. By \cref{triangles}, $|V(C')|\geq 3$. Hence, by \cref{add2} with $A:=N[C']$ and $B:=V(G-C')$, there is a table $\mathcal{X}:=(X_1,\dots ,X_6)$ of $G[N[C']]$ rooted at $N(C')$. For $i\in\{1,2,3 ,4\}$, let $v_i$ be the unique vertex in $X_i\cap N(C')$. Let $v_5$ and $v_6$ be the remaining vertices of $N(v)$. By \cref{complete}, $G[N[v]]\cong K_7$. Let $G'$ be obtained from $G$ by contracting $G[X_i]$ to a single vertex for each $i\in\{1,2,\dots ,6\}$. Then $\mathcal{P}\subseteq G'$ (see \cref{4neighbours}a), contradicting (iv).

Suppose instead that every component $C'$ of $G-N[v]$ satisfies $|N(C')|\geq 5$. Then by \cref{messy} every component of $G-N[v]$ is an isolated vertex and by \cref{triangles} each component $C'$ of $G-N[v]$ satisfies $N(C')=N(v)$. Now by \cref{minvertices} there are at least three distinct components $C$, $C'$ and $C''$ of $G-N[v]$. Hence, by \cref{complete}, $\mathcal{P}\subseteq G$ (see \cref{4neighbours}b), contradicting (iv).  
\end{proof}

\begin{figure}[!htb]
\centering
\begin{tikzpicture}[line width=1pt,vertex/.style={circle,inner sep=0pt,minimum size=0.2cm}] 
    \pgfmathsetmacro{\n}{5}; 
    \pgfmathsetmacro{\m}{\n-1};

 \node[draw=white, fill=none, label=left: b)] (K) at ($(3.5,-1.2)$) [vertex] {};
\node[draw=white, fill=none, label=left: a)] (K) at ($(-2.5,-1.2)$) [vertex] {};
  \node[draw=black,fill=gray, label=above:$X_5$] (V_0) at ($(90-0*360/\n:1.5)$) [vertex] {};
 \node[draw=black,fill=gray, label=above:$X_6$] (V_1) at ($(90-1*360/\n:1.5)$) [vertex] {};
 \node[draw=black,fill=none, label=right:$X_4$] (V_2) at ($(90-2*360/\n:1.5)$) [vertex] {}; 
\node[draw=black,fill=none, label=left:$v$] (V_3) at ($(90-3*360/\n:1.5)$) [vertex] {};
 \node[draw=black,fill=none, label=above:$X_1$] (V_4) at ($(90-4*360/\n:1.5)$) [vertex] {};
  \node[draw=black,fill=none, label={[label distance=-3mm]87:$X_2$}] (V_{a0}) at ($(90-0*360/\n:0.75)$) [vertex] {};
 \node[draw=black,fill=none,  label={[label distance=-1.5mm]above:$X_3$}] (V_{a1}) at ($(90-1*360/\n:0.75)$) [vertex] {};
 \node[draw=black,fill=none, label={[label distance=-1.5mm]85:$u_5$}] (V_{a2}) at ($(90-2*360/\n:0.75)$) [vertex] {}; 
\node[draw=black,fill=none,  label={[label distance=-1.5mm]95:$u_6$}] (V_{a3}) at ($(90-3*360/\n:0.75)$) [vertex] {};
 \node[draw=black,fill=gray,  label={[label distance=-1mm]above:$C$}] (V_{a4}) at ($(90-4*360/\n:0.75)$) [vertex] {};
    \foreach \x in {0,...,4} {
  
\draw[draw=black] (V_\x) -- (V_{a\x});
    }

\draw[draw=black](V_1)--(V_2);
\draw[draw=black](V_2)--(V_3);
\draw[draw=black](V_3)--(V_4);
\draw[draw=black](V_4)--(V_0);
\draw[draw=black](V_0)--(V_1);
\draw[draw=black](V_{a1})--(V_{a3});
\draw[draw=black](V_{a2})--(V_{a4});
\draw[draw=black](V_{a3})--(V_{a0});
\draw[draw=black](V_{a4})--(V_{a1});
\draw[draw=black](V_{a0})--(V_{a2});
\end{tikzpicture}
\begin{tikzpicture}[line width=1pt,vertex/.style={circle,inner sep=0pt,minimum size=0.2cm}] 

    \pgfmathsetmacro{\n}{5}; 
    \pgfmathsetmacro{\m}{\n-1};

  \node[draw=black,fill=gray, label=above:$C$] (V_0) at ($(90-0*360/\n:1.5)$) [vertex] {};
 \node[draw=black,fill=none, label=above:$v_1$] (V_1) at ($(90-1*360/\n:1.5)$) [vertex] {};
 \node[draw=black,fill=gray, label=right:$C''$] (V_2) at ($(90-2*360/\n:1.5)$) [vertex] {}; 
\node[draw=black,fill=none, label=left:$v_3$] (V_3) at ($(90-3*360/\n:1.5)$) [vertex] {};
 \node[draw=black,fill=none, label=above:$v_4$] (V_4) at ($(90-4*360/\n:1.5)$) [vertex] {};
  \node[draw=black,fill=none,  label={[label distance=-3mm]87:$v_5$}] (V_{a0}) at ($(90-0*360/\n:0.75)$) [vertex] {};
 \node[draw=black,fill=none, label={[label distance=-1mm]above:$v_6$}] (V_{a1}) at ($(90-1*360/\n:0.75)$) [vertex] {};
 \node[draw=black,fill=none,  label={[label distance=-1.5mm]85:$v_2$}] (V_{a2}) at ($(90-2*360/\n:0.75)$) [vertex] {}; 
\node[draw=black,fill=none, label={[label distance=-1.5mm]95:$v$}] (V_{a3}) at ($(90-3*360/\n:0.75)$) [vertex] {};
 \node[draw=black,fill=gray, label={[label distance=-1mm]above:$C'$}] (V_{a4}) at ($(90-4*360/\n:0.75)$) [vertex] {};
    \foreach \x in {0,...,4} {
  
\draw[draw=black] (V_\x) -- (V_{a\x});
    }

\draw[draw=black](V_1)--(V_2);
\draw[draw=black](V_2)--(V_3);
\draw[draw=black](V_3)--(V_4);
\draw[draw=black](V_4)--(V_0);
\draw[draw=black](V_0)--(V_1);
\draw[draw=black](V_{a1})--(V_{a3});
\draw[draw=black](V_{a2})--(V_{a4});
\draw[draw=black](V_{a3})--(V_{a0});
\draw[draw=black](V_{a4})--(V_{a1});
\draw[draw=black](V_{a0})--(V_{a2});
\end{tikzpicture}
\caption{\label{4neighbours}}
\end{figure}

\cref{exactneighbourhood} and \cref{separations} immediately imply the following corollary, which we use in the final step of the proof, in \cref{endsec}.

\begin{cor}\label{6ex}For every vertex $v\in V_6(G)$, there is at least one $v$-suitable subgraph.\end{cor}

We now prove the main result of this section.

\begin{lem}\label{6suit} If $v\in V_6(G)$ and $H$ is a $v$-suitable subgraph of $G$, then there is some $v$-suitable subgraph $H'$ of $G$ such that $N(H')\setminus N(H)\neq \emptyset$.\end{lem}

\begin{proof}
By \cref{exactneighbourhood}, $|N(H)|=4$.
Suppose for contradiction that there exist distinct vertices $w,x\in N(v)$ such that $N[x]\subseteq N[v]$ and $N[w]\subseteq N[v]$. Let $G':=G-\{v,w,x\}$. By (ii), 
$$|E(G')|\geq |E(G)|-3-3(4)\geq (5|V(G)|-11)-15=5|V(G')|-11.$$
By (v), $G'$ is a $(K_9,2)$-cockade minus at most two edges. 
Every $(K_9,2)$-cockade has at least nine vertices of degree exactly 8, so $|V_8(G')|\geq 5$. 
 Then some vertex in $V(G')\setminus N[v]$ has degree exactly 8 in $G$, contradicting \cref{deg8}.

Hence there is at most one vertex $w$ in $N(v)$ such that $N[w]\subseteq N[v]$, so there is some vertex $x$ in $N(v)\setminus N(H)$ with some neighbour $y$ in $G-N[v]$. Let $H'$ be the component of $G-N[v]$ that contains $y$. The vertex $x$ is in $N(H')$, so $N(H')\setminus N(H)\neq \emptyset$.  By \cref{exactneighbourhood} and \cref{separations}, $H'$ is $v$-suitable, as required.
\end{proof}

\section{Degree 9 Vertices}\label{9sec}
 
In this section, we focus on vertices in $V_9(G)\cap \mathcal{L}$. For each such vertex $v$, the minimum degree of $G[N(v)]$ is at least 5, by \cref{triangles} applied to each edge incident to $v$. Let $H_v$ be the complement of an edge-minimal spanning subgraph of $G[N(v)]$ with minimum degree 5. 

 
 The main result of this section, \cref{9suit}, states that for each component $C$ of $G-N[v]$, there is some $v$-suitable subgraph $C'$ with a neighbour not in the neighbourhood of $C$. We argue for this claim directly when each component $C'$ of $G-N[v]$ has $|N(C')|=4$. Otherwise, we first look at the case where the maximum distance between two vertices of degree 3 in $H_v$ is at most 2. Then we consider the case where there are two vertices of degree 3 at distance at least 3 in $H_v$. A useful technique is that a graph obtained by contracting some edge in $G[N(v)]$ must violate some condition of \cref{deg8lem}.

\begin{clm}\label{clique} If $v\in V_9(G)\cap \mathcal{L}$, then $\Delta(H_v)=3$ and the vertices of $H_v$ with degree at most 2 form a clique.\end{clm}
\begin{proof}
Since $|V(H_v)|=|N(v)|=9$,  if a vertex $u$ has degree greater than 3 in $H_v$, then $u$ has degree less than 5 in $\overline{H_v}$, a contradiction. If two non-adjacent vertices $x$ and $y$ in $H_v$ both have degree at most 2 in $H_v$, then $\overline{H_v}- xy$ is a spanning subgraph of $G[N(v)]$ with minimum degree at least 5, contradicting the definition of $H_v$. Thus the vertices of degree at most 2 form a clique of size at most 3, so there is indeed a vertex of degree 3 in $H_v$.\end{proof}

The following claim guarantees that $|V(G)|\geq 11$ if we find a vertex $v\in V_9(G)\cap \mathcal{L}$, and hence that the components of $G-N[v]$ are non-empty.

\begin{clm}\label{11vertices}If $v\in V_9(G)\cap \mathcal{L}$, then $V(G-N[v])\neq \emptyset$.\end{clm}
\begin{proof}
By (iv), $\mathcal{P}\not\subseteq G[N[v]]$, so $G[v]\not\cong K_{10}$. Hence, there is some vertex $w\in N(v)$ such that $N[w]\neq N[v]$. By the definition of $\mathcal{L}$, there is some vertex $x\in N[w]\setminus N[v]$ and $x\in V(G-N[v])$.
\end{proof}

A graph is {\it cubic} if every vertex has degree exactly 3.

\begin{clm}\label{dist3}
If $v\in V_9(G)\cap \mathcal{L}$, then there are vertices $x$ and $y$ in $V_3(H_v)$ such that $\dist_{H_v}(x,y)\geq 3$, 
unless either $|N(C)|=4$ for  every component $C$ of $G- N[v]$ or $H_v\cong K_{3,3}\dot{\cup} K_3$.\end{clm}

\begin{proof}
Suppose for contradiction that $\dist_{H_v}(x,y)\leq 2$ whenever $\{x,y\}\subseteq V_3(H_v)$, there is some component $C$ of $G- N[v]$ such that $|N(C)|\neq 4$ and $H_v\ncong K_{3,3}\dot{\cup} K_3$. By \cref{11vertices}, $V(C)\neq \emptyset$, so  by \cref{connect}, $|N(C)|\geq 5$. Let $S:=V_0(H_v)\cup V_1(H_v)\cup V_2(H_v)$. By \cref{clique}, $S$ is a clique, so $|S|\leq 3$. Since $|V(H_v)|=9$, the number of vertices of odd degree in $H_v$ is even and $V(H_v)\setminus S=V_3(H_v)$, we have $S\neq \emptyset$. We consider five cases depending on $S$ and whether there is any triangle in $H_v$.
\begin{case5}$|S|=3$\textup{:}\end{case5}
In this case, $S=V_2(H_v)$ and $H_v[S]\cong K_3$, and there is no edge in $H_v$ from a vertex in $S$ to a vertex not in $S$. Hence, $H_v-S$ is a 6-vertex cubic graph. By assumption, $H_v\ncong K_{3,3}$. There is only one other 6-vertex cubic graph, so $H_v$ is the graph depicted in \cref{dist31}. 
Then $\mathcal{P}\subseteq G[N[v]]$ (see \cref{dist31}b), contradicting (iv).

\begin{figure}[htb]
\centering
\begin{tikzpicture}[line width=1pt,vertex/.style={circle,inner sep=0pt,minimum size=0.2cm}] 

    \pgfmathsetmacro{\n}{1.3}; 
  \node[draw=white, fill=none, label=left: b)] (K) at ($(6.5*\n,0)$) [vertex] {};
\node[draw=white, fill=none, label=left: a)] (K) at ($(-0.5*\n,0)$) [vertex] {};
  \node[draw=black,fill=gray, label=above:$a_1$] (V_0) at ($(0*\n,1*\n)$) [vertex] {};
 \node[draw=black,fill=gray, label=above:$b_1$] (V_1) at ($(2*\n,1*\n)$) [vertex] {};
 \node[draw=black,fill=gray, label=above:$c_1$] (V_2) at ($(3.5*\n,1*\n)$) [vertex] {}; 
\node[draw=black,fill=gray, label=below:$a_2$] (V_3) at ($(0.6*\n,0.5*\n)$) [vertex] {};
 \node[draw=black,fill=gray, label=below:$b_2$] (V_4) at ($(1.4*\n,0.5*\n)$) [vertex] {};
  \node[draw=black,fill=gray, label=below:$a_3$] (V_5) at ($(0*\n,0*\n)$) [vertex] {};
 \node[draw=black,fill=gray, label=below:$b_3$] (V_6) at ($(2*\n,0*\n)$) [vertex] {};
 \node[draw=black,fill=gray, label=below:$c_2$] (V_7) at ($(3*\n,0*\n)$) [vertex] {}; 
\node[draw=black,fill=gray, label=below:$c_3$] (V_8) at ($(4*\n,0*\n)$) [vertex] {};

\draw[draw=black](V_0)--(V_1);
\draw[draw=black](V_2)--(V_8);
\draw[draw=black](V_2)--(V_7);
\draw[draw=black](V_7)--(V_8);
\draw[draw=black](V_0)--(V_3);
\draw[draw=black](V_0)--(V_5);
\draw[draw=black](V_3)--(V_4); 
\draw[draw=black](V_1)--(V_4);
\draw[draw=black](V_5)--(V_6);
\draw[draw=black](V_1)--(V_6);
\draw[draw=black](V_4)--(V_6);
\draw[draw=black](V_3)--(V_5);
\end{tikzpicture}
\hspace{5mm}
\begin{tikzpicture}[line width=1pt,vertex/.style={circle,inner sep=0pt,minimum size=0.2cm}] 

    \pgfmathsetmacro{\n}{5}; 
    \pgfmathsetmacro{\m}{\n-1};

  \node[draw=black,fill=gray, label=above:$a_3$] (V_0) at ($(90-0*360/\n:1.5)$) [vertex] {};
 \node[draw=black,fill=gray, label=above:$c_2$] (V_1) at ($(90-1*360/\n:1.5)$) [vertex] {};
 \node[draw=black,fill=gray, label=right:$a_2$] (V_2) at ($(90-2*360/\n:1.5)$) [vertex] {}; 
\node[draw=black,fill=gray, label=left:$b_3$] (V_3) at ($(90-3*360/\n:1.5)$) [vertex] {};
 \node[draw=black,fill=gray, label=above:$c_1$] (V_4) at ($(90-4*360/\n:1.5)$) [vertex] {};
  \node[draw=black,fill=gray,  label={[label distance=-3mm]87:$b_1$}] (V_{a0}) at ($(90-0*360/\n:0.75)$) [vertex] {};
 \node[draw=black,fill=gray, label={[label distance=-1mm]above:$b_2$}] (V_{a1}) at ($(90-1*360/\n:0.75)$) [vertex] {};
 \node[draw=black,fill=gray, label={[label distance=-1.5mm]85:$v$}] (V_{a2}) at ($(90-2*360/\n:0.75)$) [vertex] {}; 
\node[draw=black,fill=gray, label={[label distance=-1.5mm]95:$c_3$}] (V_{a3}) at ($(90-3*360/\n:0.75)$) [vertex] {};
 \node[draw=black,fill=gray, label={[label distance=-1mm]above:$a_1$}] (V_{a4}) at ($(90-4*360/\n:0.75)$) [vertex] {};
    \foreach \x in {0,...,4} {
  
\draw[draw=black] (V_\x) -- (V_{a\x});
    }

\draw[draw=black](V_1)--(V_2);
\draw[draw=black](V_2)--(V_3);
\draw[draw=black](V_3)--(V_4);
\draw[draw=black](V_4)--(V_0);
\draw[draw=black](V_0)--(V_1);
\draw[draw=black](V_{a1})--(V_{a3});
\draw[draw=black](V_{a2})--(V_{a4});
\draw[draw=black](V_{a3})--(V_{a0});
\draw[draw=black](V_{a4})--(V_{a1});
\draw[draw=black](V_{a0})--(V_{a2});
\end{tikzpicture}
\caption{\label{dist31}}
\end{figure}

\begin{case5}$|S|=2$\textup{:}\end{case5}
Since $|V(H_v)|$ is odd, there are an odd number of vertices of even degree in $H_v$. Since $S$ is a clique, $\delta(H_v)\geq 1$. Hence, by \cref{clique}, there is a unique vertex $x\in V_2(H_v)$, and since $|S|=2$, there is some vertex $v_1\in V_3(H_v)$ adjacent to $x$ in $H_v$. Let $v_2$ and $v_3$ be the other neighbours of $v_1$ in $H_v$, and note that $\{v_2,v_3\}\subseteq V_3(H_v)$. Since $\dist_{H_v}(v_1,y)\leq 2$ for every vertex $y$ in $V_3(H_v)$, each of the four remaining vertices of $H_v-S$ is adjacent to $\{v_2,v_3\}$. Since  $v_2$ and $v_3$ each have only three neighbours in $H_v$, $v_2v_3\notin E(H_v)$. Let $G'$ be obtained from $G$ by deleting every edge in $G\cap H_v$ and then contracting $v_2v_3$. Now $v\in V_8(G')$. Let $v'$ be a vertex in $N_{G'}(v)$. If $v'\in S$, then $|N_{G'}(v')\cap N_{G'}(v)|\geq 8-\deg_{H_v}(v')-1\geq 5$. If $v'$ is in $H_v-(S\cup \{v_2,v_3\})$, then $|N_{G'}(v')\cap N_{G'}(v)|= 8-\deg_{H_v}(v')= 5$. If $v'$ is the new vertex of $G'$, then $|N_{G'}(v')\cap N_{G'}(v)|=8-|N_{H_v}(v_2)\cap N_{H_v}(v_3)|-1=6$. Hence, $|N_{G'}(v')\cap N_{G'}(v)|\geq 5$ for any vertex $v'\in N_{G'}(v)$. Finally, $|N_{G'}(C)|\geq |N_G(C)|-1\geq 4$, so $G'[N_{G'}(C)]\ncong K_3$. Hence $\mathcal{P}$ is a minor of $G$ by \cref{deg8lem}, contradicting (iv).

\begin{case5}There is some triangle $v_1v_2v_3$ of $H_v$ and $S=V_0(H_v)=\{x\}$\textup{:}\end{case5}
Let $\{v_4,v_5,\dots ,v_8\}$ be the other vertices of $H_v$, where $v_4v_1\in E(H_v)$. For every vertex $y$ in $H_v-S$ we have $\dist_{H_v}(v_1,y)\leq 2$ by assumption, so $y$ is either adjacent to $v_1$ or adjacent to a neighbour of $v_1$. Since $\{v_2,v_3,v_4\}\subseteq V_3(H_v)$, we may assume without loss of generality that $\{v_2v_5,v_3v_6,v_4v_7,v_4v_8\}\subseteq E(H_v)$. Since $\Delta(H_v)=3$ and $\dist_{H_v}(v_i,v_j)\leq 2$ for $i\in \{2,3\}$ and $j\in \{7,8\}$, $H_v$ is the graph depicted in \cref{dist32}a. Then $\mathcal{P}\subseteq G[N_G[v]]$ (see \cref{dist32}b), contradicting (iv).

\begin{figure}[htb]
\centering
\begin{tikzpicture}[line width=1pt,vertex/.style={circle,inner sep=0pt,minimum size=0.2cm}] 

    \pgfmathsetmacro{\n}{2}; 

  \node[draw=white, fill=none, label=left: b)] (K) at ($(3.5*\n,0)$) [vertex] {};
\node[draw=white, fill=none, label=left: a)] (K) at ($(-0.5*\n,0)$) [vertex] {};
  \node[draw=black,fill=gray, label=above:$v_2$] (V_0) at ($((0.5*\n,1*\n)$) [vertex] {};
 \node[draw=black,fill=gray, label=above:$v_3$] (V_1) at ($(1.5*\n,1*\n)$) [vertex] {};
 \node[draw=black,fill=gray, label=above:$v_1$] (V_2) at ($(1*\n,0.7*\n)$) [vertex] {}; 
\node[draw=black,fill=gray, label=below:$v_5$] (V_3) at ($(0*\n,0.5*\n)$) [vertex] {};
 \node[draw=black,fill=gray, label=below:$v_4$] (V_4) at ($(1*\n,0.5*\n)$) [vertex] {};
  \node[draw=black,fill=gray, label=below:$v_6$] (V_5) at ($(2*\n,0.5*\n)$) [vertex] {};
 \node[draw=black,fill=gray, label=below:$v_7$] (V_6) at ($(0.5*\n,0*\n)$) [vertex] {};
 \node[draw=black,fill=gray, label=below:$v_8$] (V_7) at ($(1.5*\n,0*\n)$) [vertex] {}; 
\node[draw=black,fill=gray, label=below:$x$] (V_8) at ($(2*\n,0*\n)$) [vertex] {};

\draw[draw=black](V_0)--(V_1);
\draw[draw=black](V_0)--(V_2);
\draw[draw=black](V_1)--(V_2);
\draw[draw=black](V_0)--(V_3);
\draw[draw=black](V_1)--(V_5);
\draw[draw=black](V_2)--(V_4);
\draw[draw=black](V_3)--(V_6); 
\draw[draw=black](V_3)--(V_7);
\draw[draw=black](V_4)--(V_6);
\draw[draw=black](V_4)--(V_7);
\draw[draw=black](V_5)--(V_6);
\draw[draw=black](V_5)--(V_7);
\end{tikzpicture}
\hspace{5mm}
\begin{tikzpicture}[line width=1pt,vertex/.style={circle,inner sep=0pt,minimum size=0.2cm}] 

    \pgfmathsetmacro{\n}{5}; 
    \pgfmathsetmacro{\m}{\n-1};

  \node[draw=black,fill=gray, label=above:$v_6$] (V_0) at ($(90-0*360/\n:1.5)$) [vertex] {};
 \node[draw=black,fill=gray, label=above:$x$] (V_1) at ($(90-1*360/\n:1.5)$) [vertex] {};
 \node[draw=black,fill=gray, label=right:$v_8$] (V_2) at ($(90-2*360/\n:1.5)$) [vertex] {}; 
\node[draw=black,fill=gray, label=left:$v_3$] (V_3) at ($(90-3*360/\n:1.5)$) [vertex] {};
 \node[draw=black,fill=gray, label=above:$v$] (V_4) at ($(90-4*360/\n:1.5)$) [vertex] {};
  \node[draw=black,fill=gray, label={[label distance=-3mm]87:$v_1$}] (V_{a0}) at ($(90-0*360/\n:0.75)$) [vertex] {};
 \node[draw=black,fill=gray, label={[label distance=-1mm]above:$v_4$}] (V_{a1}) at ($(90-1*360/\n:0.75)$) [vertex] {};
 \node[draw=black,fill=gray, label={[label distance=-1.5mm]85:$v_7$}] (V_{a2}) at ($(90-2*360/\n:0.75)$) [vertex] {}; 
\node[draw=black,fill=gray, label={[label distance=-1.5mm]95:$v_5$}] (V_{a3}) at ($(90-3*360/\n:0.75)$) [vertex] {};
 \node[draw=black,fill=gray, label={[label distance=-1mm]above:$v_2$}] (V_{a4}) at ($(90-4*360/\n:0.75)$) [vertex] {};
    \foreach \x in {0,...,4} {
  
\draw[draw=black] (V_\x) -- (V_{a\x});
    }

\draw[draw=black](V_1)--(V_2);
\draw[draw=black](V_2)--(V_3);
\draw[draw=black](V_3)--(V_4);
\draw[draw=black](V_4)--(V_0);
\draw[draw=black](V_0)--(V_1);
\draw[draw=black](V_{a1})--(V_{a3});
\draw[draw=black](V_{a2})--(V_{a4});
\draw[draw=black](V_{a3})--(V_{a0});
\draw[draw=black](V_{a4})--(V_{a1});
\draw[draw=black](V_{a0})--(V_{a2});
\end{tikzpicture}
\caption{\label{dist32}}
\end{figure}

\begin{case5}There is no triangle of $H_v$ and $S=V_0(H_v)=\{x\}$\textup{:}\end{case5}
So $H_v-x$ is a cubic, triangle-free graph, with diameter 2 and exactly eight vertices. We now show that there is exactly one such graph, namely the Wagner graph. Let $v_1$ be a vertex of $H_v-S$, and let $v_2$, $v_3$ and $v_4$ be its neighbours in $H_v$. Since $H_v$ contains no triangle, $\{v_2,v_3,v_4\}$ is an independent set in $H_v$. Let $\{v_5,v_6,v_7,v_8\}$ be the remaining vertices of $H_v-S$. If $v_2$, $v_3$ and $v_4$ all share some common neighbour, say $v_5$, in $H_v$, then there are six edges in $H_v[\{v_1,\dots ,v_5\}]$, and at most three other edges in $H_v$ incident to some vertex in $\{v_1,\dots ,v_5\}$. By the handshaking lemma, $E(H_v-S)=E(H_v)=12$, since $S=V_0(H_v)$ and $V(H_v-S)=V_3(H_v)$. Hence $v_6v_7v_8$ is a triangle of $H_v$, a contradiction. If for every pair $i,j\in \{2,3,4\}$ $v_i$ and $v_j$ share a neighbour in $H_v$ distinct from $v_1$, then $|N_{H_v}[v_2]\cup N_{H_v}[v_3]\cup N_{H_v}[v_4]|\leq 3(4)-3(2)+1=7$ by inclusion-exclusion, contradicting the assumption that $\dist_{H_v}(v_1,y)$ for each of the 8 vertices $y$ in $V_3(H_v)$. Hence, without loss of generality, $v_2$ and $v_3$ have no common neighbour in $H_v$, and $\{v_2v_5,v_2v_6,v_3v_7,v_3v_8\}\subseteq E(H_v)$. Without loss of generality $v_8\in N_{H_v}(v_4)$, since $\{v_5,v_6,v_7,v_8\}\cap N_{H_v}(v_4)\neq \emptyset$. Since $v_7v_3v_8$ is a path in $H_v$ and $H_v$ contains no triangle, the other vertex adjacent to $v_8$ is either $v_5$ or $v_6$, so without loss of generality $v_8v_6\in E(H_v)$. Since $v_5v_2v_6$ and $v_4v_8v_6$ are paths in $H_v$, the remaining vertex adjacent to $v_6$ is $v_7$. Since $V_3(H_v)=V(H_v)\setminus\{x\}$ and $x\in V_0(H_v)$, the remaining two vertices adjacent to $v_5$ are $v_7$ and $v_4$. Hence $H_v$ is the Wagner Graph, plus a single isolated vertex, as illustrated in \cref{dist33}a. Then $\mathcal{P}\subseteq G[N_G[v]]$ (see \cref{dist33}b), contradicting (iv). 

\begin{figure}
\centering
\begin{tikzpicture}[line width=1pt,vertex/.style={circle,inner sep=0pt,minimum size=0.2cm}] 

    \pgfmathsetmacro{\n}{8}; 
\pgfmathsetmacro{\m}{1.2};

  \node[draw=white, fill=none, label=left: b)] (K) at ($(5.5*\m,-1.2)$) [vertex] {};
\node[draw=white, fill=none, label=left: a)] (K) at ($(-2.5*\m,-1.2)$) [vertex] {};
        \node[draw=black,fill=gray, label=above: $v_1$] (V_0) at ($(90-0*360/\n:0.75*\m)$) [vertex] {};
       \node[draw=black,fill=gray, label=45:$v_3$] (V_1) at ($(90-1*360/\n:0.75*\m)$) [vertex] {};
 \node[draw=black,fill=gray, label=right:$v_7$] (V_2) at ($(90-2*360/\n:0.75*\m)$) [vertex] {};
  \node[draw=black,fill=gray, label=-45:$v_5$] (V_3) at ($(90-3*360/\n:0.75*\m)$) [vertex] {};
       \node[draw=black,fill=gray, label=below:$v_4$] (V_4) at ($(90-4*360/\n:0.75*\m)$) [vertex] {};
  \node[draw=black,fill=gray, label=-135:$v_8$] (V_5) at ($(90-5*360/\n:0.75*\m)$) [vertex] {};
       \node[draw=black,fill=gray, label=left:$v_6$] (V_6) at ($(90-6*360/\n:0.75*\m)$) [vertex] {};
 \node[draw=black,fill=gray, label=135:$v_2$] (V_7) at ($(90-7*360/\n:0.75*\m)$) [vertex] {};
\node[draw=black, fill=gray, label=above:$x$] (V_8) at ($(0:2.5)$) [vertex] {};
    
\draw[draw=black](V_1)--(V_2);
\draw[draw=black](V_2)--(V_3);
\draw[draw=black](V_3)--(V_4);
\draw[draw=black](V_4)--(V_5);
\draw[draw=black](V_0)--(V_1);
\draw[draw=black](V_5)--(V_6);
\draw[draw=black](V_6)--(V_7);
\draw[draw=black](V_7)--(V_0);
\draw[draw=black](V_1)--(V_5);
\draw[draw=black](V_2)--(V_6);
\draw[draw=black](V_3)--(V_7);
\draw[draw=black](V_4)--(V_0);
\end{tikzpicture}
\hspace{5mm}
\begin{tikzpicture}[line width=1pt,vertex/.style={circle,inner sep=0pt,minimum size=0.2cm}] 

    \pgfmathsetmacro{\n}{5}; 
    \pgfmathsetmacro{\m}{\n-1};

  \node[draw=black,fill=gray, label=above:$x$] (V_0) at ($(90-0*360/\n:1.5)$) [vertex] {};
 \node[draw=black,fill=gray, label=above:$v_6$] (V_1) at ($(90-1*360/\n:1.5)$) [vertex] {};
 \node[draw=black,fill=gray, label=right:$v_4$] (V_2) at ($(90-2*360/\n:1.5)$) [vertex] {}; 
\node[draw=black,fill=gray, label=left:$v_2$] (V_3) at ($(90-3*360/\n:1.5)$) [vertex] {};
 \node[draw=black,fill=gray, label=above:$v_8$] (V_4) at ($(90-4*360/\n:1.5)$) [vertex] {};
  \node[draw=black,fill=gray,  label={[label distance=-3mm]87:$v$}] (V_{a0}) at ($(90-0*360/\n:0.75)$) [vertex] {};
 \node[draw=black,fill=gray, label={[label distance=-1mm]above:$v_5$}] (V_{a1}) at ($(90-1*360/\n:0.75)$) [vertex] {};
 \node[draw=black,fill=gray, label={[label distance=-1.5mm]85:$v_7$}] (V_{a2}) at ($(90-2*360/\n:0.75)$) [vertex] {}; 
\node[draw=black,fill=gray,  label={[label distance=-1.5mm]95:$v_3$}] (V_{a3}) at ($(90-3*360/\n:0.75)$) [vertex] {};
 \node[draw=black,fill=gray,  label={[label distance=-1mm]above:$v_1$}] (V_{a4}) at ($(90-4*360/\n:0.75)$) [vertex] {};
    \foreach \x in {0,...,4} {
  
\draw[draw=black] (V_\x) -- (V_{a\x});
    }

\draw[draw=black](V_1)--(V_2);
\draw[draw=black](V_2)--(V_3);
\draw[draw=black](V_3)--(V_4);
\draw[draw=black](V_4)--(V_0);
\draw[draw=black](V_0)--(V_1);
\draw[draw=black](V_{a1})--(V_{a3});
\draw[draw=black](V_{a2})--(V_{a4});
\draw[draw=black](V_{a3})--(V_{a0});
\draw[draw=black](V_{a4})--(V_{a1});
\draw[draw=black](V_{a0})--(V_{a2});
\end{tikzpicture}
\caption{\label{dist33}}
\end{figure}

\begin{case5}$S=\{x\}$ and $x\notin V_0(H_v)$\textup{:}\end{case5}

The number of vertices of odd degree in $H_v$ is even, so $x\in V_2(H_v)$. By contracting an edge of $H_v$ incident to $x$, we obtain a cubic graph on eight vertices with diameter at most 2. In Cases 3 and 4 we showed that there are only two such graphs (one with and one without a triangle), so $H_v$ is a copy of one of these in which exactly one edge is subdivided exactly once. It is quick to check that the only such graph in which $\dist (x',y')\leq 2$ whenever $x'$ and $y'$ both have degree 3 is the graph depicted in \cref{dist34}a. Then $\mathcal{P}\subseteq G[N_G[v]]$ (see \cref{dist34}b), contradicting (iv).
\end{proof}

\begin{figure}[htb]
\centering
\begin{tikzpicture}[line width=1pt,vertex/.style={circle,inner sep=0pt,minimum size=0.2cm}] 

    \pgfmathsetmacro{\n}{8}; 
    \pgfmathsetmacro{\m}{1.2};

\node[draw=white, fill=none, label=left: a)] (K) at ($(-2,-1.2)$) [vertex] {};
        \node[draw=black,fill=gray, label=above: $v_1$] (V_0) at ($(90-0*360/\n:0.75*\m)$) [vertex] {};
       \node[draw=black,fill=gray, label=45:$v_2$] (V_1) at ($(90-1*360/\n:0.75*\m)$) [vertex] {};
 \node[draw=black,fill=gray, label=right:$v_3$] (V_2) at ($(90-2*360/\n:0.75*\m)$) [vertex] {};
  \node[draw=black,fill=gray, label=-45:$v_4$] (V_3) at ($(90-3*360/\n:0.75*\m)$) [vertex] {};
       \node[draw=black,fill=gray, label=below:$v_5$] (V_4) at ($(90-4*360/\n:0.75*\m)$) [vertex] {};
  \node[draw=black,fill=gray, label=-135:$v_6$] (V_5) at ($(90-5*360/\n:0.75*\m)$) [vertex] {};
       \node[draw=black,fill=gray, label=left:$v_7$] (V_6) at ($(90-6*360/\n:0.75*\m)$) [vertex] {};
 \node[draw=black,fill=gray, label=135:$v_8$] (V_7) at ($(90-7*360/\n:0.75*\m)$) [vertex] {};
\node[draw=black, fill=gray, label={[label distance=-1.2mm]above:$x$}] (V_8) at ($(0:0.45*\m)$) [vertex] {};
    
\draw[draw=black](V_1)--(V_2);
\draw[draw=black](V_2)--(V_3);
\draw[draw=black](V_3)--(V_4);
\draw[draw=black](V_4)--(V_5);
\draw[draw=black](V_0)--(V_1);
\draw[draw=black](V_5)--(V_6);
\draw[draw=black](V_6)--(V_7);
\draw[draw=black](V_7)--(V_0);
\draw[draw=black](V_1)--(V_5);
\draw[draw=black](V_2)--(V_8);
\draw[draw=black](V_3)--(V_7);
\draw[draw=black](V_4)--(V_0);
\draw[draw=black](V_6)--(V_8);
\end{tikzpicture}
\hspace{0mm}
\begin{tikzpicture}[line width=1pt,vertex/.style={circle,inner sep=0pt,minimum size=0.2cm}] 

    \pgfmathsetmacro{\n}{5}; 
    \pgfmathsetmacro{\m}{\n-1};
\node[draw=white, fill=none, label=left: b)] (K) at ($(-2,-1.2)$) [vertex] {};
  \node[draw=black,fill=gray, label=above:$x$] (V_0) at ($(90-0*360/\n:1.5)$) [vertex] {};
 \node[draw=black,fill=gray, label=above:$v_6$] (V_1) at ($(90-1*360/\n:1.5)$) [vertex] {};
 \node[draw=black,fill=gray, label=right:$v_4$] (V_2) at ($(90-2*360/\n:1.5)$) [vertex] {}; 
\node[draw=black,fill=gray, label=left:$v_7$] (V_3) at ($(90-3*360/\n:1.5)$) [vertex] {};
 \node[draw=black,fill=gray, label=above:$v_5$] (V_4) at ($(90-4*360/\n:1.5)$) [vertex] {};
  \node[draw=black,fill=gray, label={[label distance=-3mm]87:$v$}] (V_{a0}) at ($(90-0*360/\n:0.75)$) [vertex] {};
 \node[draw=black,fill=gray, label={[label distance=-1mm]above:$v_3$}] (V_{a1}) at ($(90-1*360/\n:0.75)$) [vertex] {};
 \node[draw=black,fill=gray, label={[label distance=-1.5mm]85:$v_2$}] (V_{a2}) at ($(90-2*360/\n:0.75)$) [vertex] {}; 
\node[draw=black,fill=gray, label={[label distance=-1.5mm]95:$v_1$}] (V_{a3}) at ($(90-3*360/\n:0.75)$) [vertex] {};
 \node[draw=black,fill=gray, label={[label distance=-1mm]above:$v_8$}] (V_{a4}) at ($(90-4*360/\n:0.75)$) [vertex] {};
    \foreach \x in {0,...,4} {
  
\draw[draw=black] (V_\x) -- (V_{a\x});
    }

\draw[draw=black](V_1)--(V_2);
\draw[draw=black](V_2)--(V_3);
\draw[draw=black](V_3)--(V_4);
\draw[draw=black](V_4)--(V_0);
\draw[draw=black](V_0)--(V_1);
\draw[draw=black](V_{a1})--(V_{a3});
\draw[draw=black](V_{a2})--(V_{a4});
\draw[draw=black](V_{a3})--(V_{a0});
\draw[draw=black](V_{a4})--(V_{a1});
\draw[draw=black](V_{a0})--(V_{a2});
\end{tikzpicture}
\hspace{0mm}
\begin{tikzpicture}[line width=1pt,vertex/.style={circle,inner sep=0pt,minimum size=0.2cm}] 

    \pgfmathsetmacro{\n}{5}; 
    \pgfmathsetmacro{\m}{\n-1};
\node[draw=white, fill=none, label=left: c)] (K) at ($(-2,-1.2)$) [vertex] {};
  \node[draw=black,fill=gray, label=above:$a_1$] (V_0) at ($(90-0*360/\n:1.5)$) [vertex] {};
 \node[draw=black,fill=gray, label=above:$c_2$] (V_1) at ($(90-1*360/\n:1.5)$) [vertex] {};
 \node[draw=black,fill=gray, label=right:$b_3$] (V_2) at ($(90-2*360/\n:1.5)$) [vertex] {}; 
\node[draw=black,fill=gray, label=left:$b_2$] (V_3) at ($(90-3*360/\n:1.5)$) [vertex] {};
 \node[draw=black,fill=gray, label=above:$b_1$] (V_4) at ($(90-4*360/\n:1.5)$) [vertex] {};
  \node[draw=black,fill=gray, label={[label distance=-3mm]87:$a_2$}] (V_{a0}) at ($(90-0*360/\n:0.75)$) [vertex] {};
 \node[draw=black,fill=gray, label={[label distance=-1mm]above:$a_3$}] (V_{a1}) at ($(90-1*360/\n:0.75)$) [vertex] {};
 \node[draw=black,fill=gray,  label={[label distance=-1.5mm]85:$v$}] (V_{a2}) at ($(90-2*360/\n:0.75)$) [vertex] {}; 
\node[draw=black,fill=gray, label={[label distance=-1.5mm]95:$c_3$}] (V_{a3}) at ($(90-3*360/\n:0.75)$) [vertex] {};
 \node[draw=black,fill=gray,  label={[label distance=-1mm]above:$c_1$}] (V_{a4}) at ($(90-4*360/\n:0.75)$) [vertex] {};
    \foreach \x in {0,...,4} {
  
\draw[draw=black] (V_\x) -- (V_{a\x});
    }

\draw[draw=black](V_1)--(V_2);
\draw[draw=black](V_2)--(V_3);
\draw[draw=black](V_3)--(V_4);
\draw[draw=black](V_4)--(V_0);
\draw[draw=black](V_0)--(V_1);
\draw[draw=black](V_{a1})--(V_{a3});
\draw[draw=black](V_{a2})--(V_{a4});
\draw[draw=black](V_{a3})--(V_{a0});
\draw[draw=black](V_{a4})--(V_{a1});
\draw[draw=black](V_{a0})--(V_{a2});
\end{tikzpicture}
\caption{\label{dist34}}
\end{figure}

\begin{clm}\label{K33}If $v\in V_9(G)\cap \mathcal{L}$ and $H_v\cong K_{3,3}\dot{\cup}K_3$, then for each component $C$ of $G- N[v]$, there is some $v$-suitable subgraph $C'$ with $N(C')\setminus N(C)\neq \emptyset$.\end{clm}

\begin{proof}
Let $\{a_1,a_2,a_3,b_1,b_2,b_3,c_1,c_2,c_3\}:=V(H_v)$, with $a_ib_j\in E(H_v)$ for $i,j\in \{1,2,3\}$, and $c_ic_j\in E(H_v)$ for distinct $i,j\in \{1,2,3\}$. Suppose for contradiction that there is a path $P$ of $G$ from $a_i$ to $b_j$ with no internal vertex in $N[v]$ for some $i,j\in \{1,2,3\}$. Without loss of generality, $i=j=1$. Let $G'$ be obtained from $G$ by contracting all but one edge of $P$. Then $\mathcal{P}\subseteq G'$ (see \cref{dist34}c), contradicting (iv). Hence, there is no such path $P$. In particular, no vertex $v'$ in $\{a_1,a_2,a_3, b_1,b_2,b_3\}$ is adjacent to every vertex of $N(v)\setminus \{v'\}$. Hence, since $v\in \mathcal{L}$, for each $v'\in \{a_1,a_2,a_3, b_1,b_2,b_3\}$ there is some 
component $C$ of $G- N[v]$ such that $v'\in N(C)$. However, there is no component $C$ such that $N(C)$ contains some vertex in $\{a_1,a_2,a_3\}$ and some vertex in $\{b_1,b_2,b_3\}$. Hence, for each component $C$ of $G- N[v]$ there is a component $C'$ of $G- N[v]$ with $N(C')\setminus N(C)\neq \emptyset$. Suppose for contradiction that $C'$ is not $v$-suitable. By \cref{separations} $|N(C')|\geq 7$. Since $\overline{G[N(C')]}\subseteq H_v$, there is some vertex in $\{a_1,a_2,a_3\}\cap N(C')$ and some vertex in $\{b_1,b_2,b_3\}\cap N(C')$, a contradiction. Hence $C'$ satisfies our claim.
\end{proof}

\begin{clm} \label{2dist3}If $v\in V_9(G)\cap \mathcal{L}$ and there are two vertices $x$ and $y$ in $V_3(H_v)$ such that $\dist_{H_v}(x,y)\geq 3$ and there is some component $C$ of $G- N[v]$ with $|N(C)|\geq 5$, then for each component $C'$ of $G-N[v]$ there is a $v$-suitable subgraph $C''$ with $N(C'')\setminus N(C')\neq \emptyset$.\end{clm}
\begin{proof}

By \cref{11vertices}, $V(C)\neq \emptyset$. Choose $x$ and $y$, if possible, so that
\begin{align}\label{xyeq}N_{H_v}(x)\cup N_{H_v}(y)\subseteq V_3(H_v).\end{align}

Let $G':=G/xy$, let $x'$ be the new vertex of $G'$, and let $H'=H_v- \{x,y\}$. Note that $\deg_{G'}(v)=8$. Since $|N_G(C)|\geq 5$, we have $|N_{G'}(C)|\geq 4$, and hence $G'[N_{G'}(C)]\ncong K_3$. By \cref{deg8lem} and (iv), $G'$ does not satisfy $|N_{G'}(v')\cap N_{G'}(v)|\geq 5$ for all $v'\in N_{G'}(v)$. 

Now $\{x,y\}\subseteq V_3(H_v)$ and $\dist_{H_v}(x,y)\geq 3$, so $|N_{H_v}(x)\cap N_{H_v}(y)|=6$. Also, since $\overline{G[N(v)]}\subseteq H_v$, there is no common neighbour of $x$ and $y$ in $\overline{G[N(v)]}$, so $x'$ is dominant in $G'[N_{G'}[v]]$, and $|N_{G'}(x')\cap N_{G'}(v)|=7> 5$. 
 
  By \cref{clique}, $\Delta(H_v)=3$. If $v'\in N_{H_v}(x)\cup N_{H_v}(y)$, then since $v'$ is not adjacent to both $x$ and $y$ in $\overline{H_v}$, we have $|N_{G'}(v')\cap N_{G'}(v)|\geq |N_{\overline{H_v}}(v')|\geq 8-3=5$.
  
 Hence, the unique vertex $z$ in $H'-(N_{H_v}(x)\cup N_{H_v}(y))$ satisfies $|N_{G'}(z)\cap N_{G'}(v)|\leq 4$. Thus $z$ has at most three neighbours in $G'[N(v)\setminus \{x,y\}]$ and hence $|N_{H'}(z)|\geq 7-1-3= 3$. Since $\Delta(H')\leq \Delta(H_v)=3$, we have $\deg_{H_v}(z)=3$.
 
 There are an even number of vertices, including $z$, with odd degree in $H'$. We have $\deg_{H'}(v')\leq \Delta(H_v)-1=2$ for the six vertices $v'$ in $N_{H_v}(x)\cup N_{H_v}(y)=V(H'-z)$, so there are an odd number of vertices in $V_1(H')$. Each vertex in $V_1(H')$ has degree at most 2 in $H_v$ since $x$ and $y$ have no common neighbour in $H_v$. So $V_1(H')$ is a clique of $H_v$ by \cref{clique}, and hence a clique of $H'$. Since $|V_1(H')|$ is odd, there is a unique vertex $w$ in $V_1(H')$. By the same argument, the vertices of $V_0(H')\cup V_1(H')$ form a clique of $H'$. No vertex in $V_0(H')$ is adjacent in $H'$ to $w$, so $V_0(H')=\emptyset$. Hence, $V_1(H')=\{w\}$, $V_3(H')=\{z\}$ and $V_2(H')=V(H'-\{w,z\})$.
 
 Now $w$ is one of the six vertices of $N_{H_v}(x)\cup N_{H_v}(y)$, and $\deg_{H_v}(w)\leq \deg_{H'}(w)+1\leq 2$. In particular $x$ and $y$ do not satisfy (\ref{xyeq}), so no such pair satisfy (\ref{xyeq}). This means, there are no two vertices $x'$ and $y'$ in $V_3(H_v)$ that satisfy (\ref{xyeq}) such that $\dist_{H_v}(x',y')\geq 3$. 
 
 We consider four cases depending on whether $H'$ is connected and on the components of $G-N[v]$.

\begin{case6}$H'$ is not connected\textup{:}\end{case6}
 Since each connected component of $H'$ has an even number of vertices of odd degree, $z$ and $w$ are in the same component, and each other component is a cycle. Since $|V(H')\setminus N_{H'}[z]|=3$, there is a unique component $D$ of $H'$ not containing $z$ and $D$ is a triangle. Since $|V_2(H')|= 5$,  there is some vertex $x_0$ of degree 2 not in $D$ and not adjacent to $w$. Assume without loss of generality that $x_0$ is adjacent to $x$ in $H_v$. Since $x\in V_3(H_v)$, there is some vertex $y_0$ in $D$ such that $y_0x\notin E(H_v)$. Now $y_0$ is adjacent to no neighbour of $x_0$ in $H_v$, so $\dist_{H_v}(x,y)\geq 3$. But the vertices adjacent to $\{x_0,y_0\}$ in $H_v$ are all in $V_3(H_v)$ since $w$ is adjacent to neither $x_0$ nor $y_0$ in $H_v$. Therefore $x_0$ and $y_0$ satisfy (\ref{xyeq}), a contradiction.

\bigskip

For the remaining cases, $H'$ is a connected graph such that $|V_1(H')|=|V_3(H')|=1$ and every other vertex has degree 2. Hence, $H'$ is composed of a path $P$ from $z$ to $w$ and a cycle $Q$ of size at least 3 containing $z$, with $V(P\cap Q)=\{z\}$. Let $z_0$ be the neighbour of $z$ in the path from $z$ to $w$, and let $z_1$ and $z_2$ be the other neighbours of $z$ in $H'$.

\begin{case6}$H'$ is connected and there is some component $D$ of $G-N[v]$ such that $z\in N(D)$ and $|N(D)\cap N_{H'}(z)|\geq 2$\textup{:}\end{case6}
At least one vertex is in $\{z_1,z_2\}\cap N(D)$, so without loss of generality $z_1\in N(D)$. Either $z_0$ or $z_2$ is also in $N(D)$. Since $V(Q)\subseteq V(H')\setminus \{w\}$, we have $3\leq |V(Q)|\leq 6$. Let $G'':=G'/E(D)$. The diagrams in \cref{2dist31} demonstrate that $\mathcal{P}\subseteq G''$, contradicting (iv).

\begin{table}[!t]
\centering
\caption{\label{2dist31}}
\begin{tabular}{|c|c|}
\hline
\begin{tikzpicture}[line width=1pt,vertex/.style={circle,inner sep=0pt,minimum size=0.2cm}] 

    \pgfmathsetmacro{\n}{1}; 
    \pgfmathsetmacro{\m}{\n-1};

        \node[draw=black,fill=none, label=above:$z_1$] (V_0) at ($(0,2*\n)$) [vertex] {};
       \node[draw=black,fill=none, label=above:$z_2$] (V_1) at ($(\n,2*\n)$) [vertex] {};
 \node[draw=black,fill=none, label=left:$z$] (V_2) at ($(0,\n)$) [vertex] {};
  \node[draw=black,fill=gray, label=below:$z_0$] (V_3) at ($(0,0)$) [vertex] {};
       \node[draw=black,fill=gray, label=below:$v_1$] (V_4) at ($(\n,0)$) [vertex] {};
  \node[draw=black,fill=gray, label=below:$v_2$] (V_5) at ($(2*\n,0)$) [vertex] {};
       \node[draw=black,fill=gray, label=above:$w$] (V_6) at ($(2*\n,\n)$) [vertex] {};

\draw[draw=black](V_1)--(V_2);
\draw[draw=black](V_0)--(V_1);
\draw[draw=black](V_0)--(V_2);
\draw[draw=black](V_2)--(V_3);
\draw[draw=black](V_3)--(V_4);
\draw[draw=black](V_5)--(V_6);
\draw[draw=black](V_4)--(V_5);

\end{tikzpicture}
\begin{tikzpicture}[line width=1pt,vertex/.style={circle,inner sep=0pt,minimum size=0.2cm}] 

    \pgfmathsetmacro{\n}{5}; 
    \pgfmathsetmacro{\m}{\n-1};

  \node[draw=black,fill=gray, label=above:$D$] (V_0) at ($(90-0*360/\n:1.5)$) [vertex] {};
 \node[draw=black,fill=none, label=above:$z_2$] (V_1) at ($(90-1*360/\n:1.5)$) [vertex] {};
 \node[draw=black,fill=gray, label=right:$x'$] (V_2) at ($(90-2*360/\n:1.5)$) [vertex] {}; 
\node[draw=black,fill=gray, label=left:$v$] (V_3) at ($(90-3*360/\n:1.5)$) [vertex] {};
 \node[draw=black,fill=none, label=above:$z_1$] (V_4) at ($(90-4*360/\n:1.5)$) [vertex] {};
  \node[draw=black,fill=none, label={[label distance=-3mm]87:$z$}] (V_{a0}) at ($(90-0*360/\n:0.75)$) [vertex] {};
 \node[draw=black,fill=gray, label={[label distance=-1mm]above:$z_0$}] (V_{a1}) at ($(90-1*360/\n:0.75)$) [vertex] {};
 \node[draw=black,fill=gray,  label={[label distance=-1.5mm]85:$v_1$}] (V_{a2}) at ($(90-2*360/\n:0.75)$) [vertex] {}; 
\node[draw=black,fill=gray,  label={[label distance=-1.5mm]95:$v_2$}] (V_{a3}) at ($(90-3*360/\n:0.75)$) [vertex] {};
 \node[draw=black,fill=gray, label={[label distance=-1mm]above:$w$}] (V_{a4}) at ($(90-4*360/\n:0.75)$) [vertex] {};
    \foreach \x in {0,...,4} {
  
\draw[draw=black] (V_\x) -- (V_{a\x});
    }

\draw[draw=black](V_1)--(V_2);
\draw[draw=black](V_2)--(V_3);
\draw[draw=black](V_3)--(V_4);
\draw[draw=black](V_4)--(V_0);
\draw[draw=black](V_0)--(V_1);
\draw[draw=black](V_{a1})--(V_{a3});
\draw[draw=black](V_{a2})--(V_{a4});
\draw[draw=black](V_{a3})--(V_{a0});
\draw[draw=black](V_{a4})--(V_{a1});
\draw[draw=black](V_{a0})--(V_{a2});
\end{tikzpicture}
&
\begin{tikzpicture}[line width=1pt,vertex/.style={circle,inner sep=0pt,minimum size=0.2cm}] 

    \pgfmathsetmacro{\n}{1}; 
    \pgfmathsetmacro{\m}{\n-1};

        \node[draw=black,fill=none, label=above:$z_1$] (V_0) at ($(0,2*\n)$) [vertex] {};
       \node[draw=black,fill=gray, label=above:$z_2$] (V_1) at ($(\n,2*\n)$) [vertex] {};
 \node[draw=black,fill=none, label=left:$z$] (V_2) at ($(0,\n)$) [vertex] {};
  \node[draw=black,fill=none, label=below:$z_0$] (V_3) at ($(0,0)$) [vertex] {};
       \node[draw=black,fill=gray, label=below:$v_1$] (V_4) at ($(\n,0)$) [vertex] {};
  \node[draw=black,fill=gray, label=below:$v_2$] (V_5) at ($(2*\n,0)$) [vertex] {};
       \node[draw=black,fill=gray, label=above:$w$] (V_6) at ($(2*\n,\n)$) [vertex] {};

\draw[draw=black](V_1)--(V_2);
\draw[draw=black](V_0)--(V_1);
\draw[draw=black](V_0)--(V_2);
\draw[draw=black](V_2)--(V_3);
\draw[draw=black](V_3)--(V_4);
\draw[draw=black](V_5)--(V_6);
\draw[draw=black](V_4)--(V_5);

\end{tikzpicture}
\begin{tikzpicture}[line width=1pt,vertex/.style={circle,inner sep=0pt,minimum size=0.2cm}] 

    \pgfmathsetmacro{\n}{5}; 
    \pgfmathsetmacro{\m}{\n-1};

  \node[draw=black,fill=gray, label=above:$D$] (V_0) at ($(90-0*360/\n:1.5)$) [vertex] {};
 \node[draw=black,fill=none, label=above:$z_0$] (V_1) at ($(90-1*360/\n:1.5)$) [vertex] {};
 \node[draw=black,fill=gray, label=right:$x'$] (V_2) at ($(90-2*360/\n:1.5)$) [vertex] {}; 
\node[draw=black,fill=gray, label=left:$v$] (V_3) at ($(90-3*360/\n:1.5)$) [vertex] {};
 \node[draw=black,fill=none, label=above:$z_1$] (V_4) at ($(90-4*360/\n:1.5)$) [vertex] {};
  \node[draw=black,fill=none, label={[label distance=-3mm]87:$z$}] (V_{a0}) at ($(90-0*360/\n:0.75)$) [vertex] {};
 \node[draw=black,fill=gray, label={[label distance=-1mm]above:$z_2$}] (V_{a1}) at ($(90-1*360/\n:0.75)$) [vertex] {};
 \node[draw=black,fill=gray,  label={[label distance=-1.5mm]85:$v_1$}] (V_{a2}) at ($(90-2*360/\n:0.75)$) [vertex] {}; 
\node[draw=black,fill=gray, label={[label distance=-1.5mm]95:$v_2$}] (V_{a3}) at ($(90-3*360/\n:0.75)$) [vertex] {};
 \node[draw=black,fill=gray, label={[label distance=-1mm]above:$w$}] (V_{a4}) at ($(90-4*360/\n:0.75)$) [vertex] {};
    \foreach \x in {0,...,4} {
  
\draw[draw=black] (V_\x) -- (V_{a\x});
    }

\draw[draw=black](V_1)--(V_2);
\draw[draw=black](V_2)--(V_3);
\draw[draw=black](V_3)--(V_4);
\draw[draw=black](V_4)--(V_0);
\draw[draw=black](V_0)--(V_1);
\draw[draw=black](V_{a1})--(V_{a3});
\draw[draw=black](V_{a2})--(V_{a4});
\draw[draw=black](V_{a3})--(V_{a0});
\draw[draw=black](V_{a4})--(V_{a1});
\draw[draw=black](V_{a0})--(V_{a2});
\end{tikzpicture}
\\
$|V(Q)|=3$, $z_2\in N(D)$&$|V(Q)|=3$, $z_0\in N(D)$
\\
\hline
\begin{tikzpicture}[line width=1pt,vertex/.style={circle,inner sep=0pt,minimum size=0.2cm}] 

    \pgfmathsetmacro{\n}{1}; 
    \pgfmathsetmacro{\m}{\n-1};

        \node[draw=black,fill=gray, label=above:$v_1$] (V_0) at ($(\n,2*\n)$) [vertex] {};
       \node[draw=black,fill=none, label=above:$z_1$] (V_1) at ($(0,2*\n)$) [vertex] {};
 \node[draw=black,fill=none, label=below:$z_2$] (V_2) at ($(\n,\n)$) [vertex] {};
  \node[draw=black,fill=none, label=left:$z$] (V_3) at ($(0,\n)$) [vertex] {};
       \node[draw=black,fill=gray, label=below:$z_0$] (V_4) at ($(0,0)$) [vertex] {};
  \node[draw=black,fill=gray, label=below:$v_2$] (V_5) at ($(\n,0)$) [vertex] {};
       \node[draw=black,fill=gray, label=below:$w$] (V_6) at ($(2*\n,0)$) [vertex] {};

\draw[draw=black](V_0)--(V_1);
\draw[draw=black](V_0)--(V_2);
\draw[draw=black](V_1)--(V_3);
\draw[draw=black](V_2)--(V_3);
\draw[draw=black](V_3)--(V_4);
\draw[draw=black](V_5)--(V_6);
\draw[draw=black](V_4)--(V_5);

\end{tikzpicture}
\begin{tikzpicture}[line width=1pt,vertex/.style={circle,inner sep=0pt,minimum size=0.2cm}] 

    \pgfmathsetmacro{\n}{5}; 
    \pgfmathsetmacro{\m}{\n-1};

  \node[draw=black,fill=gray, label=above:$D$] (V_0) at ($(90-0*360/\n:1.5)$) [vertex] {};
 \node[draw=black,fill=none, label=above:$z_2$] (V_1) at ($(90-1*360/\n:1.5)$) [vertex] {};
 \node[draw=black,fill=gray, label=right:$x'$] (V_2) at ($(90-2*360/\n:1.5)$) [vertex] {}; 
\node[draw=black,fill=gray, label=left:$w$] (V_3) at ($(90-3*360/\n:1.5)$) [vertex] {};
 \node[draw=black,fill=none, label=above:$z_1$] (V_4) at ($(90-4*360/\n:1.5)$) [vertex] {};
  \node[draw=black,fill=none,  label={[label distance=-3mm]87:$z$}] (V_{a0}) at ($(90-0*360/\n:0.75)$) [vertex] {};
 \node[draw=black,fill=gray, label={[label distance=-1mm]above:$z_0$}] (V_{a1}) at ($(90-1*360/\n:0.75)$) [vertex] {};
 \node[draw=black,fill=gray, label={[label distance=-1.5mm]85:$v_2$}] (V_{a2}) at ($(90-2*360/\n:0.75)$) [vertex] {}; 
\node[draw=black,fill=gray,  label={[label distance=-1.5mm]95:$v_1$}] (V_{a3}) at ($(90-3*360/\n:0.75)$) [vertex] {};
 \node[draw=black,fill=gray, label={[label distance=-1mm]above:$v$}] (V_{a4}) at ($(90-4*360/\n:0.75)$) [vertex] {};
    \foreach \x in {0,...,4} {
  
\draw[draw=black] (V_\x) -- (V_{a\x});
    }

\draw[draw=black](V_1)--(V_2);
\draw[draw=black](V_2)--(V_3);
\draw[draw=black](V_3)--(V_4);
\draw[draw=black](V_4)--(V_0);
\draw[draw=black](V_0)--(V_1);
\draw[draw=black](V_{a1})--(V_{a3});
\draw[draw=black](V_{a2})--(V_{a4});
\draw[draw=black](V_{a3})--(V_{a0});
\draw[draw=black](V_{a4})--(V_{a1});
\draw[draw=black](V_{a0})--(V_{a2});
\end{tikzpicture}
&
\begin{tikzpicture}[line width=1pt,vertex/.style={circle,inner sep=0pt,minimum size=0.2cm}] 

    \pgfmathsetmacro{\n}{1}; 
    \pgfmathsetmacro{\m}{\n-1};

        \node[draw=black,fill=gray, label=above:$v_1$] (V_0) at ($(\n,2*\n)$) [vertex] {};
       \node[draw=black,fill=none, label=above:$z_1$] (V_1) at ($(0,2*\n)$) [vertex] {};
 \node[draw=black,fill=gray, label=below:$z_2$] (V_2) at ($(\n,\n)$) [vertex] {};
  \node[draw=black,fill=none, label=left:$z$] (V_3) at ($(0,\n)$) [vertex] {};
       \node[draw=black,fill=none, label=below:$z_0$] (V_4) at ($(0,0)$) [vertex] {};
  \node[draw=black,fill=gray, label=below:$v_2$] (V_5) at ($(\n,0)$) [vertex] {};
       \node[draw=black,fill=gray, label=below:$w$] (V_6) at ($(2*\n,0)$) [vertex] {};

\draw[draw=black](V_0)--(V_1);
\draw[draw=black](V_0)--(V_2);
\draw[draw=black](V_1)--(V_3);
\draw[draw=black](V_2)--(V_3);
\draw[draw=black](V_3)--(V_4);
\draw[draw=black](V_5)--(V_6);
\draw[draw=black](V_4)--(V_5);

\end{tikzpicture}
\begin{tikzpicture}[line width=1pt,vertex/.style={circle,inner sep=0pt,minimum size=0.2cm}] 

    \pgfmathsetmacro{\n}{5}; 
    \pgfmathsetmacro{\m}{\n-1};

  \node[draw=black,fill=gray, label=above:$D$] (V_0) at ($(90-0*360/\n:1.5)$) [vertex] {};
 \node[draw=black,fill=none, label=above:$z_0$] (V_1) at ($(90-1*360/\n:1.5)$) [vertex] {};
 \node[draw=black,fill=gray, label=right:$x'$] (V_2) at ($(90-2*360/\n:1.5)$) [vertex] {}; 
\node[draw=black,fill=gray, label=left:$w$] (V_3) at ($(90-3*360/\n:1.5)$) [vertex] {};
 \node[draw=black,fill=none, label=above:$z_1$] (V_4) at ($(90-4*360/\n:1.5)$) [vertex] {};
  \node[draw=black,fill=none, label={[label distance=-3mm]87:$z$}] (V_{a0}) at ($(90-0*360/\n:0.75)$) [vertex] {};
 \node[draw=black,fill=gray, label={[label distance=-1mm]above:$v$}] (V_{a1}) at ($(90-1*360/\n:0.75)$) [vertex] {};
 \node[draw=black,fill=gray, label={[label distance=-1.5mm]85:$v_2$}] (V_{a2}) at ($(90-2*360/\n:0.75)$) [vertex] {}; 
\node[draw=black,fill=gray,  label={[label distance=-1.5mm]95:$v_1$}] (V_{a3}) at ($(90-3*360/\n:0.75)$) [vertex] {};
 \node[draw=black,fill=gray, label={[label distance=-1mm]above:$z_2$}] (V_{a4}) at ($(90-4*360/\n:0.75)$) [vertex] {};
    \foreach \x in {0,...,4} {
  
\draw[draw=black] (V_\x) -- (V_{a\x});
    }

\draw[draw=black](V_1)--(V_2);
\draw[draw=black](V_2)--(V_3);
\draw[draw=black](V_3)--(V_4);
\draw[draw=black](V_4)--(V_0);
\draw[draw=black](V_0)--(V_1);
\draw[draw=black](V_{a1})--(V_{a3});
\draw[draw=black](V_{a2})--(V_{a4});
\draw[draw=black](V_{a3})--(V_{a0});
\draw[draw=black](V_{a4})--(V_{a1});
\draw[draw=black](V_{a0})--(V_{a2});
\end{tikzpicture}
\\
$|V(Q)|=4$, $z_2\in N(D)$&$|V(Q)|=4$, $z_0\in N(D)$
\\
\hline
\begin{tikzpicture}[line width=1pt,vertex/.style={circle,inner sep=0pt,minimum size=0.2cm}] 

    \pgfmathsetmacro{\n}{1}; 
    \pgfmathsetmacro{\m}{\n-1};

        \node[draw=black,fill=gray, label=above:$v_1$] (V_0) at ($(\n,2*\n)$) [vertex] {};
       \node[draw=black,fill=gray, label=above:$v_2$] (V_1) at ($(2*\n,2*\n)$) [vertex] {};
 \node[draw=black,fill=none, label=above:$z_1$] (V_2) at ($(0,2*\n)$) [vertex] {};
  \node[draw=black,fill=none, label=below:$z_2$] (V_3) at ($(\n,\n)$) [vertex] {};
       \node[draw=black,fill=none, label=left:$z$] (V_4) at ($(0,\n)$) [vertex] {};
  \node[draw=black,fill=gray, label=below:$z_0$] (V_5) at ($(0,0)$) [vertex] {};
       \node[draw=black,fill=gray, label=below:$w$] (V_6) at ($(\n,0)$) [vertex] {};

\draw[draw=black](V_0)--(V_1);
\draw[draw=black](V_0)--(V_2);
\draw[draw=black](V_1)--(V_3);
\draw[draw=black](V_2)--(V_4);
\draw[draw=black](V_3)--(V_4);
\draw[draw=black](V_5)--(V_6);
\draw[draw=black](V_4)--(V_5);

\end{tikzpicture}
\begin{tikzpicture}[line width=1pt,vertex/.style={circle,inner sep=0pt,minimum size=0.2cm}] 

    \pgfmathsetmacro{\n}{5}; 
    \pgfmathsetmacro{\m}{\n-1};

  \node[draw=black,fill=gray, label=above:$D$] (V_0) at ($(90-0*360/\n:1.5)$) [vertex] {};
 \node[draw=black,fill=none, label=above:$z_2$] (V_1) at ($(90-1*360/\n:1.5)$) [vertex] {};
 \node[draw=black,fill=gray, label=right:$w$] (V_2) at ($(90-2*360/\n:1.5)$) [vertex] {}; 
\node[draw=black,fill=gray, label=left:$x'$] (V_3) at ($(90-3*360/\n:1.5)$) [vertex] {};
 \node[draw=black,fill=none, label=above:$z_1$] (V_4) at ($(90-4*360/\n:1.5)$) [vertex] {};
  \node[draw=black,fill=none,  label={[label distance=-3mm]87:$z$}] (V_{a0}) at ($(90-0*360/\n:0.75)$) [vertex] {};
 \node[draw=black,fill=gray, label={[label distance=-1mm]above:$z_0$}] (V_{a1}) at ($(90-1*360/\n:0.75)$) [vertex] {};
 \node[draw=black,fill=gray,  label={[label distance=-1.5mm]85:$v_2$}] (V_{a2}) at ($(90-2*360/\n:0.75)$) [vertex] {}; 
\node[draw=black,fill=gray, label={[label distance=-1.5mm]95:$v_1$}] (V_{a3}) at ($(90-3*360/\n:0.75)$) [vertex] {};
 \node[draw=black,fill=gray,  label={[label distance=-1mm]above:$v$}] (V_{a4}) at ($(90-4*360/\n:0.75)$) [vertex] {};
    \foreach \x in {0,...,4} {
  
\draw[draw=black] (V_\x) -- (V_{a\x});
    }

\draw[draw=black](V_1)--(V_2);
\draw[draw=black](V_2)--(V_3);
\draw[draw=black](V_3)--(V_4);
\draw[draw=black](V_4)--(V_0);
\draw[draw=black](V_0)--(V_1);
\draw[draw=black](V_{a1})--(V_{a3});
\draw[draw=black](V_{a2})--(V_{a4});
\draw[draw=black](V_{a3})--(V_{a0});
\draw[draw=black](V_{a4})--(V_{a1});
\draw[draw=black](V_{a0})--(V_{a2});
\end{tikzpicture}
&
\begin{tikzpicture}[line width=1pt,vertex/.style={circle,inner sep=0pt,minimum size=0.2cm}] 

    \pgfmathsetmacro{\n}{1}; 
    \pgfmathsetmacro{\m}{\n-1};

        \node[draw=black,fill=gray, label=above:$v_1$] (V_0) at ($(\n,2*\n)$) [vertex] {};
       \node[draw=black,fill=gray, label=above:$v_2$] (V_1) at ($(2*\n,2*\n)$) [vertex] {};
 \node[draw=black,fill=none, label=above:$z_1$] (V_2) at ($(0,2*\n)$) [vertex] {};
  \node[draw=black,fill=none, label=below:$z_2$] (V_3) at ($(\n,\n)$) [vertex] {};
       \node[draw=black,fill=none, label=left:$z$] (V_4) at ($(0,\n)$) [vertex] {};
  \node[draw=black,fill=gray, label=below:$z_0$] (V_5) at ($(0,0)$) [vertex] {};
       \node[draw=black,fill=gray, label=below:$w$] (V_6) at ($(\n,0)$) [vertex] {};

\draw[draw=black](V_0)--(V_1);
\draw[draw=black](V_0)--(V_2);
\draw[draw=black](V_1)--(V_3);
\draw[draw=black](V_2)--(V_4);
\draw[draw=black](V_3)--(V_4);
\draw[draw=black](V_5)--(V_6);
\draw[draw=black](V_4)--(V_5);

\end{tikzpicture}
\begin{tikzpicture}[line width=1pt,vertex/.style={circle,inner sep=0pt,minimum size=0.2cm}] 

    \pgfmathsetmacro{\n}{5}; 
    \pgfmathsetmacro{\m}{\n-1};

  \node[draw=black,fill=gray, label=above:$D$] (V_0) at ($(90-0*360/\n:1.5)$) [vertex] {};
 \node[draw=black,fill=none, label=above:$z_0$] (V_1) at ($(90-1*360/\n:1.5)$) [vertex] {};
 \node[draw=black,fill=gray, label=right:$x'$] (V_2) at ($(90-2*360/\n:1.5)$) [vertex] {}; 
\node[draw=black,fill=gray, label=left:$w$] (V_3) at ($(90-3*360/\n:1.5)$) [vertex] {};
 \node[draw=black,fill=none, label=above:$z_1$] (V_4) at ($(90-4*360/\n:1.5)$) [vertex] {};
  \node[draw=black,fill=none, label={[label distance=-3mm]87:$z$}] (V_{a0}) at ($(90-0*360/\n:0.75)$) [vertex] {};
 \node[draw=black,fill=gray, label={[label distance=-1mm]above:$z_2$}] (V_{a1}) at ($(90-1*360/\n:0.75)$) [vertex] {};
 \node[draw=black,fill=gray, label={[label distance=-1.5mm]85:$v_2$}] (V_{a2}) at ($(90-2*360/\n:0.75)$) [vertex] {}; 
\node[draw=black,fill=gray, label={[label distance=-1.5mm]95:$v_1$}] (V_{a3}) at ($(90-3*360/\n:0.75)$) [vertex] {};
 \node[draw=black,fill=gray, label={[label distance=-1mm]above:$v$}] (V_{a4}) at ($(90-4*360/\n:0.75)$) [vertex] {};
    \foreach \x in {0,...,4} {
  
\draw[draw=black] (V_\x) -- (V_{a\x});
    }

\draw[draw=black](V_1)--(V_2);
\draw[draw=black](V_2)--(V_3);
\draw[draw=black](V_3)--(V_4);
\draw[draw=black](V_4)--(V_0);
\draw[draw=black](V_0)--(V_1);
\draw[draw=black](V_{a1})--(V_{a3});
\draw[draw=black](V_{a2})--(V_{a4});
\draw[draw=black](V_{a3})--(V_{a0});
\draw[draw=black](V_{a4})--(V_{a1});
\draw[draw=black](V_{a0})--(V_{a2});
\end{tikzpicture}
\\
$|V(Q)|=5$, $z_2\in N(D)$&$|V(Q)|=5$, $z_0\in N(D)$
\\
\hline
\begin{tikzpicture}[line width=1pt,vertex/.style={circle,inner sep=0pt,minimum size=0.2cm}] 

    \pgfmathsetmacro{\n}{1}; 
    \pgfmathsetmacro{\m}{\n-1};

        \node[draw=black,fill=gray, label=above:$v_2$] (V_0) at ($(2*\n,2*\n)$) [vertex] {};
       \node[draw=black,fill=gray, label=above:$v_1$] (V_1) at ($(\n,2*\n)$) [vertex] {};
 \node[draw=black,fill=gray, label=below:$v_3$] (V_2) at ($(2*\n,\n)$) [vertex] {};
  \node[draw=black,fill=none, label=above:$z_1$] (V_3) at ($(0,2*\n)$) [vertex] {};
       \node[draw=black,fill=none, label=below:$z_2$] (V_4) at ($(\n,\n)$) [vertex] {};
  \node[draw=black,fill=none, label=left:$z$] (V_5) at ($(0,\n)$) [vertex] {};
       \node[draw=black,fill=gray, label=below:$z_0$] (V_6) at ($(0,0)$) [vertex] {};

\draw[draw=black](V_0)--(V_1);
\draw[draw=black](V_0)--(V_2);
\draw[draw=black](V_1)--(V_3);
\draw[draw=black](V_2)--(V_4);
\draw[draw=black](V_3)--(V_5);
\draw[draw=black](V_5)--(V_6);
\draw[draw=black](V_4)--(V_5);

\end{tikzpicture}
\begin{tikzpicture}[line width=1pt,vertex/.style={circle,inner sep=0pt,minimum size=0.2cm}] 

    \pgfmathsetmacro{\n}{5}; 
    \pgfmathsetmacro{\m}{\n-1};

  \node[draw=black,fill=gray, label=above:$D$] (V_0) at ($(90-0*360/\n:1.5)$) [vertex] {};
 \node[draw=black,fill=none, label=above:$z_2$] (V_1) at ($(90-1*360/\n:1.5)$) [vertex] {};
 \node[draw=black,fill=gray, label=right:$v_1$] (V_2) at ($(90-2*360/\n:1.5)$) [vertex] {}; 
\node[draw=black,fill=gray, label=left:$x'$] (V_3) at ($(90-3*360/\n:1.5)$) [vertex] {};
 \node[draw=black,fill=none, label=above:$z_1$] (V_4) at ($(90-4*360/\n:1.5)$) [vertex] {};
  \node[draw=black,fill=none, label={[label distance=-3mm]87:$z$}] (V_{a0}) at ($(90-0*360/\n:0.75)$) [vertex] {};
 \node[draw=black,fill=gray,  label={[label distance=-1mm]above:$z_0$}] (V_{a1}) at ($(90-1*360/\n:0.75)$) [vertex] {};
 \node[draw=black,fill=gray,  label={[label distance=-1.5mm]85:$v_3$}] (V_{a2}) at ($(90-2*360/\n:0.75)$) [vertex] {}; 
\node[draw=black,fill=gray,  label={[label distance=-1.5mm]95:$v_2$}] (V_{a3}) at ($(90-3*360/\n:0.75)$) [vertex] {};
 \node[draw=black,fill=gray, label={[label distance=-1mm]above:$v$}] (V_{a4}) at ($(90-4*360/\n:0.75)$) [vertex] {};
    \foreach \x in {0,...,4} {
  
\draw[draw=black] (V_\x) -- (V_{a\x});
    }

\draw[draw=black](V_1)--(V_2);
\draw[draw=black](V_2)--(V_3);
\draw[draw=black](V_3)--(V_4);
\draw[draw=black](V_4)--(V_0);
\draw[draw=black](V_0)--(V_1);
\draw[draw=black](V_{a1})--(V_{a3});
\draw[draw=black](V_{a2})--(V_{a4});
\draw[draw=black](V_{a3})--(V_{a0});
\draw[draw=black](V_{a4})--(V_{a1});
\draw[draw=black](V_{a0})--(V_{a2});
\end{tikzpicture}
&
\begin{tikzpicture}[line width=1pt,vertex/.style={circle,inner sep=0pt,minimum size=0.2cm}] 

    \pgfmathsetmacro{\n}{1}; 
    \pgfmathsetmacro{\m}{\n-1};

        \node[draw=black,fill=gray, label=above:$v_2$] (V_0) at ($(2*\n,2*\n)$) [vertex] {};
       \node[draw=black,fill=gray, label=above:$v_1$] (V_1) at ($(\n,2*\n)$) [vertex] {};
 \node[draw=black,fill=gray, label=below:$v_3$] (V_2) at ($(2*\n,\n)$) [vertex] {};
  \node[draw=black,fill=none, label=above:$z_1$] (V_3) at ($(0,2*\n)$) [vertex] {};
       \node[draw=black,fill=gray, label=below:$z_2$] (V_4) at ($(\n,\n)$) [vertex] {};
  \node[draw=black,fill=none, label=left:$z$] (V_5) at ($(0,\n)$) [vertex] {};
       \node[draw=black,fill=none, label=below:$z_0$] (V_6) at ($(0,0)$) [vertex] {};

\draw[draw=black](V_0)--(V_1);
\draw[draw=black](V_0)--(V_2);
\draw[draw=black](V_1)--(V_3);
\draw[draw=black](V_2)--(V_4);
\draw[draw=black](V_3)--(V_5);
\draw[draw=black](V_5)--(V_6);
\draw[draw=black](V_4)--(V_5);

\end{tikzpicture}
\begin{tikzpicture}[line width=1pt,vertex/.style={circle,inner sep=0pt,minimum size=0.2cm}] 

    \pgfmathsetmacro{\n}{5}; 
    \pgfmathsetmacro{\m}{\n-1};

  \node[draw=black,fill=gray, label=above:$D$] (V_0) at ($(90-0*360/\n:1.5)$) [vertex] {};
 \node[draw=black,fill=none, label=above:$z_0$] (V_1) at ($(90-1*360/\n:1.5)$) [vertex] {};
 \node[draw=black,fill=gray, label=right:$z_2$] (V_2) at ($(90-2*360/\n:1.5)$) [vertex] {}; 
\node[draw=black,fill=gray, label=left:$x'$] (V_3) at ($(90-3*360/\n:1.5)$) [vertex] {};
 \node[draw=black,fill=none, label=above:$z_1$] (V_4) at ($(90-4*360/\n:1.5)$) [vertex] {};
  \node[draw=black,fill=none, label={[label distance=-3mm]87:$z$}] (V_{a0}) at ($(90-0*360/\n:0.75)$) [vertex] {};
 \node[draw=black,fill=gray, label={[label distance=-1mm]above:$v_3$}] (V_{a1}) at ($(90-1*360/\n:0.75)$) [vertex] {};
 \node[draw=black,fill=gray, label={[label distance=-1.5mm]85:$v_2$}] (V_{a2}) at ($(90-2*360/\n:0.75)$) [vertex] {}; 
\node[draw=black,fill=gray,  label={[label distance=-1.5mm]95:$v_1$}] (V_{a3}) at ($(90-3*360/\n:0.75)$) [vertex] {};
 \node[draw=black,fill=gray, label={[label distance=-1mm]above:$v$}] (V_{a4}) at ($(90-4*360/\n:0.75)$) [vertex] {};
    \foreach \x in {0,...,4} {
  
\draw[draw=black] (V_\x) -- (V_{a\x});
    }

\draw[draw=black](V_1)--(V_2);
\draw[draw=black](V_2)--(V_3);
\draw[draw=black](V_3)--(V_4);
\draw[draw=black](V_4)--(V_0);
\draw[draw=black](V_0)--(V_1);
\draw[draw=black](V_{a1})--(V_{a3});
\draw[draw=black](V_{a2})--(V_{a4});
\draw[draw=black](V_{a3})--(V_{a0});
\draw[draw=black](V_{a4})--(V_{a1});
\draw[draw=black](V_{a0})--(V_{a2});
\end{tikzpicture}
\\
$|V(Q)|=6$, $z_2\in N(D)$&$|V(Q)|=6$, $z_0\in N(D)$
\\
\hline
\end{tabular}
\end{table}

\begin{case6}$H'$ is connected and there is some component $D$ of $G-N[v]$ such that $z_0\in N(D)$, $N(D)\cap \{z_1,z_2\}\neq \emptyset$ and $N(D)\cap \{x,y\}\neq \emptyset$\textup{:}\end{case6}
Without loss of generality, $z_1\in N(D)$. Note that $\{z_1,z_0, x'\}\subseteq N_{G'}(D)$, and let $G'':=G'/E(D)$. The diagrams in \cref{2dist32} demonstrate that $\mathcal{P}\subseteq G''$, contradicting (iv).

\begin{table}[!t]
\centering
 \caption{\label{2dist32}}
 \begin{tabular}{|c|c|}
 \hline
\begin{tikzpicture}[line width=1pt,vertex/.style={circle,inner sep=0pt,minimum size=0.2cm}] 

    \pgfmathsetmacro{\n}{1}; 
    \pgfmathsetmacro{\m}{\n-1};

        \node[draw=black,fill=none, label=above:$z_1$] (V_0) at ($(0,2*\n)$) [vertex] {};
       \node[draw=black,fill=gray, label=above:$z_2$] (V_1) at ($(\n,2*\n)$) [vertex] {};
 \node[draw=black,fill=gray, label=left:$z$] (V_2) at ($(0,\n)$) [vertex] {};
  \node[draw=black,fill=none, label=below:$z_0$] (V_3) at ($(0,0)$) [vertex] {};
       \node[draw=black,fill=gray, label=below:$v_1$] (V_4) at ($(\n,0)$) [vertex] {};
  \node[draw=black,fill=gray, label=below:$v_2$] (V_5) at ($(2*\n,0)$) [vertex] {};
       \node[draw=black,fill=gray, label=above:$w$] (V_6) at ($(2*\n,\n)$) [vertex] {};

\draw[draw=black](V_1)--(V_2);
\draw[draw=black](V_0)--(V_1);
\draw[draw=black](V_0)--(V_2);
\draw[draw=black](V_2)--(V_3);
\draw[draw=black](V_3)--(V_4);
\draw[draw=black](V_5)--(V_6);
\draw[draw=black](V_4)--(V_5);

\end{tikzpicture}
\begin{tikzpicture}[line width=1pt,vertex/.style={circle,inner sep=0pt,minimum size=0.2cm}] 

    \pgfmathsetmacro{\n}{5}; 
    \pgfmathsetmacro{\m}{\n-1};

  \node[draw=black,fill=gray, label=above:$D$] (V_0) at ($(90-0*360/\n:1.5)$) [vertex] {};
 \node[draw=black,fill=none, label=above:$x'$] (V_1) at ($(90-1*360/\n:1.5)$) [vertex] {};
 \node[draw=black,fill=gray, label=right:$w$] (V_2) at ($(90-2*360/\n:1.5)$) [vertex] {}; 
\node[draw=black,fill=gray, label=left:$v$] (V_3) at ($(90-3*360/\n:1.5)$) [vertex] {};
 \node[draw=black,fill=none, label=above:$z_1$] (V_4) at ($(90-4*360/\n:1.5)$) [vertex] {};
  \node[draw=black,fill=none, label={[label distance=-3mm]87:$z_0$}] (V_{a0}) at ($(90-0*360/\n:0.75)$) [vertex] {};
 \node[draw=black,fill=gray, label={[label distance=-1mm]above:$z$}] (V_{a1}) at ($(90-1*360/\n:0.75)$) [vertex] {};
 \node[draw=black,fill=gray, label={[label distance=-1.5mm]85:$z_2$}] (V_{a2}) at ($(90-2*360/\n:0.75)$) [vertex] {}; 
\node[draw=black,fill=gray, label={[label distance=-1.5mm]95:$v_2$}] (V_{a3}) at ($(90-3*360/\n:0.75)$) [vertex] {};
 \node[draw=black,fill=gray, label={[label distance=-1mm]above:$v_1$}] (V_{a4}) at ($(90-4*360/\n:0.75)$) [vertex] {};
    \foreach \x in {0,...,4} {
  
\draw[draw=black] (V_\x) -- (V_{a\x});
    }

\draw[draw=black](V_1)--(V_2);
\draw[draw=black](V_2)--(V_3);
\draw[draw=black](V_3)--(V_4);
\draw[draw=black](V_4)--(V_0);
\draw[draw=black](V_0)--(V_1);
\draw[draw=black](V_{a1})--(V_{a3});
\draw[draw=black](V_{a2})--(V_{a4});
\draw[draw=black](V_{a3})--(V_{a0});
\draw[draw=black](V_{a4})--(V_{a1});
\draw[draw=black](V_{a0})--(V_{a2});
\end{tikzpicture}
&
\begin{tikzpicture}[line width=1pt,vertex/.style={circle,inner sep=0pt,minimum size=0.2cm}] 

    \pgfmathsetmacro{\n}{1}; 
    \pgfmathsetmacro{\m}{\n-1};

        \node[draw=black,fill=gray, label=above:$v_1$] (V_0) at ($(\n,2*\n)$) [vertex] {};
       \node[draw=black,fill=none, label=above:$z_1$] (V_1) at ($(0,2*\n)$) [vertex] {};
 \node[draw=black,fill=gray, label=below:$z_2$] (V_2) at ($(\n,\n)$) [vertex] {};
  \node[draw=black,fill=gray, label=left:$z$] (V_3) at ($(0,\n)$) [vertex] {};
       \node[draw=black,fill=none, label=below:$z_0$] (V_4) at ($(0,0)$) [vertex] {};
  \node[draw=black,fill=gray, label=below:$v_2$] (V_5) at ($(\n,0)$) [vertex] {};
       \node[draw=black,fill=gray, label=below:$w$] (V_6) at ($(2*\n,0)$) [vertex] {};

\draw[draw=black](V_0)--(V_1);
\draw[draw=black](V_0)--(V_2);
\draw[draw=black](V_1)--(V_3);
\draw[draw=black](V_2)--(V_3);
\draw[draw=black](V_3)--(V_4);
\draw[draw=black](V_5)--(V_6);
\draw[draw=black](V_4)--(V_5);

\end{tikzpicture}
\begin{tikzpicture}[line width=1pt,vertex/.style={circle,inner sep=0pt,minimum size=0.2cm}] 

    \pgfmathsetmacro{\n}{5}; 
    \pgfmathsetmacro{\m}{\n-1};

  \node[draw=black,fill=gray, label=above:$D$] (V_0) at ($(90-0*360/\n:1.5)$) [vertex] {};
 \node[draw=black,fill=none, label=above:$x'$] (V_1) at ($(90-1*360/\n:1.5)$) [vertex] {};
 \node[draw=black,fill=gray, label=right:$v$] (V_2) at ($(90-2*360/\n:1.5)$) [vertex] {}; 
\node[draw=black,fill=gray, label=left:$w$] (V_3) at ($(90-3*360/\n:1.5)$) [vertex] {};
 \node[draw=black,fill=none, label=above:$z_1$] (V_4) at ($(90-4*360/\n:1.5)$) [vertex] {};
  \node[draw=black,fill=none, label={[label distance=-3mm]87:$z_0$}] (V_{a0}) at ($(90-0*360/\n:0.75)$) [vertex] {};
 \node[draw=black,fill=gray, label={[label distance=-1mm]above:$z$}] (V_{a1}) at ($(90-1*360/\n:0.75)$) [vertex] {};
 \node[draw=black,fill=gray, label={[label distance=-1.5mm]85:$z_2$}] (V_{a2}) at ($(90-2*360/\n:0.75)$) [vertex] {}; 
\node[draw=black,fill=gray, label={[label distance=-1.5mm]95:$v_1$}] (V_{a3}) at ($(90-3*360/\n:0.75)$) [vertex] {};
 \node[draw=black,fill=gray, label={[label distance=-1mm]above:$v_2$}] (V_{a4}) at ($(90-4*360/\n:0.75)$) [vertex] {};
    \foreach \x in {0,...,4} {
  
\draw[draw=black] (V_\x) -- (V_{a\x});
    }

\draw[draw=black](V_1)--(V_2);
\draw[draw=black](V_2)--(V_3);
\draw[draw=black](V_3)--(V_4);
\draw[draw=black](V_4)--(V_0);
\draw[draw=black](V_0)--(V_1);
\draw[draw=black](V_{a1})--(V_{a3});
\draw[draw=black](V_{a2})--(V_{a4});
\draw[draw=black](V_{a3})--(V_{a0});
\draw[draw=black](V_{a4})--(V_{a1});
\draw[draw=black](V_{a0})--(V_{a2});
\end{tikzpicture}
\\
$|V(Q)|=3$&$|V(Q)|=4$
\\
\hline
\begin{tikzpicture}[line width=1pt,vertex/.style={circle,inner sep=0pt,minimum size=0.2cm}] 

    \pgfmathsetmacro{\n}{1}; 
    \pgfmathsetmacro{\m}{\n-1};

        \node[draw=black,fill=gray, label=above:$v_1$] (V_0) at ($(\n,2*\n)$) [vertex] {};
       \node[draw=black,fill=gray, label=above:$v_2$] (V_1) at ($(2*\n,2*\n)$) [vertex] {};
 \node[draw=black,fill=none, label=above:$z_1$] (V_2) at ($(0,2*\n)$) [vertex] {};
  \node[draw=black,fill=gray, label=below:$z_2$] (V_3) at ($(\n,\n)$) [vertex] {};
       \node[draw=black,fill=gray, label=left:$z$] (V_4) at ($(0,\n)$) [vertex] {};
  \node[draw=black,fill=none, label=below:$z_0$] (V_5) at ($(0,0)$) [vertex] {};
       \node[draw=black,fill=gray, label=below:$w$] (V_6) at ($(\n,0)$) [vertex] {};

\draw[draw=black](V_0)--(V_1);
\draw[draw=black](V_0)--(V_2);
\draw[draw=black](V_1)--(V_3);
\draw[draw=black](V_2)--(V_4);
\draw[draw=black](V_3)--(V_4);
\draw[draw=black](V_5)--(V_6);
\draw[draw=black](V_4)--(V_5);

\end{tikzpicture}
\begin{tikzpicture}[line width=1pt,vertex/.style={circle,inner sep=0pt,minimum size=0.2cm}] 

    \pgfmathsetmacro{\n}{5}; 
    \pgfmathsetmacro{\m}{\n-1};

  \node[draw=black,fill=gray, label=above:$D$] (V_0) at ($(90-0*360/\n:1.5)$) [vertex] {};
 \node[draw=black,fill=none, label=above:$x'$] (V_1) at ($(90-1*360/\n:1.5)$) [vertex] {};
 \node[draw=black,fill=gray, label=right:$v_1$] (V_2) at ($(90-2*360/\n:1.5)$) [vertex] {}; 
\node[draw=black,fill=gray, label=left:$w$] (V_3) at ($(90-3*360/\n:1.5)$) [vertex] {};
 \node[draw=black,fill=none, label=above:$z_1$] (V_4) at ($(90-4*360/\n:1.5)$) [vertex] {};
  \node[draw=black,fill=none, label={[label distance=-3mm]87:$z_0$}] (V_{a0}) at ($(90-0*360/\n:0.75)$) [vertex] {};
 \node[draw=black,fill=gray, label={[label distance=-1mm]above:$z$}] (V_{a1}) at ($(90-1*360/\n:0.75)$) [vertex] {};
 \node[draw=black,fill=gray, label={[label distance=-1.5mm]85:$z_2$}] (V_{a2}) at ($(90-2*360/\n:0.75)$) [vertex] {}; 
\node[draw=black,fill=gray, label={[label distance=-1.5mm]95:$v_2$}] (V_{a3}) at ($(90-3*360/\n:0.75)$) [vertex] {};
 \node[draw=black,fill=gray, label={[label distance=-1mm]above:$v$}] (V_{a4}) at ($(90-4*360/\n:0.75)$) [vertex] {};
    \foreach \x in {0,...,4} {
  
\draw[draw=black] (V_\x) -- (V_{a\x});
    }

\draw[draw=black](V_1)--(V_2);
\draw[draw=black](V_2)--(V_3);
\draw[draw=black](V_3)--(V_4);
\draw[draw=black](V_4)--(V_0);
\draw[draw=black](V_0)--(V_1);
\draw[draw=black](V_{a1})--(V_{a3});
\draw[draw=black](V_{a2})--(V_{a4});
\draw[draw=black](V_{a3})--(V_{a0});
\draw[draw=black](V_{a4})--(V_{a1});
\draw[draw=black](V_{a0})--(V_{a2});
\end{tikzpicture}
&
\begin{tikzpicture}[line width=1pt,vertex/.style={circle,inner sep=0pt,minimum size=0.2cm}] 

    \pgfmathsetmacro{\n}{1}; 
    \pgfmathsetmacro{\m}{\n-1};

        \node[draw=black,fill=gray, label=above:$v_2$] (V_0) at ($(2*\n,2*\n)$) [vertex] {};
       \node[draw=black,fill=gray, label=above:$v_1$] (V_1) at ($(\n,2*\n)$) [vertex] {};
 \node[draw=black,fill=gray, label=below:$v_3$] (V_2) at ($(2*\n,\n)$) [vertex] {};
  \node[draw=black,fill=none, label=above:$z_1$] (V_3) at ($(0,2*\n)$) [vertex] {};
       \node[draw=black,fill=gray, label=below:$z_2$] (V_4) at ($(\n,\n)$) [vertex] {};
  \node[draw=black,fill=gray, label=left:$z$] (V_5) at ($(0,\n)$) [vertex] {};
       \node[draw=black,fill=none, label=below:$z_0$] (V_6) at ($(0,0)$) [vertex] {};

\draw[draw=black](V_0)--(V_1);
\draw[draw=black](V_0)--(V_2);
\draw[draw=black](V_1)--(V_3);
\draw[draw=black](V_2)--(V_4);
\draw[draw=black](V_3)--(V_5);
\draw[draw=black](V_5)--(V_6);
\draw[draw=black](V_4)--(V_5);

\end{tikzpicture}
\begin{tikzpicture}[line width=1pt,vertex/.style={circle,inner sep=0pt,minimum size=0.2cm}] 

    \pgfmathsetmacro{\n}{5}; 
    \pgfmathsetmacro{\m}{\n-1};

  \node[draw=black,fill=gray, label=above:$D$] (V_0) at ($(90-0*360/\n:1.5)$) [vertex] {};
 \node[draw=black,fill=none, label=above:$x'$] (V_1) at ($(90-1*360/\n:1.5)$) [vertex] {};
 \node[draw=black,fill=gray, label=right:$v$] (V_2) at ($(90-2*360/\n:1.5)$) [vertex] {}; 
\node[draw=black,fill=gray, label=left:$v_3$] (V_3) at ($(90-3*360/\n:1.5)$) [vertex] {};
 \node[draw=black,fill=none, label=above:$z_1$] (V_4) at ($(90-4*360/\n:1.5)$) [vertex] {};
  \node[draw=black,fill=none, label={[label distance=-3mm]87:$z_0$}] (V_{a0}) at ($(90-0*360/\n:0.75)$) [vertex] {};
 \node[draw=black,fill=gray, label={[label distance=-1mm]above:$z$}] (V_{a1}) at ($(90-1*360/\n:0.75)$) [vertex] {};
 \node[draw=black,fill=gray, label={[label distance=-1.5mm]85:$z_2$}] (V_{a2}) at ($(90-2*360/\n:0.75)$) [vertex] {}; 
\node[draw=black,fill=gray, label={[label distance=-1.5mm]95:$v_1$}] (V_{a3}) at ($(90-3*360/\n:0.75)$) [vertex] {};
 \node[draw=black,fill=gray, label={[label distance=-1mm]above:$v_2$}] (V_{a4}) at ($(90-4*360/\n:0.75)$) [vertex] {};
    \foreach \x in {0,...,4} {
  
\draw[draw=black] (V_\x) -- (V_{a\x});
    }

\draw[draw=black](V_1)--(V_2);
\draw[draw=black](V_2)--(V_3);
\draw[draw=black](V_3)--(V_4);
\draw[draw=black](V_4)--(V_0);
\draw[draw=black](V_0)--(V_1);
\draw[draw=black](V_{a1})--(V_{a3});
\draw[draw=black](V_{a2})--(V_{a4});
\draw[draw=black](V_{a3})--(V_{a0});
\draw[draw=black](V_{a4})--(V_{a1});
\draw[draw=black](V_{a0})--(V_{a2});
\end{tikzpicture}
\\
$|V(Q)|=5$ & $|V(Q)|=6$
\\
\hline
\end{tabular}
\end{table}

\begin{case6}$H'$ is connected and there is no component $D$ of $G-N[v]$ such that either $z\in N(D)$ and $|N(D)\cap N_{H'}(z)|\geq 2$ or $z_0\in N(D)$, $N(D)\cap \{z_1,z_2\}\neq \emptyset$ and $N(D)\cap \{x,y\}\neq \emptyset$\textup{:}\end{case6}

Recall that $|N_{G'}(z)\cap N_{G'}(v)|\leq 4$. Hence $z$ has at least three non-neighbours in $G'[N(v)]$. Since $\overline{G'[N(v)\setminus \{x,y\}]}\subseteq H'$ and $x'$ is dominant in $G'[N_{G'}(v)]$, $z$ is non-adjacent in $G'$ to each vertex in $N_{H'}(z)$. Hence, for every vertex $z'\in N_{H'}[z]$ there is a component $C_{z'}$ of $G-N[v]$ such that $z'\in N(C_{z'})$, since $v\in \mathcal{L}$.

By \cref{separations}, each component $C_{z'}$ of $G-N[v]$ satisfying $|N(C_{z'})|\leq 6$ is $v$-suitable.

Recall that $C'$ is an arbitrary component of $G-N[v]$. We now show that, for some $z'\in N_{H_v}[z]$, $C_{z'}$ is $v$-suitable and $N(C_{z'})\setminus N(C')\neq \emptyset$, as required.

Suppose first that there is no component $D$ of $G-N[v]$ such that $z\in N(D)$ and $|N(D)\cap N_{H'}(z)|\geq 1$. Then $|N(C_z)|\leq 6$. Furthermore, $z\notin N(C_{z_0})$ and either $N(C_{z_0})\cap \{z_1,z_2\}=\emptyset$ or $N(C_{z_0})\cap \{x,y\}=\emptyset$ since Case 3 does not apply, so $|N(C_{z_0})|\leq 6$. Hence, $C_z$ and $C_{z_0}$ are both $v$-suitable. By assumption, $N(C)$ does not contain both $z$ and $z_0$, so $z'\notin N(C)$ for some vertex $z'\in \{z,z_0\}$. Hence, $N(C_{z'})\setminus N(C')\neq \emptyset$, and the claim holds.

 Now assume that there is some component $D$ of $G-N[v]$ such that $z\in N(D)$ and $|N(D)\cap N_{H'}(z)|\geq 1$. Since Case 2 does not apply, $|N(D)\cap N_{H'}(z)|=1$. Let $\{z',z''\}:=N_{H'}(z)\setminus N(D)$. If $|N(C_{z'})|\leq 6$ and $|N(C_{z''})|\leq 6$ (in which case $C_{z'}$ and $C_{z''}$ are both $v$-suitable), and $\{z',z''\}\nsubseteq N(C')$, then the claim holds. So we may assume that either $D':=C'$ satisfies $\{z',z''\}\subseteq N(D')$ or some $D'\in \{C_{z'},C_{z''}\}$ satisfies $|N(D')|\geq 7$. Now $D'$ is distinct from $D$ since $N(D')\cap \{z',z''\}\neq \emptyset$, and $|N_{G'}(D')|\geq 3$ since $|N(D')|\geq 4$ by \cref{connect}. Let $G''$ be obtained from $G'$ by contracting $D$ onto $z$. Then $v\in V_8(G'')$, $|N_{G''}(v)\cap N_{G''}(v')|\geq 5$ for every vertex $v'\in N_{G''}(v)$, and $|N_{G''}(D')|=|N_{G'}(D')|\geq 3$. Furthermore, there is at most one cycle in $\overline{G''[N(v)]}$, namely $Q$, so $\overline{K_3}$ and $\overline{C_4}$ are not both induced subgraphs of $G''[N(v)]$. Hence by \cref{deg8lem}, $\mathcal{P}\subseteq G''$, contradicting (iv). 
\end{proof}

We finally reach the main result of this section.
\begin{lem}\label{9suit} If $v\in V_9(G)\cap \mathcal{L}$ and $C$ is a component of $G-N[v]$, then there is some $v$-suitable subgraph $C'$ such that $N(C')\setminus N(C)\neq \emptyset$.\end{lem}
\begin{proof}
Suppose first that each component $C'$ of $G-N[v]$ has $|N(C')|= 4$. Then every component of $G-N[v]$ is $v$-suitable by \cref{separations}. Suppose for contradiction that there is no $v$-suitable subgraph $C'$ such that $N(C')\setminus N(C)\neq \emptyset$.  Then $N(C')\subseteq N(C)$ for every component $C'$ of $G-N[v]$, so there are at least five vertices in $N(v)$ with no neighbour outside of $N[v]$. Since $v\in \mathcal{L}$, each of these vertices is dominant in $G[N[v]]$. Let $G'$ be obtained from $G$ by contracting $C$ onto some vertex $x$ of $N(C)$ and then deleting all other components of $G-N[v]$. There are at most three non-dominant vertices in $G'$, so $|E(G')|\geq {10\choose 2}-3=42=5|V(G')|-8$, contradicting (vi).

Suppose instead that there is some component $C'$ of $G-N[v]$ with $|N(C')|\geq 5$. By \cref{dist3,K33}, we may assume that there are two vertices $x$ and $y$ in $V_3(H_v)$ such that $\dist_{H_v}(x,y)\geq 3$. The result then follows directly from \cref{2dist3}.
\end{proof}
\cref{9suit} immediately implies the following corollary, which we use in \cref{endsec}.
 \begin{cor}\label{9ex}For every vertex $v\in V_9(G)$ there is at least one $v$-suitable subgraph.\end{cor}

\section{Final Step}\label{endsec}

We now complete the proof sketched in \cref{sketch}.

 
\begin{proof}[Proof of \cref{extremalgraphs}]
Let $G$ be the minimum counterexample defined at the start of \cref{basic}. 
By \cref{triangles,7N,deg8}, $\mathcal{L}\subseteq V_6(G)\cup V_9(G)$, so for every vertex $v\in \mathcal{L}$ there is some $v$-suitable subgraph of $G$ by \cref{6ex,9ex}. Choose $v\in \mathcal{L}$ and $H$ a $v$-suitable subgraph of $G$ so that $|V(H)|$ is minimised. Let $u$ be a vertex of $\mathcal{L}$ in $H$. Since $u\in V(H)$ and $H$ is a component of $G-N[v]$, $u$ is not adjacent to $v$, so $v$ is in some component $C$ of $G-N[u]$. Since $v\in \mathcal{L}$, $C$ is $u$-suitable. By \cref{6suit,9suit}, there is some $u$-suitable subgraph $C'$ of $G$ with $N(C')\setminus N(C)\neq \emptyset$. 

Now $N(C')\subseteq N(u)$, so $v\notin N(C')$. Since $N(C')\setminus N(C)\neq \emptyset$, we have that $C$ and $C'$ are distinct (and thus disjoint), so $v\notin N[C']$ and $C'$ is disjoint from $N[v]$. Hence $G[V(C')\cup (N(C')\setminus N(C))\cup \{u\}]$ is a connected subgraph of $G-N[v]$, and thus a subgraph of $H$. But $u\in V(H)\setminus V(C')$, so $|V(C')|<|V(H)|$, contradicting our choice of $v$ and $H$. This contradiction shows that in fact there are no counterexamples to \cref{extremalgraphs}.
\end{proof}

\subsection*{Acknowledgements}
The authors are grateful to the referees for their helpful comments, and especially for an insightful suggestion that lead to a significant shortening of \cref{9sec}.

\newcommand{\etalchar}[1]{$^{#1}$}

\newpage
\appendix
\section{Appendix}
\label{app}

We now prove the two well known lemmas used in \cref{intro}.

\begin{lem}
\label{newbest}
For every $(t+1)$-connected graph $H$ and every non-negative integer $s<|V(H)|$,
every $(K_s,t)$-cockade is $H$-minor-free.
\end{lem}
\begin{proof}
Let $G$ be a $(K_s,t)$-cockade. We proceed by induction on $|V(G)|+|E(G)|$. The claim is trivial if $G=K_s$, since $s<|V(H)|$.
Assume that there are $(K_s,t)$-cockades $G_1$ and $G_2$ distinct from $G$ such that $G_1\cup G_2=G$ and $G_1\cap G_2\cong K_t$. Note that $G_1$ and $G_2$ are proper subgraphs of $G$, and hence by induction are $H$-minor-free.
Suppose for contradiction that $G$ contains an
$H$-minor. Then there is a set of pairwise disjoint connected
subgraphs of $G$ such that if every edge inside one of these subgraphs
is contracted and every vertex not in one of these subgraphs is
deleted, then the graph obtained is a supergraph $H'$ of
$H$ such that $|V(H')|=|V(H)|$. Each of these subgraphs will contract down to a
separate vertex, so we call these subgraphs {\it prevertices}. There
are exactly $t$ vertices in $G_1\cap G_2$, so the set $S$ of prevertices that intersect $G_1\cap
G_2$ has size at most $t$. Since $H$ is $(t+1)$-connected, each prevertex not in $S$ is
in the same connected component of $G- S$. Without loss of generality,
each prevertex not in $S$ is a subgraph of $G_1$. Now, there is no
path of $G$ between two non-adjacent vertices of $G_1$ that is
internally disjoint from $G_1$. Hence, by deleting every vertex of
$G_2\setminus G_1$ and then contracting the remaining edges of the
prevertices and deleting the remaining vertices that are not in any
prevertex, we obtain $H'$, contradicting the assumption the $G_1$ contains no
$H$-minor.
\end{proof}

\begin{proof}[Proof of \cref{gencolour}]
Let $G$ be an $n$-vertex $H$-minor-free graph. We proceed by induction on $n$. 
 The base case with $n\leq 2c-1$ is
trivial. 
For $n\geq 2c$, $|E(G)|<c|V(G)|$, implying $G$ has average degree less than $2c$. Thus
$G$ has a vertex $v$ of degree at most $2c-1$. By induction, $G-v$ is
$2c$-colourable. Some colour is not used on the neighbours of $v$,
which can be assigned to $v$. Hence $G$ is $2c$-colourable. It remains
to prove that $G$ is $(2c-1)$-colourable under the assumption that
$|V(H)|\leq 2c$. First suppose that $\deg(v)\leq 2c-2$. By induction,
$G-v$ is $(2c-1)$-colourable. Some colour is not used on the
neighbours of $v$, which can be assigned to $v$. Hence $G$ is
$(2c-1)$-colourable. Now assume that $\deg(v)=2c-1$. There is some
pair of non-adjacent vertices $x$ and $y$ in $N(v)$, as otherwise $G$
contains $K_{2c}$ and hence $H$ (since $|V(H)|\leq 2c$). Let $G'$ be
the graph obtained from $G$ by contracting the edges $vx$ and $vy$
into a new vertex $z$. By induction, $G'$ is $(2c-1)$-colourable.
Colour each vertex of $G-\{v,x,y\}$ by the colour assigned to the corresponding vertex in
$G'$. Colour $x$ and $y$ by the colour assigned to $z$. Since every
vertex adjacent to $x$ or $y$ in $G-v$ is adjacent to $z$ in $G'$,
this defines a ($2c-1$)-colouring of $G-v$. Now $v$ has $2c-1$
neighbours, two of which have the same colour. Thus there is an unused
colour on the neighbours of $v$, which can be assigned to $v$.
Therefore $G$ is ($2c-1$)-colourable.
\end{proof}

\end{document}